\RequirePackage{fix-cm}
\documentclass[10pt, oneside, english]{amsart}
\usepackage{mathpazo}
\usepackage{eulervm}
\usepackage[T1]{fontenc}
\usepackage[latin9]{inputenc}
\usepackage{color}
\usepackage{babel}
\usepackage{units}
\usepackage{amstext}
\usepackage{amsthm}
\usepackage{amssymb}
\usepackage{setspace}
\usepackage[unicode=true,pdfusetitle,
 bookmarks=true,bookmarksnumbered=false,bookmarksopen=false,
 breaklinks=false,pdfborder={0 0 1},backref=false,colorlinks=true]
{hyperref}
\usepackage{enumitem}

\def\namedlabel#1#2{\begingroup
    #2%
    \def\@currentlabel{#2}%
    \phantomsection\label{#1}\endgroup
}

\makeatletter
\usepackage{setspace}
\usepackage[margin=.8in]{geometry}
\usepackage{geometry}
\usepackage{todonotes}

\numberwithin{equation}{section}
\numberwithin{figure}{section}
\theoremstyle{plain}
\newtheorem{thm}{\protect\theoremname}[section]
\newtheorem{introtheorem}{Theorem}
\theoremstyle{definition}
\newtheorem{defn}[thm]{\protect\definitionname}
\newtheorem*{defn*}{\protect\definitionname}

\theoremstyle{remark}
\newtheorem{rem}[thm]{\protect\remarkname}

\theoremstyle{plain}
\newtheorem{fact}[thm]{\protect\factname}
\theoremstyle{plain}
\newtheorem{prop}[thm]{\protect\propositionname}
\theoremstyle{plain}
\newtheorem{lem}[thm]{\protect\lemmaname}
\theoremstyle{plain}
\newtheorem{cor}[thm]{\protect\corollaryname}
\theoremstyle{definition}
\newtheorem{example}[thm]{\protect\examplename}

\usepackage[psamsfonts]{eucal}
\DeclareMathOperator{\VFA}{VFA}
\DeclareMathOperator{\ACFA}{ACFA}
\DeclareMathOperator{\ACVF}{ACVF}
\DeclareMathOperator{\ACF}{ACF}
\DeclareMathOperator{\OGA}{OGA}

\DeclareMathOperator{\VF}{VF}
\DeclareMathOperator{\sep}{sep}
\DeclareMathOperator{\tensor}{\otimes}

\DeclareMathOperator{\Aut}{Aut}

\DeclareMathOperator{\tp}{tp}
\DeclareMathOperator{\alg}{alg}
\DeclareMathOperator{\tdeg}{tdeg}
\DeclareMathOperator{\VFE}{VFE}

\DeclareMathOperator{\SCVF}{SCVF}
\DeclareMathOperator{\SCFE}{SCFE}

\DeclareMathOperator{\OGE}{OGE}
\DeclareMathOperator{\FF}{FE}

\DeclareMathOperator{\NTPtwo}{NTP_{2}}
\DeclareMathOperator{\NIP}{NIP}
\DeclareMathOperator{\qftp}{qftp}
\DeclareMathOperator{\ACFtwo}{ACF^{2}}

\newcommand{\sphul}{{\sigma\cdot p^{-\infty}}}
\newcommand{\inv}{{\sigma^{-\infty}}}

\newcommand{\Ns}{{\mathbf{N}[\sigma]}}
\newcommand{\Nsp}{{\mathbf{N}[\sigma, p^{-1}]}}

\newcommand{\Nsphul}{{\mathbf{N}[\frac{\sigma}{p^\infty}]}}
\newcommand{\Zsphul}{{\mathbf{Z}[\frac{\sigma}{p^\infty}]}}

\newcommand{\TwVFA}{\widetilde{\VFA}}
\newcommand{\wVFA}{\VFA}
\newcommand{\wVFE}{\VFE}
\newcommand{\TVFA}{\widetilde{\VFA}}
\newcommand{\TVFE}{\widetilde{\VFE}}
\newcommand{\TwVFE}{\widetilde{\VFE}}
\newcommand{\wVFEe}{\VFE_{\leq e}}
\newcommand{\TwVFEe}{\widetilde{\VFE_e}}
\newcommand{\TwOGA}{\widetilde{\omega\OGA}}
\newcommand{\wOGE}{\omega\OGE}
\newcommand{\Qs}{\mathbf{Q}\left(\sigma\right)}

\providecommand{\set}[2]{\{#1 \mid #2\}}
\providecommand{\sequence}[2]{(#1)_{#2}}

\def\Ind{\setbox0=\hbox{$x$}\kern\wd0\hbox to 0pt{\hss$\mid$\hss} \lower.9\ht0\hbox to 0pt{\hss$\smile$\hss}\kern\wd0} 

\def\Notind{\setbox0=\hbox{$x$}\kern\wd0\hbox to 0pt{\mathchardef \nn=12854\hss$\nn$\kern1.4\wd0\hss}\hbox to 0pt{\hss$\mid$\hss}\lower.9\ht0 \hbox to 0pt{\hss$\smile$\hss}\kern\wd0} 

\def\ind{\mathop{\mathpalette\Ind{}}}

\makeatother

\providecommand{\corollaryname}{Corollary}
\providecommand{\definitionname}{Definition}
\providecommand{\examplename}{Example}
\providecommand{\factname}{Fact}
\providecommand{\lemmaname}{Lemma}
\providecommand{\propositionname}{Proposition}
\providecommand{\remarkname}{Remark}
\providecommand{\theoremname}{Theorem}
\makeatletter
\let\@wraptoccontribs\wraptoccontribs
\makeatother

\begin{document}
\title{Contracting Endomorphisms of Valued Fields}
\author{Yuval Dor}
\address{Apple Inc., Israel}
\email{yuval.dor4@gmail.com}

\author{Yatir Halevi}
\address{Department of Mathematics\\ University of Haifa\\ 199 Abba Khoushy Avenue \\ Haifa \\Israel}
\email{ybenarih@campus.haifa.ac.il}

\contrib[with an appendix by]{Itay Kaplan}
\thanks{The second author was partially supported by ISF grants No. 555/21 and 290/19 and by the Fields Institute for Research in Mathematical Sciences. The author of the appendix would like to thank the Israel Science Foundation for their support of this research (grant no. 1254/18)}

\maketitle

\begin{center}
    \textit{In memory of Zo\'e Chatzidakis, 1955-2025.}
\end{center}


\begin{abstract}
We prove that the class of separably algebraically closed valued fields equipped with a distinguished Frobenius endomorphism $x \mapsto x^q$ is decidable, uniformly in $q$. The result is a simultaneous generalization of the work of Chatzidakis and Hrushovski (in the case of the trivial valuation) and the work of the first author and Hrushovski (in the case where the fields are algebraically closed).

The logical setting for the proof is a model completeness result for valued fields equipped with an endomorphism $\sigma$ which is locally infinitely contracting and fails to be onto. Namely we prove the existence of a model complete theory $\widetilde{\mathrm{VFE}}$ amalgamating the theories $\mathrm{SCFE}$ and $\widetilde{\mathrm{VFA}}$ introduced in \cite{SCFEe} and \cite{dor-hrushovskiVFA}, respectively. In characteristic zero, we also prove that $\widetilde{\mathrm{VFE}}$ is NTP$_2$ and classify the stationary types: they are precisely those orthogonal to the fixed field and the value group.
\end{abstract}


\section{Introduction}

\subsection{Motivation and Background}

 By a \emph{transformal valued field} we mean a valued field $K$ equipped with a distinguished endomorphism $\sigma \colon K \to K$ of fields such that $\sigma^{-1}\left(\mathcal{O}\right) = \mathcal{O}$ and $\sigma^{-1}\left(\mathcal{M}\right) = \mathcal{M}$; here $\mathcal{O}$ and $\mathcal{M}$ denote the valuation ring and the maximal ideal, respectively.

Of particular interest is the case where $K$ is of positive characteristic $p$ and $\sigma = p^n$ coincides with a power of the Frobenius endomorphism of $K$; in this case, we say that $K$ is a \emph{Frobenius transformal valued field}.

The following is a logical corollary of our main results:
\begin{introtheorem}[Theorems \ref{thm:wvfe-decidable} and \ref{T:asymp theory}]
\label{thm-frob-sep-val-decidable}
    The elementary theory of the class of separably algebraically closed Frobenius transformal valued fields is decidable.
\end{introtheorem}

Theorem \ref{thm-frob-sep-val-decidable} follows from an abstract model completeness result; a key geometric input here is the the twisted Lang-Weil estimates of Hrushovski \cite{hrushovski2004elementary} and Shuddhodan-Varshavsky \cite{ShVa}. More precisely, we will define a theory $\TwVFE$ of transformal valued fields which is closely related to the elementary theory of separably algebraically closed Frobenius transformal valued fields; Theorem \ref{thm-frob-sep-val-decidable} is then deduced from a quantifier elimination statement for $\TwVFE$. Before we turn to the precise statement it will be useful to give an overview of the model theory of the Frobenius endomorphism.
\vspace{1mm}

\subsubsection{$\ACFA$ and The Limit Theory of the Frobenius Action.} The seminal work of Hrushovski \cite{hrushovski2004elementary} initiated the study of transformal geometry. Whereas algebraic geometry takes place over a field, transformal geometry takes place over a \textit{difference field}; that is, a field $K$ equipped with a distinguished endomorphism $\sigma \colon K \to K$.

The connection between transformal geometry and algebraic geometry is mediated by Hrushovski's Frobenius specialization functors. Model theoretically, the class of difference fields admits a model companion $\ACFA$, which agrees precisely with the limit theory of the Frobenius action on an algebraically closed field. This justifies the useful heuristic that the generic endomorphism $\sigma$ can be regarded as an infinitely large prime power.

The theory $\ACFA$ was studied in detail in the work of Chatzidakis and Hrushovski \cite{ChHr-Dif}.

\subsubsection{$\TVFA$} In Hrushovski's work \cite{hrushovski2004elementary} a key role is played by a moving lemma and a theory of specialization of difference varieties. This leads naturally to the use of valuations, and hence of transformal valued fields of the following form.

Let $K$ be a transformal valued field. We say that $K$ is $\omega$-\emph{increasing} if the induced action of $\sigma$ on $\Gamma$ is $\omega$-increasing in the sense that $n\alpha < \sigma\alpha$ for all $0 < \alpha \in \Gamma$ and all $n \in \mathbf{N}$. The Frobenius transformal valued fields are not $\omega$-increasing, but a nonprincipal ultraproduct of them is.

The class of $\omega$-increasing transformal valued fields admits a model companion $\TVFA$. In characteristic zero this is due to Durhan \cite{azgin2010valued} and in all characteristics due to Hrushovski and the first author \cite{dor-hrushovskiVFA}. Furthermore, the theory $\TVFA$ is precisely the limit theory of the Frobenius action on an algebraically closed and nontrivially valued field.

\subsubsection{Transformally Separable Extensions} Let $K$ be a difference field. We say that $K$ is \textit{inversive} or \textit{transformally perfect} if $\sigma$ is onto on $K$. We say that an embedding $\nicefrac{L}{K}$ of difference fields is \textit{transformally separable} if $K$ and $L^{\sigma}$ are linearly disjoint over $K^{\sigma}$; when $\sigma = p^n$ is the Frobenius endomorphism, this corresponds to the usual notion of a separable field extension.

\subsubsection{} One of the major goals of the present work is to initiate a systematic study of transformally separable extensions (with no valuation); this is carried out in Section \ref{s:tran separable}. Here we build and expand on the ideas of Chatzidakis-Hrushovski (\cite{SCFEe}).

In order to obtain a well behaved theory we restrict attention to a universal theory of difference fields introduced in \cite{SCFEe} which we call here $\FF$; see Definition \ref{D:FE} for a precise statement. The class of models of $\FF$ enjoys excellent functorial properties, and includes algebraically closed fields, Frobenius ultraproducts, as well as all those difference fields which are perfect and inversive.

For transformally separable extensions of $\FF$ we obtain a structure theory which is in many ways analogous to the classical structure theory of separable field extensions; see Theorem \ref{thm:main-thm-on-tsep}. For example, we exhibit a functorial relative transformal separable algebraic closure, characterized set theoretically as the set of simple roots of "twisted" difference polynomials; the theory of separable generation can likewise be generalized (see Theorem \ref{thm:main-thm-on-tsep}).

Using the results of Section \ref{s:tran separable} functorial constructions available in the inversive realm descend to the category of models of $\FF$; we use this repeatedly in order to apply the machinery of \cite{dor-hrushovskiVFA}. We expect the results to be useful in other settings; see for example the work of Simone Ramello \cite{simone} which vastly generalizes our quantifier elimination results in characteristic zero.

\subsubsection{Notation} For any theory of $T$ difference fields, possibly enriched, one can define an expansion $T_{\lambda}$ of $T$ by definitions in which the models are the same, but embeddings are required to be transformally separable (see \ref{s:trans lambda functions}).

\subsubsection{$\SCFE$} In \cite{SCFEe} Chatzidakis and Hrushovski introduce a theory $\SCFE$ of difference fields which are not inversive. The theory $\SCFE$ is simple and the expansion $\SCFE_{\lambda}$ is model complete. The nonprincipal ultraproduct of separably algebraically closed and Frobenius difference fields is a model of $\SCFE$.
\vspace{1mm}

\subsection{Main Results}

We let $\VFE$ be the theory of $\omega$-increasing transformal valued fields with underlying difference field a model of $\FF$.

 We prove the following:

\begin{introtheorem}\label{introthm:vfeec}(See Theorem \ref{thm:model-companion})
\begin{enumerate}
    \item The class of models of $\VFE$ whose inversive hull is a model of $\TVFA$ is elementary; write $\TwVFE$ for the common first order theory.
    \item The theory $\TwVFE_{\lambda}$ is model complete. The models of $\TwVFE$ which fail to be inversive are uniquely characterized up to elementary equivalence by the characteristic, the degree of imperfection, and the isomorphism type of the action of $\sigma$ on the algebraic closure of the prime field (and all possibilities occur).
    \end{enumerate}
\end{introtheorem}

Using Theorem \ref{introthm:vfeec} and the main results of \cite{hrushovski2004elementary} it is not difficult to verify the following:

\begin{introtheorem}\label{frobisvfe} (Theorem \ref{T:asymp theory})
Let $\widetilde{T}$ be the asymptotic theory of separably algebraically closed and nontrivially valued Frobenius transformal valued fields; that is, the set of sentences true in all such transformal valued fields, outside a finite set of exceptional prime powers. Then $\widetilde{T}$ is decidable and $\widetilde{T}_{\lambda}$ is model complete; indeed, every model of $\widetilde{T}$ is a model of $\TwVFE$.
\end{introtheorem}

We prove various other results:
\begin{itemize}
     \item The theory $\TwVFE$ admits an explicit axiomatization which parallels that of $\TVFA$ and is consequently decidable (Theorem \ref{thm:wvfe-decidable}).
    \item In $\TwVFE$ the residue field field and the value group are stably embedded and model theoretically fully orthogonal; the induced structure is that of a pure model of $\ACFA$ and $\TwOGA$, respectively (Theorem \ref{thm:stably-embedded}).
    \item The model theoretic algebraic closure in $\TwVFE$ is determined in Theorem \ref{T:model theoretic acl}; it is as small as can possibly be expected given the behavior in $\TVFA$ and $\SCFE$.
    \item The theory $\TwVFE$ admits elimination of quantifiers to the level of the field theoretic algebraic closure in a language where the transformal $\lambda$-functions are considered basic; see Corollary \ref{cor:qe} and Corollary \ref{cor:qe-explained} for a precise statement. This level of quantifier reduction is likewise optimal given its behavior in $\TVFA$ and $\SCFE$.
    \item In characteristic zero we study stability theoretic properties of $\TwVFE$. We show that $\TwVFE_{\mathbf{Q}}$ enjoys the $\NTPtwo$; over highly saturated models, forking is witnessed by formulas with the $\NIP$ (\ref{cor:twvfe-ntp2}). A basic intuition on $\NTPtwo$ theories is that they are $\NIP$ theories with "random noise" (see e.g \cite{johnson2019finiteburdenmultivaluedalgebraically}); our results support this.
\end{itemize}
    In characteristic zero, we are able to pin down the stationary types:

\begin{introtheorem}[Corollary \ref{C:stationary type}]
In the theory $\TwVFE_{\mathbf{Q}}$, a global invariant type is stationary if and only if it is orthogonal to the fixed field and to the value group. 
\end{introtheorem}

In $\SCFE_{\mathbf{Q}}$ this follows from the trichotomy for $\ACFA_{\mathbf{Q}}$ and the stationarity theorem of \cite{SCFEe}; we deduce this here in the presence of a valuation using domination results.

\subsubsection{} 

The proof that $\TwVFE_{\mathbf{Q}}$ (and $\TVFA_{\mathbf{Q}}$) enjoys the NTP$_2$ uses a domination result by the residue field and the value group over spherically complete models; in order to lift it to the general case we use an observation due to Chernikov \cite{chernikov2014theories}. Itay Kaplan relaxed Chernikov's argument to show that if forking of a global type is always witnessed by an NIP formula, then the theory is NTP$_2$. We thank him for his permission to include it here as an appendix.

\subsection{Acknowledgments}

We would like to thank Zo\'e Chatzidakis, Ehud Hrushovski and Simone Ramello for useful discussions. Lastly, we would like to thank the anonymous referee for his$\backslash$her thorough reading, numerous helpful comments and suggestions and for catching several inaccuracies in some of the proofs.

\section{Preliminaries}
\subsection{Conventions and Notation}

\subsubsection{The Characteristic Exponent of a Field}

Recall that the \emph{characteristic exponent} of a field is equal
to $p$ if the field is of positive characteristic $p>0$ and is equal
to $1$ if the field if of characteristic zero. By convention, a natural
number $p$ which is either prime or equal to $1$ is fixed through
this text, and all valued fields under consideration are assumed of
equal characteristic, of characteristic exponent equal to $p$.

\subsubsection{Field Theory}

For basic reference on field theory we refer the reader to \cite{karp,FrJaBook}.
Fix a field $k$ of characteristic exponent $p$. We write $\phi x=x^{p}$
for the Frobenius endomorphism of $k$, thus it is the identity map
in characteristic exponent $p=1$. We write $k^{p^{-\infty}},k^{\sep}$
and $k^{\alg}$ for the perfect hull and for the choice of a separable
and algebraic closure of $k$, respectively. The \emph{Ershov invariant}
or the \emph{imperfection degree} of $k$ is the natural number
$e$ with the property that $\left[k\colon k^{p}\right]=p^{e}$ if
the former quantity is finite, and equal to $\infty$ otherwise.

Let $\nicefrac{K}{k}$ be a field extension. We say that $K$ is \emph{separable}
over $k$ if the fields $K$ and $k^{p^{-\infty}}$ are linearly disjoint
over $k$ in $K^{p^{-\infty}}$. We say that $K$ is
\emph{separably generated} over $k$ if there is a transcendence basis $x$
of $K$ over $k$ such that $K$ is separably algebraic over $k\left(x\right)$.
If $K$ is separably generated over $k$ then it is separable over
$K$; the converse holds when $K$ is a finitely generated field extension
of $k$. We say that $K$ is \emph{primary} over $k$ if the field
$k$ is relatively separably algebraically closed in $K$, that is,
every element of $K$ separably algebraic over $k$ already lies in
$k$. It is evident that purely inseparably algebraic extensions are primary. Finally, the field $K$ is \emph{regular} over $k$ if it
is both separable and primary over $k$.

Let $\nicefrac{K}{k}$ be a field extension. We say that $K$ is \emph{relatively
perfect} over $k$ or that \emph{the perfect hull of $K$ splits
over $k$} if the field $K$ is separable over $k$ and one has $K\otimes_{k}k^{p^{-\infty}}=K^{p^{-\infty}}$. Equivalently,\footnote{We are thankful to the referee for this remark.} the field extension $\nicefrac{K}{k}$ is relatively perfect if it is separable and some (every) $p$-basis of $k$ is a $p$-basis of $K$.
Beware that in \cite{karp} relatively perfect extensions are \emph{not}
assumed separable. If $k$ is of finite Ershov invariant then a separable
extension $K$ of $k$ is relatively perfect precisely in the event
that the Ershov invariants coincide.

\subsubsection{\label{subsec:Difference-Rings}Difference Algebra}

We follow the conventions and notation of \cite{hrushovski2004elementary} and \cite{dor-hrushovskiVFA}. See Section 3 of \cite{dor-hrushovskiVFA} for a summary of the basic facts from the model theory of difference fields we will use.

By a \emph{difference ring} we mean a (commutative and unital) ring $R$ equipped with a distinguished (injective) endomorphism $\sigma$ of rings; a \emph{homomorphism} of difference rings is a ring homomorphism compatible with the action of $\sigma$ in the evident manner. By a \emph{difference field} we mean a difference ring whose underlying ring is a field. 

Let $K$ be a difference field. We use exponential notation for the action of $\sigma$, thus for an element (or a tuple) $a \in K$ we write $a^\sigma = \sigma\left(a\right)$. We also write $K^\sigma = \sigma\left(K\right)$ for the image of $K$ under the endomorphism $\sigma$ and regard $K^{\sigma}$ as a subfield of $K$.

Let $K$ be a difference field. We say that $K$ is \emph{inversive} if $\sigma$ is an automorphism of $K$. In the literature this is sometimes taken as a part of the definition, but the whole point of the present work is to dispense with this assumption. In any event, if $K$ is a difference field, then there is an inversive difference field $K^{\sigma^{-\infty}}$ over $K$ which embeds uniquely over $K$ in any other. In this situation, we say that $K^{\sigma^{-\infty}}$ is the \textit{inversive hull} of $K$; it is unique up to a unique isomorphism of difference fields over $K$. The inversive hull of $K$ is given explicitly as the directed union of isomorphic copies of $K$, so any property of difference fields stable in directed unions passes from $K$ to its inversive hull.

\subsection{\label{subsec:notation} Notation}

We write $\mathbf{N}\left[\sigma\right]$  for the commutative semiring of finite formal sums of the form $\lambda_n \sigma^n + \ldots + \lambda_1 \sigma + \lambda_0$ where the $\lambda_i$ are natural numbers. If $K$ is a difference field then we write $K\left[x^{\mathbf{N}\left[\sigma\right]}\right]$ for the ring of difference polynomials over $K$. For $a$ an element (or a tuple) in an ambient difference field extension of $K$, we write $K\left(a^{\Ns}\right)$ for the smallest difference field extension of $K$ containing $a$.
We write $K^{\sigma \cdot p^{-\infty}}$ for the perfect hull of $K^{\sigma}$. 

Let $K$ be a difference field and $a$ an element in an ambient extension. We say that $a$ is \emph{transformally algebraic} over $K$ if it is a zero of a nonzero difference polynomial over $K$; otherwise, the element $a$ is said to be \emph{transformally transcendental} over $K$. The notion of a transformal transcendence basis and transformal independence of a tuple behaves as expected; see \cite{dor-hrushovskiVFA}, Section 3 for more.

\section{Descent, Base Extension, and the Field Theoretic Tensor Product}

In modern algebraic geometry and commutative algebra, the method of
descent along a morphism is ubiquitous. In the present work, in order
to avoid excess generality, we restrict ourselves exclusively to the
case of fields. Grothendieck's relative point of view is nevertheless
valid: it is often possible to study an embedding $\nicefrac{L}{K}$
of difference fields after extending $K$, and many statements true
over $K$ inversive and algebraically closed admit relative versions.
The purpose of this section is to set up the notation and terminology
necessary in order to make a systematic use of descent. The material
here is standard and no claim to originality is made.

\subsubsection{\label{subsec:Linear-Disjointness}Linear Disjointness}

Let $B\hookleftarrow A\hookrightarrow C$ be difference fields, jointly
embedded over $A$ in an ambient extension $\Omega$. We say that
$B$ and $C$ are \emph{linearly disjoint over $A$ }in $\Omega$
if whenever the sequence $a_{1},\ldots,a_{n}\in B$ is linearly independent
over $A$ then it remains linearly independent over $C$. For more
on this notion see \cite[Section 2.5]{FrJaBook}.

\subsubsection{\label{subsec:Abstract-Linear-Disjointness}Abstract Linear Disjointness}

Linear disjointness is strictly speaking meaningless in the absence
of an ambient extension $\Omega$. Nevertheless, we will encounter
the situation where the choice of $\Omega$ is irrelevant. Thus the
following definition will be useful: 
\begin{defn}
Let $B\hookleftarrow A\hookrightarrow C$ be embeddings of difference
fields. 
\begin{enumerate}
\item By a \emph{compositum} of $B$ and $C$ over $A$ we mean a difference
field $F$ over $A$, equipped with embeddings of $B$ and $C$ over
$A$, such that $F$ is generated as a difference field by the images
of $B$ and $C$. 
\item Let $F$ be a compositum of $B$ and $C$ over $A$. We say that $F$
is a \emph{linearly free compositum} if $B$ and $C$ are linearly
disjoint over $A$ in $F$ in the sense of \ref{subsec:Linear-Disjointness}. 
\item We say that $B$ and $C$ are \emph{abstractly linearly disjoint}
over $A$ or simply \emph{linearly disjoint} over $A$ if they admit
a linearly free compositum over $A$. 
\end{enumerate}
\begin{rem}
Let $B$ and $C$ be difference fields over $A$, jointly embedded over $A$
in an ambient extension $\Omega$. Then whether or not $B$ and $C$
are linearly disjoint over $A$ in $\Omega$ depends only on the underlying
field structure and not on the action of $\sigma$. By contrast, abstract
linear disjointness depends on the difference field structure. Namely
the category of fields admits linearly disjoint amalgamation over
an algebraically closed base, whereas the category of difference fields
only admits linearly disjoint amalgamation over a base which is algebraically
closed and inversive.
\end{rem}

Let $B$ and $C$ be difference fields over $A$. Then a linearly
free compositum $F$ of $B$ and $C$ over $A$, if it exists, is
unique up to a unique isomorphism. More precisely, the difference field
$F$ is characterized by the following universal mapping property.
Let $E$ be a difference field extension of $A$. Then the datum of
an embedding of $F$ in $E$ over $A$ is equivalent to the datum
of embeddings of $B$ and $C$ in $E$ over $A$, so as to render
them linearly disjoint over $A$ in the image.
\end{defn}

\subsubsection{Tensor Product of Difference Fields}
\begin{defn}
\label{def:field-tensor}
Let $B$ and $A_{0}$ be difference fields,
abstractly linearly disjoint over $A$. Then we write $B_0=B\otimes_{A}A_{0}$
for the linearly free compositum of $A$; by the discussion of \ref{subsec:Linear-Disjointness},
it is uniquely determined up to a unique isomorphism. In this situation,
we say that $B_0=B\otimes_{A}A_{0}$ is the \emph{(difference) field theoretic tensor
product} of $B$ and $A_{0}$ over $A$.
\end{defn}

In the settings of Definition \ref{def:field-tensor}, we say that
$\nicefrac{B_{0}}{A_{0}}$ is obtained from $\nicefrac{B}{A}$ via
\emph{base extension}.

\begin{rem}
    Let $K$ be an abstract field. We may regard $K$ as a difference field via the identity automorphism. In this way the definitions apply to abstract fields as well.
\end{rem}

\begin{rem}
\label{rem:warning}(Warning) The notation of Definition \ref{def:field-tensor}
is potentially ambiguous as we now explain. Let us use temporarily
the notation $\overline{\otimes}$ for the field theoretic tensor
product. Let $L$ and $M$ be difference fields over $K$ and write
$R=L\otimes_{K}M$ for the usual tensor product of rings. The ring
$R$ comes equipped with a canonical endomorphism, compatible with
both inclusions. The difference fields $L$ and $M$ are abstractly
linearly disjoint over $K$ if and only if the ring $R$ is an integral
domain and $\sigma$ is injective on $R$. In this situation, the
endomorphism $\sigma$ of $R$ lifts uniquely to the field of fractions
$F$ of $R$, and the field theoretic tensor product is just the field
of fractions of the usual tensor product, i.e, we have $L\overline{\otimes}_{K}M=\mathrm{Frac}\left(R\right)$.
However, in this work we will deal exclusively with fields, as opposed
to rings, so no confusion should arise.
\end{rem}

\begin{rem}
The field theoretic tensor is associative and compatible with directed
unions whenever defined. This follows immediately from the universal
mapping property of the linearly free compositum of \ref{subsec:Abstract-Linear-Disjointness}.
\end{rem}

\begin{rem}
Let $B$ and $A_{0}$ be difference fields, abstractly linearly disjoint
over $A$, and write $B_{0}=B\otimes_{A}A_{0}$ for the field theoretic
tensor product. In model theory, it is common to regard linear disjointness
as an abstract independence relation, obeying properties similar to
forking independence. Moreover, we have given a description of the
field theoretic tensor product which is of category theoretic flavor.
However, despite the symmetry of the construction, the role played
by $B$ and $A_{0}$ in our application is not symmetric: the idea
is that for $A_{0}$ an arbitrary extension of $A$, properties of
the extension $\nicefrac{B}{A}$ are inherited by the base extension
$\nicefrac{B_{0}}{A_{0}}$ and vice versa; see the discussion of \ref{subsec:descent}.
\end{rem}

\subsubsection{\label{subsec:descent}Invariance under Base Extension}

Here is the key definition. 
\begin{defn}
\label{def:descent-along-base-extension}
Let $\mathcal{P}$ be a property
of embeddings of difference fields. Let $\nicefrac{B}{A}$ be an extension
of difference fields and assume that $\nicefrac{B_{0}}{A_{0}}$ is
obtained from $\nicefrac{B}{A}$ via base extension, thus $B_0 = B \otimes_{A} A_0$ (in particular $B$ and $A_0$ are linearly disjoint over $A$). We say that the property $\mathcal{P}$
is \emph{invariant under base extension} if the two extensions $\nicefrac{B}{A}$
and $\nicefrac{B_{0}}{A_{0}}$ simultaneously enjoy the property
$\mathcal{P}$.
\end{defn}

We will only be interested in properties of embeddings of difference
fields which are invariant under base extension; this is the case for
many natural classes of interest, see Fact \ref{fact:basic-descent}.
Moreover, we will encounter the following situation. Suppose we wish
to prove a statement about an extension $\nicefrac{B}{A}$ of difference
fields. If we can show that the hypothesis and the conclusion are
both invariant under base extension, then we can pick an extension
$A_{0}$ of $A$ abstractly linearly disjoint from $B$ over $A$,
and study the base extension $\nicefrac{B_{0}}{A_{0}}$ instead, with
no loss of generality. In practice, we will almost always take $A_{0}$
to be the perfect or inversive hull of $A$.

\subsubsection{\label{subsec:Basic-Properties}Basic Properties}

The following elementary properties of abstract linear disjointness
will be used repeatedly, often without mention. See e.g \cite[Chapter 2]{FrJaBook}.

Fix a tower $A\subseteq B\subseteq C$ of difference fields, and let
$A_{0}$ be a difference field extension of $A$. Let us assume that
$A_{0}$ and $C$ are linearly disjoint over $A$. Then $A_{0}$ and
$B$ are linearly disjoint over $A$. Furthermore, the difference
field $B_{0}=A_{0}\otimes_{A}B$ is linearly disjoint from $C$ over
$A$. The converse holds as well using associativity of the field
theoretic tensor product.

Moreover, linear disjointness is of finite character in the
following sense. Let $B\hookleftarrow A\hookrightarrow C$ be embeddings
of difference fields and assume that $B$ is the directed union of
difference field extensions $B_{i}$ of $A$; then $B$ and $C$ are
linearly disjoint over $A$ if and only if the $B_{i}$ are all linearly
disjoint from $C$ over $A$.

\subsubsection{Descent for Field Extensions}

Recall that the field extension $\nicefrac{K}{k}$ is \emph{relatively
perfect} if it is separable and one has $K\otimes_{k}k^{p^{-\infty}}=K^{p^{-\infty}}$.
\begin{fact}
\label{fact:basic-descent} The following properties of field extensions
are invariant under base extension: separable extensions, relatively
perfect extensions, primary extensions, and regular extensions.
\end{fact}

\begin{proof}
The results are well known, so we only sketch the proofs. Let $\nicefrac{K}{k}$
be a field extension and $\nicefrac{M}{m}$ a base extension, i.e,
we have $M=K\otimes_{k}m$.

\textbf{Separable Extensions.} We want to prove that $\nicefrac{K}{k}$
is separable if and only if $\nicefrac{M}{m}$ is so. If $M$ is separable
over $m$, then by definition the fields $M$ and $m^{p^{-\infty}}$
are linearly disjoint over $m$. By transitivity of linear disjointness,
the fields $K$ and $m^{p^{-\infty}}$ are linearly disjoint over
$k$. In view of the inclusion $k^{p^{-\infty}}\subseteq m{}^{p^{-\infty}}$,
it follows in particular that $K$ and $k^{p^{-\infty}}$ are linearly
disjoint over $k$, which is to say that $K$ is separable over $k$.

Now assume conversely that $K$ is separable over $k$; we want to
prove that $M$ is separable over $m$. By finite character of separable
extensions and the discussion of \ref{subsec:Basic-Properties}, we
may assume that $K$ is finitely generated over $k$. So $K$ is separably
generated over $k$, i.e, there is a transcendence basis $x$ of $K$
over $k$ such that $K$ is separably algebraic over $k\left(x\right)$.
It then follows immediately from the definition of linear disjointness
that $x$ is also a separating transcendence basis of $M$ over $m$.

\textbf{Primary Extensions.} First assume that $M$ is primary over
$m$. Then $M$ and $m^{\sep}$ are linearly disjoint over $m$. By
transitivity of linear disjointness the fields $K$ and $m^{\sep}$
are linearly disjoint over $k$, hence in particular $K$ and $k^{\sep}$
are linearly disjoint over $k$.

Now assume that $K$ is primary over $k$; we show that $M$ is primary over $m$. We will give a model
theoretic argument. By finite character we may assume that $K=k\left(\alpha\right)$
for a finite sequence $\alpha$. Recall that the theory of algebraically
closed fields eliminates imaginaries in the field language and that
the definable closure of a subfield is the perfect hull. It is moreover
a stable theory. So $K$ is primary over $k$ if and only if the type
$\tp_{\ACF}\left(\nicefrac{\alpha}{k}\right)$ admits a unique global
nonforking extension. Now in the theory $\ACF$, forking is characterized
by a drop in the transcendence degree. Since $K$ and $m$ are linearly
disjoint over $k$, a transcendence basis of $K$ over $k$ lifts
to a transcendence basis of $M$ over $m$. So as $\tp\left(\nicefrac{\alpha}{m}\right)$
does not fork over $k$ we find that it admits a unique global nonforking
extension too. So $\tp\left(\nicefrac{\alpha}{m}\right)$ is stationary,
whence $M$ is primary over $m$.

\textbf{Regular Extensions.} This follows from the fact that an extension
is regular if and only if it is both primary and separable.

\textbf{Relatively Perfect Extensions.} We want to prove that $\nicefrac{K}{k}$
is relatively perfect if and only if $\nicefrac{M}{m}$ is so. Note
that relatively perfect extensions are in particular separable; we
have already shown invariance of separable extensions, so assume that
one (and hence both) of these extensions is separable.

Let us assume first that $k$ and $m$ are both perfect. By definition,
a relatively perfect extension of a perfect field is just a perfect
extension. Thus we must show that if one of $K$ or $M$ is perfect
then so is the other. If $K$ is perfect, then $M$ is perfect, since it is
the tensor product of the perfect fields $K$ and $m$. Conversely,
assume that $M$ is perfect. The field $M$ is separable over $K$
as it is obtained from the separable extension $\nicefrac{m}{k}$
via base extension. Since no separable extension of an imperfect field
can possibly be perfect, we find that $K$ is perfect.

Next consider the case where $m=k^{p^{-\infty}}$ is the perfect hull
of $k$. If $K$ is relatively perfect over $k$, then by definition
$M$ is the perfect hull of $K$, and hence perfect; so it is a relatively
perfect extension of $m$. Conversely, assume $M$ is perfect. The
field $M$ is purely inseparably algebraic over $K$, since it is
obtained from the purely inseparably algebraic extension $\nicefrac{m}{k}$
via base extension. Since it is perfect, it is also the maximal purely
inseparably algebraic extension of $K$, i.e, we have $M=K^{p^{-\infty}}$,
which means that $\nicefrac{K}{k}$ is relatively perfect.

The general case follows from transitivity, combining the two special
cases considered above together.
\end{proof}
The class of separable extensions is in general not transitive in
towers. However, transitivity holds if the extension on the bottom
is relatively perfect: 
\begin{fact}
\label{fact:rel-perf-transitive-in-towers}Let $K\subseteq M\subseteq L$
be a tower of fields and assume $M$ is relatively perfect over $K$.
Then $L$ is separable over $K$ if and only if it is separable over
$M$. 
\end{fact}

\begin{proof}
Let us assume first that $L$ is separable over $M$. We must prove
that $L$ is separable over $K$. Since $\nicefrac{M}{K}$ is relatively
perfect, it is in particular separable. The field extension $\nicefrac{L}{K}$
is therefore separable as the composition of such.

Let us assume conversely that $L$ is separable over $K$ and that
$M$ is relatively perfect over $K$; we must prove that $L$ is separable
over $M$. Let $\overline{K}=K^{p^{-\infty}}$ be the perfect hull
of $K$. Let us set $\overline{M}=M\otimes_{K}\overline{K}$ and likewise
$\overline{L}$. Since separable and relatively perfect extensions
are both invariant under base extension (Fact \ref{fact:basic-descent}),
replacing everything by the overlined data, we may assume that $K$
is perfect. But then it follows immediately from the definition of
relatively perfect extensions that $M$ is perfect. In this situation,
the field $L$ is automatically separable over $M$, as wanted. 
\end{proof}

\section{Transformally Separable Extensions}\label{s:tran separable}

Let $\nicefrac{K}{k}$ be a field extension. Recall that $K$ is said to be \textit{separable} over $k$ if the fields $K$ and $k^{p^{-\infty}}$ are linearly disjoint over $k$. The purpose of this section is to study this notion in the transformal setting.

\subsection{Transformally Separable Extensions}
The following definition will play a central role in this work.
\begin{defn}
Let $\nicefrac{L}{K}$ be an extension of difference fields. 
\begin{enumerate}
\item We say that $L$ is \emph{transformally separable }over $K$ if
the difference fields $K$ and $L^{\sigma}$ are linearly disjoint
over $K^{\sigma}$. 
\item We say that $L$ is \emph{transformally separably algebraic} over
$K$ if it is transformally separable and transformally algebraic
over $K$. 
\item We say that $L$ is \emph{purely transformally inseparably algebraic}
over $K$ if $K^{\sigma^{-\infty}}=L^{\sigma^{-\infty}}$, that is, for all
$a\in L$ we have $a^{\sigma^{n}}\in K$ for $n\gg0$ sufficiently
large.
\end{enumerate}
\end{defn}
 
\begin{rem}
\label{R:equiv def of trans sep} Let $\nicefrac{L}{K}$ be an extension
of difference fields. One checks that the following conditions are
equivalent:
\begin{enumerate}
\item The difference field extension $\nicefrac{L}{K}$ is transformally
separable.
\item The difference fields $K^{\sigma^{-1}}$ and $L$ are linearly disjoint
over $K$.
\item The difference fields $K^{\sigma^{-\infty}}$ and $L$ are linearly
disjoint over $K$.
\end{enumerate}
\end{rem}

\begin{rem}
Let $\nicefrac{L}{K}$ be a transformally separable extension. Then
$K^{\sigma^{-\infty}}$ and $L$ are linearly disjoint over $K$,
hence $K^{\sigma^{-\infty}}\cap L=K$, the intersection taken in $L^{\sigma^{-\infty}}$.
This condition is however not in general sufficient. See \cite[Remark 2.7.6]{FrJaBook}
for an example where $\sigma$ coincides with the Frobenius endomorphism.
\end{rem}

\begin{rem}
\label{rem:failure-of-transitivity}The class of transformally separable
extensions is in general not transitive in towers. More precisely,
let $K\subseteq M\subseteq L$ be a tower of difference fields; then
it may be that $\nicefrac{L}{K}$ is transformally separable but $\nicefrac{L}{M}$
is not. For example, let $K$ be an inversive difference field, let
$M$ be a difference field extension of $K$ which is not inversive,
and let $L$ be the inversive hull of $M$. Since $K$ is inversive,
the extension $\nicefrac{L}{K}$ is transformally separable; but $\nicefrac{L}{M}$
is not. See however \ref{subsec:Transitivity-in-Towers} for a partial
converse.
\end{rem}

The following summarizes some basic formal properties of transformally
separable extensions; it will be used repeatedly throughout this work. 
\begin{prop}
\label{basic-properties-of-tsep}Let $K\subseteq M\subseteq L$ be
a tower of difference fields. 
\begin{enumerate}
\item The composition of transformally separable extensions is again one.
More precisely, if $\nicefrac{M}{K}$ and $\nicefrac{L}{M}$ are transformally
separable, then $\nicefrac{L}{K}$ is transformally separable. 
\item The class of transformally separable extensions of a fixed base goes
downwards. More precisely, if $\nicefrac{L}{K}$ is transformally
separable then $\nicefrac{M}{K}$ is transformally separable. 
\item The class of transformally separable extensions is of finite character.
More precisely, the extension $\nicefrac{L}{K}$ is transformally
separable if and only if it is the directed union of finitely generated
transformally separable extensions. 
\item The class of transformally separable extensions is invariant under base
extension. More precisely, let $\nicefrac{L}{K}$ be a difference
field extension. Let $K_{0}$ be a difference field extension of $K$,
linearly disjoint from $L$ over $K$, and set $L_{0}=L\tensor_{K}K_{0}$;
then $\nicefrac{L}{K}$ is transformally separable if and only if
$\nicefrac{L_{0}}{K_{0}}$ is so.
\end{enumerate}
\end{prop}
\begin{proof}
This follows from formal properties of linear disjointness. Let us
prove, for example, that the class of transformally separable extensions
descends along base extension; we use the notation of (4). 

Let us assume first that $\nicefrac{L_{0}}{K_{0}}$ is transformally separable. We will show that $\nicefrac{L}{K}$ is transformally separable. The assumption that $\nicefrac{L_{0}}{K_{0}}$ is transformally separable means that $L_{0}$ and $K_{0}^{\sigma^{-\infty}}$ are linearly disjoint over $K_{0}$. By transitivity of linear disjointness, the fields $L$ and $K_{0}^{\sigma^{-\infty}}$ are linearly disjoint over $K$. In view of the embedding $K^{\sigma^{-\infty}}\subseteq K_{0}^{\sigma^{-\infty}}$, it follows in particular that $L$ and $K^{\sigma^{-\infty}}$ are linearly disjoint over $K$, which is to say that $\nicefrac{L}{K}$ is transformally separable, as wanted.

Let us assume conversely that $\nicefrac{L}{K}$ is transformally
separable; we must show that $L_0$ is transformally separable over $K_0$. By transitivity of linear disjointness and Remark \ref{R:equiv def of trans sep}, it is enough to prove that $L$ is linearly disjoint from $(K_0)^{\sigma^{-1}}$ over $K$. 

So let $x_1,\dots,x_n\in L$ be linearly dependent over $(K_0)^{{\sigma}^{-1}}$ , i.e. there are $m_i^{\sigma^{-1}}\in (K_0)^{\sigma^{-1}}$, not all zero, such that the equality $\sum m_i^{\sigma^{-1}}x_i=0$ holds. We must show that the $x_i$ are linearly dependent over $K$ already. 
Applying $\sigma$, we get that $\sum m_ix_i^{\sigma}=0$. By linear disjointness of $L$ from $K_0$ over $K$, there are $k_i\in K$ with $\sum k_ix_i^{\sigma}=0$ and hence $\sum k_i^{\sigma^{-1}}x_i=0$. So the $x_i$ are linearly dependent over $K^{\sigma^{-1}}$; since $L$ is transformally separable over $K$, they are linearly dependent over $K$ already, and hence in particular over $K_0$.
\end{proof}

\subsection{Difference Fields with Twists}

Let $K$ be a difference field of characteristic exponent $p$. We
write $\phi x=x^{p}$ for the Frobenius endomorphism of $K$, thus
$\phi$ is the identity map in characteristic exponent $p=1$.

\begin{defn}
\label{def:twisted-difference-fields}
Let $K$ be a difference field.
We say that $K$ is \emph{closed under twists} if every power of
the Frobenius endomorphism of $K$ factors through $\sigma$. In other
words, for every $n\in\mathbf{N}$ there is an endomorphism $\tau\colon K\to K$
of fields such that $\tau\circ\phi^{n}=\sigma$.
\end{defn}

In the situation of Definition \ref{def:twisted-difference-fields},
the endomorphism $\tau$ is unique, and we write $\tau=\frac{\sigma}{p^{n}}$.
\begin{rem}
\label{rem:closed-under-twists-characterization}Let $K$ be a difference
field. Then the following are equivalent: 
\begin{enumerate}
\item The difference field $K$ is closed under twists.
\item For every $x\in K$ and every $n\in\mathbf{N}$, the polynomial $X^{p^{n}}-x^{\sigma}$
splits in $K$.
\item The field $K$ contains the perfect hull of $K^{\sigma}$.
\item The field $K^{\sigma^{-1}}$ contains the perfect hull of $K$.
\end{enumerate}
\end{rem}

\begin{rem}
\label{rem:inversive-twisted-is-perfect}Let $K$ be a difference
field which is closed under twists. If $K$ is inversive, then $K$
is perfect. Indeed, by definition we have the inclusions $K\subseteq K^{p^{-\infty}}\subseteq K^{\sigma^{-1}}$;
since $K=K^{\sigma^{-1}}$, the inclusion $K\subseteq K^{p^{-\infty}}$
is an equality.
\end{rem}

\begin{rem}\label{R:comp closed under twist}
The compositum of difference fields closed under twists inside an
ambient difference field is closed under twists, regardless of linear
disjointness assumptions. 
\end{rem}

\begin{prop}
\label{prop:elementary-twisted}The following holds:

(1) Let $K$ be a difference field. Then there is a difference field
$\widetilde{K}$ over $K$ closed under twists which embeds uniquely,
over $K$ in every other extension with this property. The field $\widetilde{K}$
is purely inseparably algebraic over $K$.

(2) Let $K$ be a difference field which is closed under twists. Let
$L$ be an abstract purely inseparably algebraic field extension of
$K$. Then $\sigma$ lifts uniquely to $L$, and equipped with this
lift, the difference field $L$ is closed under twists.

(3) The class of difference fields closed under twists is stable in
directed unions, relatively perfect extensions and algebraic extensions.
\end{prop}

\begin{proof}
(1) Let $\widetilde{K}$ be the splitting field of the polynomials
of the form $X^{p^{n}}-x^{\sigma}$ over $K$, for $x\in K$ and $n\in\mathbf{N}$
varying. Then $\widetilde{K}$ is a normal purely inseparably algebraic
extension of $K$. By construction, if the polynomial $f\in K\left[x\right]$
splits in $\widetilde{K}$, then the same is true of the polynomial
obtained by applying
$\sigma$ to the coefficients. By uniqueness of splitting fields,
the endomorphism $\sigma$ of $K$ lifts to $\widetilde{K}$. Since
$\widetilde{K}$ is purely inseparably algebraic over $K$, the lifting
is unique. We claim that $\widetilde{K}$ is closed under twists,
and that it embeds uniquely over $K$ in every other extension with
this property.

First we prove that $\widetilde{K}$ is closed under twists. Fix an
element $a\in\widetilde{K}$; we must show that $a^{\sigma}$ has
all $p$-th power roots in $\widetilde{K}$. It is enough to prove this after replacing $a$ by
$a^{p^{n}}$ for $n\gg0$. Since $\widetilde{K}$
is purely inseparably algebraic over $K$, we may therefore assume
without loss of generality that the element $a$ lies in $K$. But
then it follows from the construction that the polynomials $X^{p^{n}}-a^{\sigma}$
all split in $\widetilde{K}$.

Now let $L$ be an extension of $K$ closed under twists. By Remark
\ref{rem:closed-under-twists-characterization}, for every $x\in K$
and every $n\in\mathbf{N}$, the polynomial $X^{p^{n}}-x^{\sigma}$
splits in $L$. By uniqueness of splitting fields, there is an abstract
field embedding of $\widetilde{K}$ in $L$ over $K$; since $\widetilde{K}$
is purely inseparably algebraic over $K$, this embedding is unique
and compatible with the action of $\sigma$.

(2) The lift of $\sigma$ is unique if it exists, as $L$ is purely
inseparably algebraic over $K$. Moreover, by assumption, we have
the series of inclusions $K\subseteq L\subseteq K^{p^{-\infty}}\subseteq K^{\sigma^{-1}}$;
if $a\in L$ then $a^{\sigma}$ lies in $K$, hence in $L$. So $L$
is a difference field. Moreover, we have the inclusions $L\subseteq L^{p^{-\infty}}=K^{p^{-\infty}}\subseteq K^{\sigma^{-1}}\subseteq L^{\sigma^{-1}}$,
so $L$ is closed under twists.

(3) Stability under directed unions is clear. Next assume that $L$
is relatively perfect over $K$ and that $K$ is closed under twists;
we want to prove that $L$ is closed under twists as well. By assumption
we have $L^{\sigma\cdot p^{-\infty}}=L^{\sigma}\otimes_{K^{\sigma}}K^{\sigma\cdot p^{-\infty}}$;
so the inclusion $K^{\sigma\cdot p^{-\infty}}\subseteq K$ gives an
inclusion $L^{\sigma\cdot p^{-\infty}}\subseteq L$. Finally, assume
that $L$ is algebraic over $K$. There is a canonical factorization
$K\subseteq M\subseteq L$ where $\nicefrac{M}{K}$ is separably algebraic
and $\nicefrac{L}{M}$ purely inseparable; in view of (2), we may
assume that $L$ is separably algebraic over $K$. In this situation,
the field $L$ is relatively perfect over $K$; this case has already been dealt with, so we are done.
\end{proof}

\subsection{Twisted Difference Polynomials}

In this subsection we introduce twisted difference polynomials.

\subsubsection{ }

Recall that $\mathbf{N}\left[\sigma\right]$ is the commutative semiring
of finite formal sums of the form $\lambda_{0}+\lambda_{1}\cdot\sigma+\ldots+\lambda_{N}\cdot\sigma^{N}$
where the $\lambda_{n}$ are natural numbers, with pointwise addition and multiplication
given by the rule $\sigma^{n}\cdot\sigma^{m}=\sigma^{n+m}$ for $n,m\in\mathbf{N}$.
We write $\mathbf{N}\left[\frac{\sigma}{p^{\infty}}\right]$ for the
union $\bigcup_{n=0}^{\infty}\mathbf{N}\left[\frac{\sigma}{p^{n}}\right]$,
so every element of $\mathbf{N}\left[\frac{\sigma}{p^{\infty}}\right]$
can be written in the form $\lambda_{0}+\lambda_{1}\cdot\frac{\sigma}{q_{1}}+\ldots+\lambda_{N}\cdot\frac{\sigma^{N}}{q_{N}}$
where the $q_{i}$ are powers of $p$.

\subsubsection{\label{subsec:The-Ring-of-Twisted}The Ring of Twisted Difference
Polynomials}

Let $K$ be a difference field which is closed under twists. We now
define the ring of \emph{twisted difference polynomials} over $K$,
denoted by $A=K\left[x^{\mathbf{N}\left[\frac{\sigma}{p^{\infty}}\right]}\right]$.
The elements of $A$ are formal finite $K$-linear combinations of
twisted difference monomials of the form $x^{\nu}$ where $\nu\in \mathbf{N}\left[\frac{\sigma}{p^{\infty}}\right]$.
Multiplication is given by the rule $x^{\nu}\cdot x^{\mu}=x^{\nu+\mu}$,
and we have an action of $\sigma$ on $K\left[x^{\mathbf{N}\left[\frac{\sigma}{p^{\infty}}\right]}\right]$ given by $\left(x^{\nu}\right)^{\sigma}=x^{\sigma\cdot\nu}$.
The difference ring $A$ is transformally integral in the sense that the underlying ring is an integral domain and $\sigma$ is injective. It is moreover closed under
twists in the sense that for every $f\in A$ and $n\in\mathbf{N}$
there is a unique $g\in A$ with $g^{p^{n}}=f^{\sigma}$. Moreover,
for an element $f\in A$ there is an evident evaluation map $K\to K$
denoted by $a\mapsto fa$.

\subsubsection{\label{subsec:Derivatives-of-Twisted}Derivatives of Twisted Difference
Polynomials}

Let $K$ be a difference field, closed under twists. Let $A=K\left[x^{\mathbf{N}\left[\frac{\sigma}{p^{\infty}}\right]}\right]$
be the difference ring of twisted difference polynomials over $K$.
The conditions\footnote{In exponential characteristic $p\neq1$ the requirement that $dx^{\sigma}=0$
is redundant; we have $x^{\sigma}=y^{p}$ where $y=x^{\frac{\sigma}{p}}$,
so $dx^{\sigma}=p\cdot dy=0$} $dx=1$ and $dx^{\sigma^{n}}=0$ for $n>0$ specify a unique $K$-derivation
$d\colon A\to A$. For $f\in A$ we write $df=f'$ and refer to $f'$
as the \emph{transformal derivative} or simply the \emph{derivative}
of $f$. 

\subsection{The Theory $\FF$}\label{ss:FF}

We will be interested in the following class of difference fields,
for which the notion of a transformally separable extension is particularly
well behaved. In \cite{SCFEe} this theory is denoted by $T_\sigma$.
\begin{defn}\label{D:FE}
Let $\FF$ denote the following first order theory of difference fields
$K$: 
\begin{enumerate}
\item The difference field $K$ is closed under twists in the sense of Definition
\ref{def:twisted-difference-fields}; that is, the field $K$ contains
the perfect hull of $K^{\sigma}$. 
\item The field $K$ is a primary extension of $K^{\sigma}$; that is, the
field $K^{\sigma}$ is relatively separably algebraically closed in
$K$. 
\end{enumerate}
\end{defn}

\begin{rem}
We shall regard models of $\FF$ as structures for the language of difference fields; in this language an embedding of models of $\FF$ is an embedding of difference fields, that is, a ring homomorphism compatible with the action of $\sigma$.
\end{rem}

It is clear that the EC models of $\FF$ are inversive; to rule this out we introduce the following definition:

\begin{defn}\label{D:FE-lambda}

We let $\FF_{\lambda}$ be the theory whose models are difference fields as in \ref{D:FE} but embeddings are required to be transformally separable. This can be made first order either with relations expressing linear independence of tuples in $F^n$ over $F^{\sigma}$ or in a functional language, namely with the transformal $\lambda$-functions (see \ref{s:trans lambda functions}). The exact choice of language will not matter for us; regardless of the choice, the theory $\FF_{\lambda}$ is a conservative expansion by definitions of $\FF$, i.e, every model of $\FF$ admits a unique expansion to a model of $\FF_{\lambda}$.
\end{defn}

\begin{rem}
Let $K$ be a difference field. Then the following are equivalent: 
\begin{enumerate}
\item The field $K^{\sigma}$ is separably algebraically closed in $K$. 
\item The field $K$ is separably algebraically closed in $K^{\sigma^{-1}}$. 
\item The field $K$ is separably algebraically closed in its inversive
hull $K^{\sigma^{-\infty}}$.
\end{enumerate}
\end{rem}

\begin{lem}
\label{fe-hull} Let $K$ be a difference field. Then there is a difference
field extension $E$ of $K$ which is a model of $\FF$ and which
embeds uniquely over $K$ in any other. Moreover, if $K$ is closed
under twists, then $E$ is separably algebraic over $K$, and at all
events it is algebraic over $K$. 
\end{lem}

In the situation of Lemma \ref{fe-hull}, we will say (for lack of
better terminology) that $E$ is the \emph{smallest model of $\FF$
over $K$}.
\begin{proof}
We may assume that $K$ is closed under twists, by Proposition \ref{prop:elementary-twisted}.

Let $K^{\sigma^{-\infty}}$ denote the inversive hull of $K$, and
let $E$ be the separable algebraic closure of $K$ inside. Then $E$
is a difference field: if $a\in K^{\sigma^{-\infty}}$ is separably
algebraic over $K$, then $a^{\sigma}$ is separably algebraic over
$K^{\sigma}$ and hence over $K$. We claim that $E$ is advertised;
for this, we must prove $E$ is a model of $\FF$ which embeds uniquely
over $K$ in any other.

First we check that $E$ is a model of the theory $\FF$. By construction,
the field $E$ is separably algebraically closed in $K^{\sigma^{-\infty}}$;
on the other hand, we have $E^{\sigma^{-\infty}}=K^{\sigma^{-\infty}}$,
so $E$ is separably algebraically closed in its inversive hull. Since
$E$ is an algebraic extension of $K$, which is closed under twists,
it is again closed under twists, using Proposition \ref{prop:elementary-twisted}
again. It follows that $E$ is a model of the theory $\FF$.

To finish, we must prove that $E$ embeds uniquely over $K$ in every
model $L$ of $\FF$ over $K$. So fix such a model $L$. The embedding
of $K$ in $L$ lifts uniquely to an embedding of $K^{\sigma^{-\infty}}$
in $L^{\sigma^{-\infty}}$. Since the field $L$ is a model of $\FF$,
it is separably algebraically closed in its inversive hull; so the
restriction of this embedding to $E$ factors through $L$, by definition
of the relative separable algebraic closure. So $E$ embeds uniquely
over $K$ in $L$, and the proof is finished.
\end{proof}

\subsubsection{Closure Properties}
\begin{lem}
\label{lem:purely inseparable FE}Let $K$ be a model of $\FF$ and
let $L$ be an abstract purely inseparably algebraic extension of
$K$. Then $\sigma$ lifts uniquely to $L$, and this lift renders
$L$ a model of $\FF$. 
\end{lem}

\begin{proof}
Since $K$ is closed under twists, it follows from Proposition \ref{prop:elementary-twisted}
that $\sigma$ lifts uniquely to $L$, and that $L$ is again closed
under twists when equipped with this lift. It remains to be shown
that the algebraic closure of $L$ in $L^{\sigma^{-\infty}}$ is purely
inseparable over $L$. But $L$ is purely inseparably algebraic over $K$; in
view of the isomorphism $K^{\sigma^{-\infty}}=L^{\sigma^{-\infty}}$, the claim follows
from the fact that $L$ is purely inseparably algebraic over $K$.
\end{proof}
\begin{lem}
\label{lem:transformally-transcendental-FE}
Let $K$ be a model of $\FF$ and let $x$ be an element which is
transformally transcendental over $K$, in an ambient extension. Then the difference field
$E=K\left(x^{\mathbf{N}\left[\frac{\sigma}{p^{\infty}}\right]}\right)$
is a model of $\FF$.
\end{lem}

\begin{proof}
Implicit in the choice of notation is the fact that $E$ is closed
under twists. To finish, we must prove that $E$ is primary over $E^{\sigma}$,
or equivalently over $E^{\sigma\cdot p^{-\infty}}$. The fields $K$
and $E^{\sigma\cdot p^{-\infty}}$ are linearly disjoint over $K^{\sigma\cdot p^{-\infty}}$ so
by descent for primary extensions, it is enough to prove that $E$ is primary over $K\tensor_{K^{\sigma\cdot p^{-\infty}}}E^{\sigma\cdot p^{-\infty}}$.
But this is evident from the construction: the field $E$ is generated
over the latter by the single element $x$ which is algebraically
transcendental.
\end{proof}
\begin{example}
    Let $L,M$ be models of $\FF$ linearly disjoint over a base model $K$ of $\FF$; then the tensor product $L\tensor_{K}M$ need not be a model of $\FF$. For example, let $K = \mathbf{Q} \left(x^\Ns\right)$ where $x$ is transformally transcendental, let $L$ be the inversive hull of $K$, and let $M = \mathbf{Q}\left(y^\Ns\right)$ where $y^2 = x$. Then $N = L \tensor_{K}M$ is not a model of $\FF$; its inversive hull is an algebraic extension, obtained by adjoining square roots of the elements $x^{\sigma^n}$ for all $n \in \mathbf{Z}$.
    See, however, Proposition \ref{prop:amalgamation-for-FE} for a partial converse when $K$ is separably algebraically closed and $L,M$ are transformally separable over $K$.
\end{example}

We will make use of the following (see \cite{ChHr-Dif}):
\begin{fact}
\label{fact:(The-Independence-Theorem)}(The Independence Theorem)
Let $E$ be an algebraically closed field and let $A,B,C$ be algebraically
closed field extensions of $E$, linearly disjoint over $E$ in some
ambient extension. Then the field $\left(AC\right)^{\alg}\left(BC\right)^{\alg}$
is a regular field extension of $AB$.
\end{fact}

\begin{prop}\label{prop:amalgamation-for-FE}
(The Amalgamation Property) 
Let $A$ be a model of $\FF$ which is
separably algebraically closed. Let $B,C$ be models of $\FF$ transformally
separable over $A$. Then $B$ and $C$ are linearly disjoint over
$A$ and the difference field $B\otimes_{A}C$ is a model of $\FF$.
\end{prop}

\begin{rem}
\label{rem:The-following-Remark}The following Remark will be used
in the proof without further mention. Let $\mathcal{C}$ be the class
of difference fields $F$ with the property that $F$ is relatively
separably algebraically closed in its inversive hull $F^{\sigma^{-\infty}}$.
Then $\mathcal{C}$ is downwards closed in transformally separable
extensions and in primary extensions. More precisely, let $\nicefrac{E}{F}$
be an extension of difference fields which is transformally separable
or primary. If $E$ is a member of $\mathcal{C}$, then $F$ is also
a member of $\mathcal{C}$.
\end{rem}

\begin{proof}
The category of difference fields admits linearly disjoint amalgamation
over a base which is inversive and algebraically closed. Since $A$
is closed under twists and separably algebraically closed, the inversive
hull of $A$ is algebraically closed. By transitivity of linear disjointness,
using the fact that $B,C$ are transformally separable over $A$,
we find that they are linearly disjoint over $A$. 

It remains to be shown that $B\otimes_{A}C$ is a model of $\FF$.
The tensor product of difference fields closed
under twists is again closed under twists; so we must prove that $B\otimes_{A}C$
is relatively separably algebraically closed in its inversive hull.
We prove this in several steps.

\textbf{Step 1.} We may assume that $A$ is perfect. Let $A_{0}=A^{p^{-\infty}}$
be the perfect hull of $A$. Since transformally separable extensions
of models of $\FF$ are algebraically separable (see Lemma \ref{closed-under-twists-algebraically-separable}), the fields $B,C$
are algebraically separable over $A$. Let $B_{0}=B\otimes_{A}A_{0}$
and likewise $C_{0}$. Since purely inseparably algebraic extensions
of models of $\FF$ are models of $\FF$ (Lemma \ref{lem:purely inseparable FE}) we have that $B_{0},C_{0}$ are models of $\FF$.
Using Remark \ref{rem:The-following-Remark} we can replace everything
by the subscripted data. By a similar argument we can assume that
$B$ is perfect.

\textbf{Step 2.} We may assume that the field $B$ is separably algebraically
closed. Namely let $B_{0}=B^{\sep}$. Then $B_{0}$ is a model of $\FF$ and
it is transformally separable over $B$. Using Remark \ref{rem:The-following-Remark}
it is enough to prove the claim after $B$ has been replaced by $B_{0}$.

Combining these two steps and using symmetry we may assume that the
fields $A,B,C$ are all algebraically closed.

\textbf{Step 3.} We may assume that $A$ is inversive. Let $A_{0}=A^{\sigma^{-\infty}}$
be the inversive hull of $A$. Let $B_{0}=\left(B\otimes_{A}A_{0}\right)^{\alg}$
and likewise $C_{0}=\left(C\otimes_{A}A_{0}\right)^{\alg}$. Then
$B_{0},C_{0}$ are algebraically closed and hence models of $\FF$.
Moreover, the fields $B_{0},C_{0}$ are transformally separable over
$A_{0}$, as the latter is inversive. By Fact \ref{fact:(The-Independence-Theorem)}
the field $B_{0}\otimes_{A_{0}}C_{0}$ is regular over $B\otimes_{A}C$
so using Remark \ref{rem:The-following-Remark} it is enough to prove
the claim for the subscripted data.

\textbf{Step 4.} We may assume that $B$ is inversive. Let $F=B\otimes_{A}C$,
let $B_{0}=B^{\sigma^{-\infty}}$ be the inversive hull of $B$ and
let $F_{0}=B_{0}\otimes_{A}C$. Since $B$ is algebraically closed
the field $B_{0}$ is regular over $B$. Since regular extensions
are invariant under base extension (Fact \ref{fact:basic-descent}), the field $F_{0}$
is regular over $F$. Using Remark \ref{rem:The-following-Remark},
we can replace $B$ by $B_{0}$.

Combining all these steps and using symmetry we may assume that $A,B,C$
are all inversive (and hence perfect). So $B\otimes_{A}C$ is perfect
and inversive as the tensor product of fields with this property;
it is then a model of $\FF$, which is what we wanted.
\end{proof}

\begin{rem}\label{rem:purely-insep}
Fix $e\in\mathbf{N}$ and write $\FF_{\leq e}$
for the class of models of $\FF$ of degree of imperfection at most
$e$. Let $A$ be a model of $\FF_{\leq e}$ which is separably algebraically
closed. Let $A_{1},A_{2}$ be models of $\FF_{e}$ transformally separable
over $A$ and set $B=A_{1}\tensor_{A}A_{2}$. Then $B$ is a model
of $\FF$ transformally separable over the $A_{i}$, hence over $A$.
If $\left(x_{i}\right)$ and $\left(y_{j}\right)$ are relative $p$-bases
of $A_{1}$ and $A_{2}$ over $A$ respectively, then $\left(x_{i},y_{j}\right)$
is a relative $p$-basis of $B$ over $A$; so the Ershov invariant
of $B$ will exceed $e$ in general (see \ref{subsec:Relative-and-Absolute}). However there is model $\widetilde{B}$
of $\FF$ purely inseparably algebraic over $B$ and transformally
separable over the $A_{i}$ which is a model of $\FF_{\leq e}$. To prove
this assume for simplicity of notation that $A$ is perfect and that
the $A_{i}$ are of Ershov invariant exactly $e$. Then the purely
inseparably algebraic extension $\widetilde{B}$ of $B$ obtained
by adjoining all $p$-th power roots to the elements $z_{1},\ldots,z_{e}$
is of Ershov invariant equal to $e$; here $z_{i}=x_{i}y_{i}$ for
$i=1,\ldots,e$. The field $\widetilde{B}$ is then evidently as promised.
\end{rem}

\subsection{Almost Transformally Separable Extensions}

The notion of an \emph{almost transformally separable} extension
is a soft variant of the notion of a transformally separable extension.
This allows breaking up many arguments into two parts, using the criterion
of Proposition \ref{criterion-algsep-almosttsep-for-fe}.

\subsubsection{\label{subsec:Algebraic-Disjointness}Algebraic Disjointness}

Let $B\hookleftarrow A\hookrightarrow C$ be fields, jointly embedded
over $A$ in an ambient extension $\Omega$. We say that $B$ and
$C$ are \emph{algebraically disjoint} over $A$ if whenever the
sequence $a_{1},\ldots,a_{n}\in B$ is algebraically independent over
$A$ then it remains algebraically independent over $C$. If the transcendence
degree of $B$ over $A$ is finite, then algebraic disjointness is
equivalent to the requirement that $\tdeg_{A}B=\tdeg_{C}CB$.

\begin{fact}\cite[Lemma 2.6.7]{FrJaBook}
\label{fact:alg-disjointness-equiv-lin-disjointness-for-reg}Let $L\hookleftarrow K\hookrightarrow M$
be field extensions, jointly embedded over $K$ in an ambient extension
$\Omega$. If $L$ and $M$ are linearly disjoint over $K$, then they are algebraically disjoint over $K$; the converse holds provided that $L$ is regular over $K$.
\end{fact}

\subsubsection{Almost Transformally Separable Extensions}
\begin{defn}
Let $\nicefrac{L}{K}$ be an extension of difference fields. We say
that $L$ is \emph{almost transformally separable }over $K$ if
the fields $K$ and $L^{\sigma}$ are algebraically disjoint over
$K^{\sigma}$ in $L$ the sense of \ref{subsec:Algebraic-Disjointness}. 
\end{defn}

\begin{rem}
\label{tsep-implies-almost-tsep}Let $\nicefrac{L}{K}$ be an extension
of difference fields. If $L$ is transformally separable over $K$,
then it is almost transformally separable over $K$; this follows
immediately from Fact \ref{fact:alg-disjointness-equiv-lin-disjointness-for-reg}.
\end{rem}

\begin{rem}
Let $\nicefrac{L}{K}$ be an extension of difference fields. Then
the following conditions are equivalent: 
\begin{enumerate}
\item The difference field extension $\nicefrac{L}{K}$ is almost transformally
separable. 
\item The difference fields $K^{\sigma^{-1}}$ and $L$ are algebraically
disjoint over $K$ in $L^{\sigma^{-1}}$.
\item The difference fields $K^{\sigma^{-\infty}}$ and $L$ are algebraically
disjoint over $K$ in $L^{\sigma^{-\infty}}$.
\end{enumerate}
\end{rem}

\begin{rem}
\label{insens-alg-almost-tsep} Let $K$ be an algebraically closed
difference field; then an extension of $K$ is almost transformally
separable if and only if it is transformally separable. This is an
immediate consequence of Fact \ref{fact:alg-disjointness-equiv-lin-disjointness-for-reg},
as every field extension of an algebraically closed field is a regular
extension. More generally, let $\nicefrac{L}{K}$ be an extension
 of difference fields. Then $\nicefrac{L}{K}$ is almost transformally
 separable if and only if $L^{\alg}$ is an almost transformally separable
 extension of $K^{\alg}$, for some (any) choice of an extension of $\sigma$ from $L$ to $L^{\alg}$.
\end{rem}

\begin{prop}
\label{basic-properties-almost-tsep} Let $K\subseteq M\subseteq L$
be a tower of difference fields. 
\begin{enumerate}
\item The composition of almost transformally separable extensions is again
one. More precisely, if $\nicefrac{M}{K}$ and $\nicefrac{L}{M}$
are almost transformally separable, then $L$ is almost transformally
separable over $K$.
\item The class of almost transformally separable extensions is downwards
closed. More precisely, if $\nicefrac{L}{K}$ is almost transformally
separable, then $\nicefrac{M}{K}$ is almost transformally separable.
\item The class of almost transformally separable extensions is of finite
character. More precisely, the extension $\nicefrac{L}{K}$ is almost transformally
separable if and only if it is the directed union of finitely generated
almost transformally separable extensions. 
\item The class of almost transformally separable extensions is invariant under base extension. More precisely, let $\nicefrac{L}{K}$ be a difference
field extension. Let $\overline{K}$ be a difference field extension
of $K$, linearly disjoint from $L$ over $K$, and set $\overline{L}=L\tensor_{K}\overline{K}$;
then $\nicefrac{L}{K}$ is almost transformally separable if and only
if $\nicefrac{\overline{L}}{\overline{K}}$ is almost transformally
separable. 
\end{enumerate}
\end{prop}

\begin{proof}
As in Proposition \ref{basic-properties-of-tsep}, this
follows from formal properties of algebraic disjointness.
\end{proof}

\subsection{A Criterion for Transformally Separable Extensions}

In this subsection we give a criterion for extensions to be transformally
separable when the base is a model of $\FF$.

We start with the following; see also \cite[Lemma 5.1]{hrushovski2004elementary}. 
\begin{lem}
\label{L:trans-alg-over inversive is algebraic over sigma} Let $\nicefrac{M}{K}$
be a transformally algebraic extension of difference field with $K$
inversive. Then $M$ is algebraic over $M^{\sigma}$. 
\end{lem}

\begin{proof}
The following elegant proof is due to the anonymous referee. By a limit argument we may assume that $M$ is finitely generated over $K$; since $M$ is transformally algebraic over $K$, the transcendence degree $\tdeg_{K}{M}$ is finite. Since $K$ is inversive, we have the inclusions $K \subseteq M^{\sigma} \subseteq M$; the claim follows by comparing transcendence degrees in the displayed tower. 
\end{proof}

\begin{lem}
\label{closed-under-twists-algebraically-separable}Let $K$ be a
difference field which is closed under twists. Then a transformally
separable extension of $K$ is algebraically separable. 
\end{lem}

\begin{proof}
Let $\nicefrac{L}{K}$ be transformally separable with $K$ closed
under twists. Then $K^{\sigma^{-1}}$ and $L$ are linearly disjoint
over $K$. Since $K$ is closed under twists, we have the inclusions
$K\subseteq K^{p^{-1}}\subseteq K^{\sigma^{-1}}$; by formal properties
of linear disjointness, we find that $K^{p^{-1}}$ and $L$ are linearly
disjoint over $K$, whence $\nicefrac{L}{K}$ is algebraically separable.
\end{proof}

\begin{lem}
\label{L:trans sep alg- perf hull splits} Let $\nicefrac{L}{K}$
be a transformally separably algebraic extension of difference fields, with
$K$ and $L$ closed under twists. Then the perfect hull of $L$ splits
over $K$. 
\end{lem}
\begin{proof}
Everything here is invariant under base extension (see Fact \ref{fact:basic-descent}
and Proposition \ref{basic-properties-of-tsep}), so assume that $K$
is inversive (Remark \ref{rem:inversive-twisted-is-perfect}). By
Lemma \ref{L:trans-alg-over inversive is algebraic over sigma},
the field $L$ is algebraic over $L^{\sigma}$, hence in particular
algebraic over $L^{\sigma\cdot p^{-\infty}}$. But algebraic extensions
of perfect fields are perfect, so $L$ is perfect, as wanted.
\end{proof}

The following criterion will be used repeatedly:
\begin{prop}
\label{criterion-algsep-almosttsep-for-fe}
Let $\nicefrac{L}{K}$
be an extension of difference fields with $K$ a model of $\FF$. Then $\nicefrac{L}{K}$
is transformally separable if and only if it is almost transformally
separable and algebraically separable. 
\end{prop}

\begin{proof}
If $K$ is a model of $\FF$, then it is closed under twists; a transformally
separable extension is then algebraically separable by Lemma \ref{closed-under-twists-algebraically-separable}.
Moreover, by Remark \ref{tsep-implies-almost-tsep}, an extension
which is transformally separable is almost transformally separable,
no matter the assumption on the ground field. Thus an extension of
a model of $\FF$ which is transformally separable is algebraically
separable and almost transformally separable; it remains to show the
converse.

We first reduce this to the case where $K$ is perfect, using the
discussion of \ref{subsec:descent}. By elementary properties of transformally
separable and almost transformally separable extensions (Proposition
\ref{basic-properties-almost-tsep} and Proposition \ref{basic-properties-of-tsep}),
the hypothesis and the conclusion are both invariant under base extension.
Furthermore, by Lemma \ref{lem:purely inseparable FE}, a purely inseparably
algebraic extension of a model of $\FF$ is a model of $\FF$. We
may therefore assume without loss of generality that $K$ is perfect.

It remains to be shown that if $K$ is a perfect model of $\FF$,
then an extension is transformally separable if and only if it is
almost transformally separable. Since $K$ is a model of $\FF$, the field $K^{\sigma^{-1}}$
is a primary extension of $K$; our assumption that $K$ is perfect
implies that $K^{\sigma^{-1}}$ is in fact regular over $K$. By Fact
\ref{fact:alg-disjointness-equiv-lin-disjointness-for-reg}, linear
disjointness form $K^{\sigma^{-1}}$ over $K$ is equivalent to algebraic
disjointness, which is precisely what we needed.
\end{proof}

\subsection{\label{subsec:Transitivity-in-Towers}Transitivity in Towers}

The class of transformally separable extensions is not in general
transitive in towers; see Remark \ref{rem:failure-of-transitivity}. In this
subsection we investigate to what extent the converse holds.

\begin{prop}
\label{transitive-almost-tsep-if-bottom-is-talg}The class of almost
transformally separable extensions is transitive in towers, provided
that the bottom extension is transformally algebraic. More precisely,
let $K\subseteq M\subseteq L$ be a tower of difference fields and
assume that $M$ is transformally algebraic over $K$. Then $\nicefrac{L}{K}$
is almost transformally separable if and only if $\nicefrac{M}{K}$
and $\nicefrac{L}{M}$ are so. 
\end{prop}
\begin{proof}
It is enough to prove that $\nicefrac{L}{M}$ is almost transformally
separable; the other directions follow formally from Proposition \ref{basic-properties-almost-tsep}.

Let us assume first that $K$ is inversive. Since $M$ is transformally algebraic over $K$, it follows from 
Lemma \ref{L:trans-alg-over inversive is algebraic over sigma} that
the field $M$ is an algebraic extension of $M^{\sigma}$. So $L$ is automatically almost transformally separable over $M$, as wanted.

It remains to reduce to the case where $K$ is inversive. Let $L^{\alg}$
be an algebraic closure of $L$ in the category of difference fields;
replacing $L$ by $L^{\alg}$ and $M,K$ by the algebraic closures
inside $L^{\alg}$ does not affect the statement and the conclusion
descends, using Remark \ref{insens-alg-almost-tsep}; thus we may
assume that $K$ is algebraically closed, and hence that $\nicefrac{L}{K}$
is transformally separable. Let $K_{0}$ be the inversive hull of
$K$. Since $\nicefrac{L}{K}$ is transformally separable, the difference
field $L$ is linearly disjoint from $K_{0}$ over $K$.

Let us set $M_{0}=M\tensor_{K}K_{0}$ and $L_{0}=L\tensor_{K}K_{0}$.
By Proposition \ref{basic-properties-almost-tsep}, the class of almost
transformally separable extensions is invariant under base extensions;
one checks moreover that the same is true of the class of transformally
algebraic extensions. It is therefore harmless to replace $K,M$ and
$L$ by $K_{0},M_{0}$ and $L_{0}$ respectively, which brings us
back to the situation already dealt with.
\end{proof}

Recall that an extension $\nicefrac{L}{K}$ of difference fields is said to be transformally separably algebraic if it is transformally separable and transformally algebraic.

\begin{prop}
\label{tsep-alg-of-fe-preserves-ershov} 
Let $\nicefrac{L}{K}$ be
a transformally separably algebraic extension of models of $\FF$. 
\begin{enumerate}
\item The perfect hull of $L$ splits over $K$. 
\item If the difference field $K$ is inversive then so is $L$. 
\end{enumerate}
\end{prop}
\begin{rem}
It follows in particular that the Ershov invariant does not change
in a transformally separably algebraic extension of models of $\FF$.

Note that for (2), any transformally algebraic extension of an inversive field is transformally separably algebraic.
\end{rem}
\begin{proof}
(1) This is a special case of Lemma \ref{L:trans sep alg- perf hull splits}, since models of $\FF$ are in particular closed under twists.

(2) Assume that $K$ is inversive; we show that $L$ is inversive, or equivalently that the inclusion $L^\sigma \subseteq L$ is an equality. For this, it is enough to show that $L$ is transformally separable over $L^\sigma$.
By the criterion of Proposition \ref{criterion-algsep-almosttsep-for-fe},
in order to prove that $L$ is transformally separable over $L^{\sigma}$,
it will be sufficient to prove that $L$ is algebraically separable
and almost transformally separable over $L^{\sigma}$.

First we show that $L$ is algebraically separable over $L^{\sigma}$.
By Remark \ref{rem:inversive-twisted-is-perfect}, inversive models
of $\FF$ are perfect; so using (1) we find that $L^{\sigma}$ is
perfect, as the isomorphic image of the perfect field $L$. So $L$
is automatically algebraically separable over $L^{\sigma}$. Since
$K$ is inversive, the extension $L$ of $K$ is almost transformally
separably algebraic; this class is transitive in towers, using Proposition
\ref{transitive-almost-tsep-if-bottom-is-talg}, so $L$ is almost
transformally separable over $L^{\sigma}$. This concludes the proof.
\end{proof}

\begin{prop}
\label{P:rel-tsep-alg-c} The class of transformally separable extensions
of models of $\FF$ is transitive in towers, provided that the bottom
extension is transformally algebraic. More precisely, let $K\subseteq M\subseteq L$
be a tower of models of $\FF$ with $\nicefrac{M}{K}$ transformally
algebraic; then $\nicefrac{L}{K}$ is transformally separable if and
only if $\nicefrac{L}{M}$ and $\nicefrac{M}{K}$ are so. 
\end{prop}

\begin{rem}
This should be compared with \cite[Lemma 5.2(2)]{hrushovski2004elementary}. 
\end{rem}

\begin{proof}
By formal properties of transformally separable extensions as in Proposition
\ref{basic-properties-of-tsep}, it will be sufficient to prove the
following: if $\nicefrac{L}{K}$ is a transformally separable extension
of models of $\FF$, and $K\subseteq M\subseteq L$ is a model of
$\FF$ transformally separably algebraic over $K$, then $\nicefrac{L}{M}$
is transformally separable. Using the criterion of Proposition \ref{criterion-algsep-almosttsep-for-fe}
it is enough to prove that $\nicefrac{L}{M}$ is almost transformally
separable and algebraically separable.

The fact that $\nicefrac{L}{M}$ is almost transformally separable
follows from Proposition \ref{transitive-almost-tsep-if-bottom-is-talg};
so we must show that $\nicefrac{L}{M}$ is algebraically separable.
By Proposition \ref{tsep-alg-of-fe-preserves-ershov} the perfect
hull of $M$ splits over $K$, so the claim follows from Fact \ref{fact:rel-perf-transitive-in-towers}. 
\end{proof}

\subsection{Simple Roots and Transformally Separably Algebraic Extensions}

Let $\nicefrac{E}{F}$ be a field extension. Then $E$ is separably
algebraic over $F$ if and only if every element of $E$ is a simple
root of a polynomial with coefficients in $F$. The purpose of this
subsection is to generalize this to transformally separably algebraic
extensions of models of $\FF$. Recall the discussion of \ref{subsec:The-Ring-of-Twisted}
on twisted difference polynomials and their derivatives.
\begin{defn}
Let $\nicefrac{L}{K}$ be an extension of difference fields closed
under twists, let $fx\in K\left[x^{\mathbf{N}\left[\frac{\sigma}{p^{\infty}}\right]}\right]$
be a twisted difference polynomial, and let $a\in L$ be an element;
then $a$ is said to be a \emph{simple root} of $f$ if $fa=0$
and $f'a\neq0$.
\end{defn}

\begin{rem}
\label{substitution}Let $\nicefrac{L}{K}$ be an extension of difference
fields closed under twists and let $fx\in K\left[x^{\mathbf{N}\left[\frac{\sigma}{p^{\infty}}\right]}\right]$
be a twisted difference polynomial. Let $a\in L$ be an element and
let $gx\in K\left[x\right]$ be the polynomial obtained from $fx$
via the substitutions $x^{\sigma^{n}\cdot p^{m}}\mapsto a^{\sigma^{n}\cdot p^{m}}$
for $n>0$ and $m\in\mathbf{Z}$; then $f'a=g'a$. In particular,
if we regard an ordinary polynomial over $K$ as a twisted difference
polynomial, then the derivative defined above has its usual meaning. 
\end{rem}

\begin{lem}
\label{L:almost-simple-root-almost-tsep}Let $K$ be a difference
field and let $L=K\left(a^{\mathbf{N}\left[\sigma\right]}\right)$,
where $a$ is a finite tuple. Then the following are equivalent: 
\begin{enumerate}
\item The difference field $L$ is transformally algebraic and almost transformally
separable over $K$. 
\item The tuple $a$ is algebraic over $K\left(a^{\sigma\cdot\mathbf{N}\left[\sigma\right]}\right)$. 
\end{enumerate}
\end{lem}

\begin{proof}
$(1)\implies(2)$ We may assume here that $a$ is a single element.
By assumption, the element $a$ is transformally algebraic over $K$,
so for some $n\in\mathbf{N}$ the sequence $a,\ldots,a^{\sigma^{n}}$
is algebraically dependent over $K$; let $n$ be chosen minimal with
respect to this property. Then by minimality of $n$ the sequence
$a,\ldots,a^{\sigma^{n-1}}$ is algebraically independent over $K$.
Since $L$ is almost transformally separable over $K$, the sequence
$a^{\sigma},\ldots,a^{\sigma^{n}}$ must then likewise be algebraically
independent over $K$; by Steinitz exchange, the element $a$ is algebraic
over $K\left(a^{\sigma},\ldots,a^{\sigma^{n}}\right)$, hence in particular
algebraic over $K\left(a^{\sigma\cdot\mathbf{N}\left[\sigma\right]}\right)$.

$(2)\implies(1)$ Since the tuple $a$ is finite there is some large
natural number $N$ with $a\in K\left(a^{\sigma},\ldots,a^{\sigma^{N}}\right)^{\alg}$.
Since $\sigma$ is an endomorphism of fields, it follows inductively
that for all $n\in\mathbf{N}$ we have $a^{\sigma^{n}}\in K^{\sigma^{n}}\left(a^{\sigma^{n+1}},\ldots,a^{\sigma^{n+N}}\right)^{\alg}$
and so in particular that \[a^{\sigma^{n}}\in K\left(a^{\sigma^{n+1}},\ldots,a^{\sigma^{n+N}}\right)^{\alg}.\]

Enlarging $a$ by the tuple $a\ldots a^{\sigma^{N-1}}$ does not affect
the statement; using the analysis of the previous paragraph we then
have that $a\in K\left(a^{\sigma}\right)^{\alg}$ and inductively
that $a,\ldots, a^{\sigma^{n-1}}\in K\left(a^{\sigma^{n}}\right)^{\alg}$
for all $n\in\mathbf{N}$. So the transcendence degree of $K\left(a^{\mathbf{N}\left[\sigma\right]}\right)$
over $K$ is finite, namely it is bounded by the length of the tuple
$a$; this means that $a$ is transformally algebraic over $K$.

It remains to be shown that $\nicefrac{L}{K}$ is almost transformally
separable: in other words, that the fields $K$ and $L^{\sigma}$
are algebraically disjoint over $K^{\sigma}$. By assumption the field
$L$ is algebraic over $KL^{\sigma}$. Moreover, we have shown that
$L$ is transformally algebraic over $K$, so the transcendence degree
$\tdeg_{K}L$ is finite. We then have: 
\[
\tdeg_{K}L=\tdeg_{K}KL^{\sigma}\leq\tdeg_{K^{\sigma}}L^{\sigma}
\]
But $\sigma$ is a field endomorphism, so $\tdeg_{K^{\sigma}}L^{\sigma}=\tdeg_{K}L$;
the displayed inequalities must therefore be equalities, so $L^{\sigma}$
and $K$ are algebraically disjoint over $K^{\sigma}$.
\end{proof}

Here is the key statement:
\begin{prop}
\label{prop-on-simple-roots} Let $\nicefrac{L}{K}$ be a transformally
algebraic extension of difference fields with $K$ a model of $\FF$ and $L$
closed under twists. Then the following are equivalent: 
\begin{enumerate}
\item The difference field $L$ is transformally separably algebraic over
$K$.
\item Every element of $L$ is a simple root of a twisted difference polynomials
with coefficients in $K$. 
\item The difference field $L$ is generated as a twisted difference field
by simple roots of twisted difference polynomials with coefficients
in $K$.
\item The difference field $L$ is a directed union of difference fields
of the form $K\left(a^{\mathbf{N}\left[\frac{\sigma}{p^{\infty}}\right]}\right)$,
where $a$ is a finite tuple separably algebraic over $K\left(a^{\sigma\cdot\mathbf{N}\left[\frac{\sigma}{p^{\infty}}\right]}\right)$. 
\end{enumerate}
\end{prop}

\begin{proof}
$(1)\implies(2)$ Fix an element $a\in L$; we must show that it is
a root of a simple twisted difference polynomial over $K$. Since
$L$ is transformally algebraic over $K$, there is a finite nonempty
set $I\subseteq\mathbf{N}\left[\frac{\sigma}{p^{\infty}}\right]$
and nonzero elements $c_{\nu}\in K$ such that $\sum c_{\nu}a^{\nu}=0$;
indeed this is precisely what it means for $L$ to be transformally
algebraic over $K$. If $fx=\sum c_{\nu}x^{\nu}$ then $f'$ is nonzero
precisely in the event that some $\nu\in I$ has constant term prime
to $p$; let us show how to arrange this.

First assume that every element $\nu\in I$ has constant term zero.
Then $\sum c_{\nu}^{\frac{1}{\sigma}}a^{\frac{\nu}{\sigma}}=0$, so
the elements $a^{\frac{\nu}{\sigma}}$ are linearly dependent over
$K^{\sigma^{-1}}$; since $L$ is transformally separable over $K$,
they must have been linearly dependent over $K$ already. Thus inductively
we can assume that some $\nu\in I$ has a nonzero constant term. Now
assume that the constant term of every $\nu\in I$ is divisible by
$p$; then the elements $a^{\frac{\nu}{p}}$ are linearly dependent
over $K^{p^{-1}}$. The latter is contained in $K^{\sigma^{-1}}$
as $K$ is a model of $\FF$, and by assumption $L$ is transformally separable
over $K$; thus we may assume that at least one $\nu\in I$ has constant
term prime to $p$.

We have found a nonzero twisted difference polynomial $f$ such that
$f'\neq0$ and $fa=0$. If $f'a\neq0$, then we are done. Otherwise,
replace $f$ by $f'$ and argue inductively; this cannot go on forever.

$(2) \implies (3)$ This is clear.

$(3) \implies (4)$ Let us assume that $L$ is generated as a twisted
difference field by simple roots of twisted difference polynomials
over $K$. Then $L$ is the directed union of difference fields $K\left(a^{\mathbf{N}\left[\frac{\sigma}{p^{\infty}}\right]}\right)$
where $a=\left(a_{1},\ldots,a_{n}\right)$ is a finite tuple and each
$a_{i}$ is a simple root of a twisted difference polynomials over
$K$. Then each $a_{i}$ is separably algebraic over $K\left(a_{i}^{\sigma\cdot\mathbf{N}\left[\frac{\sigma}{p^{\infty}}\right]}\right)$
by Remark \ref{substitution}; it follows in particular that $a$
is separably algebraic over $K\left(a^{\sigma\cdot\mathbf{N}\left[\frac{\sigma}{p^{\infty}}\right]}\right)$,
hence the claim.

$(4)\implies(1)$ Since transformally separably algebraic extension
are of finite character, we may assume that $L=K\left(a^{\mathbf{N}\left[\frac{\sigma}{p^{\infty}}\right]}\right)$,
where $a$ is a finite tuple separably algebraic over $K\left(a^{\sigma\cdot\mathbf{N}\left[\frac{\sigma}{p^{\infty}}\right]}\right)$.
We have already seen in Lemma \ref{L:almost-simple-root-almost-tsep}
that $L$ is then almost transformally separable over $K$. By Proposition
\ref{criterion-algsep-almosttsep-for-fe}, it is enough to prove that
$\nicefrac{L}{K}$ is algebraically separable. We will show a stronger
statement, namely that $L$ is relatively perfect over $K$.

By Proposition \ref{criterion-algsep-almosttsep-for-fe} and Lemma
\ref{L:almost-simple-root-almost-tsep}, the extension $L^{\sigma\cdot p^{-\infty}}$
of $K^{\sigma\cdot p^{-\infty}}$ is transformally separable. Let
\[E=K\tensor_{K^{\sigma\cdot p^{-\infty}}}L^{\sigma\cdot p^{-\infty}}\]
be the displayed difference field. The field extension $L^{\sigma\cdot p^{-\infty}}$
of $K^{\sigma\cdot p^{-\infty}}$is relatively perfect; by descent,
the field $E$ is relatively perfect over $K$. Since $L$ is separably
algebraic over $E$, it follows from elementary properties of relatively
perfect extensions that $L$ is relatively perfect over $K$, which
is what we wanted.
\end{proof}

\subsection{Transformal $p$-Independence, Transformal $p$-Bases, and Transformal
Separable Generation}

The purpose of this subsection is to generalize the notion of separable generation to the transformal setting. 

 \subsubsection{\label{subsec:Relative-and-Absolute}Relative and Absolute $p$-Bases}
Let $\nicefrac{L}{K}$ be a separable field extension. Recall that
the sequence $\left(x_{i}\right)$ of elements of $L$ is said to
be $p$-\emph{independent} over $K$ in $L$ if the polynomials
$X^{p}-x_{i}^{p}$ are irreducible over $K\otimes_{K^{p}}L^{p}$,
with linearly disjoint splitting fields. This is equivalent to the
requirement that the sequence of monomials $\prod x_{i}^{m_{i}}$
for $0\leq m_{i}\leq p-1$ is linearly independent over $K\otimes_{K^{p}}L^{p}$,
as these monomials span the splitting field. Such a sequence is automatically
algebraically independent over $K$, and $p$-independence enjoys
the Steinitz exchange property.

By a \emph{relative $p$-basis} of $L$ over $K$ we mean a maximal
sequence $\left(x_{i}\right)$ of elements of $L$ which is $p$-independent
over $K$. So if $\left(x_{i}\right)$ is $p$-independent over $K$
in $L$, then $L$ is separable over $K\left(x_{i}\right)$, and if
$\left(x_{i}\right)$ is a relative $p$-basis, then the perfect hull
of $L$ splits over $K\left(x_{i}\right)$. There is also an absolute
version, namely by an (absolute) \emph{$p$-basis} of $L$ we mean
a relative $p$-basis over the prime field, or any other perfect subfield.
Every set of generators of $L$ over a perfect subfield contains an
absolute $p$-basis.

\subsubsection{Transformal Separable Generation}

Let $\nicefrac{L}{K}$ be a field extension. Recall that $L$ is \emph{separably
generated} over $K$ if there is a transcendence basis $x=\left(x_{i}\right)_{i\in I}$
of $L$ over $K$ with $L$ separably algebraic over $K\left(x\right)$;
we then say that $x$ is a \emph{separating transcendence basis} of $L$ over $K$.
If $L$ is finitely generated over $K$ then it is separable if and
only if it is separably generated, in which case a transcendence basis
is separating if and only if it is a relative $p$-basis as in \ref{subsec:Relative-and-Absolute}.
\begin{defn}
Let $\nicefrac{L}{K}$ be an extension of models of $\FF$. Then $L$
is \emph{transformally separably generated} over $K$ if there is
a transformal transcendence basis $x=\left(x_{i}\right)_{i\in I}$
of $L$ over $K$ such that $L$ is transformally separably algebraic
over $K\left(x^{\mathbf{N}\left[\frac{\sigma}{p^{\infty}}\right]}\right)$.
If $x$ can be chosen finite then $L$ is \emph{finitely transformally
separably generated} over $K$. The transformal transcendence basis
$x$ is then said to be a \emph{separating transformal transcendence
basis} of $L$ over $K$.
\end{defn}

\begin{rem}
\label{rem:obvious-that-tsepgen-is-tsep}Let $\nicefrac{L}{K}$ be
an extension of models of $\FF$ which is transformally separably
generated. Then $L$ is transformally separable over $K$. Indeed,
purely transformally transcendental extensions are transformally separable,
and the composition of transformally separable extensions is again
one. 
\end{rem}

\begin{rem}
If $L$ is transformally separable over $K$, then it is not in general
transformally separably generated over $K$. For example, let $K=\mathbf{Q}$
and $L=K\left(x^{\mathbf{N}\left[\sigma^{\pm1}\right]}\right)$ for
$x$ an element transformally transcendental over $K$.
\end{rem}
\begin{lem}\label{lem:lemma-for-tsep-gen}
Let $\nicefrac{L}{K}$ be a transformally separable extension with $K$ a model of $\FF$ and $L$ closed under twists. Let us set $E=K\tensor_{K^{\sigma\cdot p^{-\infty}}}L^{\sigma\cdot p^{-\infty}}$.
\begin{enumerate}
    \item The field $L$ is an algebraically separable field extension
of $E$.
\item Let us assume that $L$ is a model of $\FF$ that is finitely generated over $K$ as a model
of $\FF$, i.e, there is a finite tuple $a$ such that $L$ is the
smallest model of $\FF$ over $K\left(a^{\mathbf{N}\left[\sigma\right]}\right)$;
then $L$ is separably generated over $E$.
\end{enumerate}
\end{lem}

\begin{proof}
(1) We must show that $L$ is algebraically separable over $E$. The
field extension $L^{\sigma\cdot p^{-\infty}}$ of $K^{\sigma\cdot p^{-\infty}}$
is relatively perfect by definition.
By descent for relatively perfect extensions, the field $E$ is
relatively perfect over $K$. Since $L$ is algebraically separable
over $K$, it follows from Fact \ref{fact:rel-perf-transitive-in-towers}
that $L$ is algebraically separable over $E$. 

(2) Fix a finite tuple $a\in L$ which generates $L$ as a model of
$\FF$ over $K$. Let us set $L_{0}=K\left(a^{\mathbf{N}\left[\frac{\sigma}{p^{\infty}}\right]}\right)$
and $E_{0}=K\tensor_{K^{\sigma\cdot p^{-\infty}}}L_{0}^{\sigma\cdot p^{-\infty}}$.
Then by (1), the field $L_{0}$ is algebraically separable over $E_{0}$.
It is moreover finitely generated over $E_{0}$ as a field, namely
by the tuple $a$, hence separably generated over $E_{0}$. Now the
field $L$ is separably algebraic over $L_{0}$, namely it is the
separable algebraic closure of $L_{0}$ in its inversive hull (Lemma
\ref{fe-hull}). Moreover, the field $E$ is separably algebraic over
$E_{0}$. Thus a separating transcendence basis of $L_{0}$ over $E_{0}$
is also a separating transcendence basis of $L$ over $E$, as wanted.
\end{proof}
\begin{prop}
\label{prop:tsep-gen}Let $\nicefrac{L}{K}$ be an extension of models
of $\FF$. Then $L$ is transformally separable over $K$ if and only
if it is the directed union of models of $\FF$ which are finitely
transformally separably generated over $K$.
\end{prop}

\begin{rem}
This should be compared with \cite[Lemma 3.3]{ChHr-Dif}.
\end{rem}

\begin{proof}
The easy direction follows from Remark \ref{rem:obvious-that-tsepgen-is-tsep}.
For the converse, fix a finite tuple $a\in L$ and let $L_{0}$ be
the smallest model of $\FF$ over $K$ inside $L$ containing $a$;
we will show that $L_{0}$ is transformally separably generated over
$K$. Relabeling, we may assume that $L=L_{0}$. Using the notation
of Lemma \ref{lem:lemma-for-tsep-gen}, the field $L$ is separably
generated over $E$; after possibly enlarging $a$ within $L$, we can fix a subtuple $b\subseteq a$ which is a separating
transcendence basis for $L$ over $E=K\tensor_{K^{\sigma\cdot p^{-\infty}}}L^{\sigma\cdot p^{-\infty}}$.  

We claim that $b$ is a separating transformal transcendence basis
for $L$ over $K$. For this, we must check that $b$ is transformally
independent over $K$ and that $L$ is transformally separably algebraic
over $M=K\left(b^{\mathbf{N}\left[\frac{\sigma}{p^{\infty}}\right]}\right)$.

First we verify that the tuple $b$ is transformally independent over $K$. By construction the tuple $b$ is algebraically independent over $E$; it follows from the definition of transformal independence that it is transformally independent over $K$ in $L$; thus $M$ is a model of $\FF$, using Lemma \ref{lem:transformally-transcendental-FE}. 

Next we verify that $b$ is a separating transformal transcendence basis of $L$ over $K$, namely that $L$ is transformally separably algebraic over $M$. Since $b$ is an (algebraic) separating transcendence basis for $L$ over $E$, we again see formally that every element of $L$ is simple root of a twisted difference polynomial with coefficients in $M$; by Proposition
\ref{prop-on-simple-roots}, the field $L$ is transformally separably algebraic over $M$.
\end{proof}

\subsection{Summary}
For ease of later reference, we collect here a summary of the main results of the previous section.

\begin{thm}\label{thm:main-thm-on-tsep} 
\begin{enumerate}
\item (The Relative Transformal Separable Algebraic Closure) Let $\nicefrac{E}{F}$
be an embedding of models of $\FF$. Then there is a model $\widetilde{F}$
of $\FF$ with $F\subseteq\widetilde{F}\subseteq E$ transformally
separably algebraic over $F$ characterized uniquely by the following
property: whenever $L$ is a model of $\FF$ transformally separably
algebraic over $F$, then every embedding of $L$ in $E$ over $F$
lands inside $\widetilde{F}$. We say that $\widetilde{F}$ is the
\textbf{relative transformal separable algebraic closure }of $F$
in $E$.
\item (Transformally Separably Algebraic Extensions and Simple Roots) Let
$\nicefrac{E}{F}$ be an embedding of models of $\FF$. Then the set
of elements of $E$ which are simple roots of twisted difference polynomials
over $F$ is a difference field, which is a model of $\FF$. In fact,
it is precisely the relative transformal separable algebraic closure
of $F$ in $E$.
\item (Transitivity) Let $F\subseteq\widetilde{F}\subseteq E$ be a tower
of models of $\FF$ where $\widetilde{F}$ is the relative transformal
separable algebraic closure of $F$ in $E$; then $E$ is transformally
separable over $\widetilde{F}$ if and only if it is transformally
separable over $F$.
\item (Descent over the Inversive Hull) Let $F$ be a model of $\FF$ and
let $F^{\sigma^{-\infty}}$ be its inversive hull. Then there is a
canonical equivalence of categories between models of $\FF$ transformally
separably algebraic over $F$ and perfect, inversive difference fields
transformally algebraic over $F^{\sigma^{-\infty}}$.
\item (Transformal Separable Generation) Let $\nicefrac{E}{F}$ be a transformally
separable extension of models of $\FF$, and assume that $E$ is finitely
generated over $F$ as a model of $\FF$. Then $E$ is transformally
separably generated over $F$. To wit, there is a transformal transcendence
basis $x\in E$ of $E$ over $F$ such that $E$ is transformally
separably algebraic over $F\left(x^{\mathbf{N}\left[\frac{\sigma}{p^{\infty}}\right]}\right)$.
\item (The Amalgamation Property and Closure under Tensor Products) Let
$A$ be a model of $\FF$ which is separably algebraically closed
and let $A_{1},A_{2}$ be models of $\FF$ transformally separable
over $A$. Then $A_{1}$ and $A_{2}$ are linearly disjoint over $A$,
and the difference field $A_{1}\tensor_{A}A_{2}$ is a model of $\FF$.
\end{enumerate}
\end{thm}

\begin{proof}
The truth of (1) and (2) follows immediately from Proposition \ref{prop-on-simple-roots}.
Namely let $\mathcal{X}$ be the set of elements of $E$ which are
simple roots of twisted difference polynomials over $F$. The smallest
model $\widetilde{F}$ of $\FF$ inside $E$ containing $\mathcal{X}$
is transformally separably algebraic over $F$; on the other hand,
all elements of $\widetilde{F}$ are simple roots over $F$, so $\mathcal{X}=\widetilde{F}$. 

The truth of (3) follows from Proposition \ref{P:rel-tsep-alg-c}.
In (4), the equivalence is given by associating to a model $E$ of
$\FF$ transformally separably algebraic over $F$ its inversive hull
$E^{\sigma^{-\infty}}$. Indeed, the inversive hull of a model of
$\FF$ is again one, and by Proposition \ref{tsep-alg-of-fe-preserves-ershov},
all models of $\FF$ transformally algebraic over $F^{\sigma^{-\infty}}$
are inversive. We can recover $E$ inside $E^{\sigma^{-\infty}}$
as the relative transformal separable algebraic closure of $F$ inside,
since $E^{\sigma^{-\infty}}$ is by construction purely transformally
inseparably algebraic over $E$. The proof of (5) is given in Proposition
\ref{prop:tsep-gen}, and finally (6) is Proposition \ref{prop:amalgamation-for-FE}.
\end{proof}

\begin{defn}
    Let $\nicefrac{M}{K}$ be an extension of models of $\FF$. We say that $M$ is \emph{relatively inversive} over $K$ if it is transformally separable over $K$ and whenever $L$ is a model of $\FF$ with $K \subseteq M \subseteq L$ and $L$ is transformally separable over $K$, then $L$ is also transformally separable over $M$.
\end{defn}

\begin{rem}
Let $\nicefrac{M}{K}$ be a relatively inversive extension of models of $\FF$.
    \begin{enumerate}
        \item If $K$ is an inversive model of $\FF$, then $M$ must likewise be inversive (in the definition, consider the inversive hull of $M$).
        \item The field $M$ is relatively perfect over $K$.
        \item Transformally separably algebraic extensions of models of $\FF$ are relatively inversive.
    \end{enumerate}
\end{rem}

\subsection{\label{s:trans lambda functions}The Transformal $\lambda$-functions}

Let $K$ be a difference field. We now define the \emph{transformal
$\lambda$-functions} of $K$. Let $n\in\mathbf{N}$ be fixed for
the moment; we will define the function $\lambda=\lambda_{n}$. This
function takes as an input an $n+1$-tuple of elements of $K$ and
returns an $n$-tuple of elements of $K$. To define it,
fix a tuple $\left(a_{1},\ldots,a_{n},b\right)\in K^{n}\times K$.
If the $a_{1},\ldots,a_{n}$ are linearly dependent over $K^{\sigma}$,
or else if $b$ is not a $K^{\sigma}$-linear combination of the $a_{i}$,
then $\lambda\left(a_{1},\ldots,a_{n},b\right)=0$. Otherwise, the
coordinates of $\lambda\left(a_{1},\ldots,a_{n},b\right)$ are the
unique elements $c_{1}, \ldots ,c_{n} \in K$ satisfying $\sum_{i=1}^{n}c_{i}^{\sigma}a_{i}=b$.
Note that the transformal $\lambda$-functions of $K$ are definable
in the pure difference field language.

Let $F\subseteq K$ be an extension of difference fields. Then $K$
is transformally separable over $F$ if and only if $F$ is closed
under the transformal $\lambda$-functions of $K$. For suppose that
$F$ and $K^{\sigma}$ are linearly disjoint over $F^{\sigma}$. Fix
$\left(a_{1},\ldots,a_{n},b\right)\in F^{n}\times F$ with the $a_{1},\ldots,a_{n}$
linearly independent over $K^{\sigma}$ and $b$ in the $K^{\sigma}$-linear
span of the $a_{i}$. Then the tuple $\left(a_{1},\ldots,a_{n},b\right)$
is linearly dependent over $K^{\sigma}$; by linear disjointness,
it must be linearly dependent over $F^{\sigma}$ already, whence $\lambda\left(a_{1},\ldots,a_{n},b\right)$
has all coordinates in $F$. For the converse, assume $F$
and $K^{\sigma}$ are not linearly disjoint over $F^{\sigma}$. Let
$a_{1},\ldots,a_{n},b$ be a tuple of elements of $F$ linearly dependent
over $K^{\sigma}$ but linearly independent over $F^{\sigma}$. We
may assume that $n$ is chosen minimal with respect to this property.
Then one of the coordinates of $\lambda\left(a_{1},\ldots,a_{n},b\right)$
is not in $F$, hence $F$ is not closed under the transformal
$\lambda$-functions of $K$.

\subsubsection{}

Let $E$ be a model of $\FF$. Recall that the \emph{Frobenius twists}
of $E$ are the difference fields $\left(E,\sigma\circ\phi^{m}\right)$
for $m\in\mathbf{Z}$ where $\phi x=x^{p}$ is the Frobenius endomorphism
of $E$.
\begin{defn}\label{def:tsep-alg-closed-trivially}
Let $E$ be a model of $\FF$.
We say that $E$ is a model of $\SCFE$ if, in addition, it obeys the following two conditions:
\begin{enumerate}
    \item The field $E$ is separably algebraically closed as an abstract field.
    \item The difference field $E$ and all of its Frobenius twists obey
the following property. Let $V$ be an absolutely irreducible affine
algebraic variety over $E$. Let $C\subseteq V\times V^{\sigma}$ be
a locally closed absolutely irreducible subvariety and assume that the first projection is dominant and the second projection $C\to V^{\sigma}$ is finite \'etale; then there
is a point $a\in V\left(E\right)$ such that $\left(a,a^{\sigma}\right)\in C\left(E\right)$.
\end{enumerate}
\end{defn}

\begin{rem}
The axiomatization of Definition \ref{def:tsep-alg-closed-trivially} is slightly different from the axiomatization of \cite{SCFEe} but as their proof shows it is sufficient to handle the case where the second projection is finite. Note that by Proposition \ref{prop-on-simple-roots}, if $E$ and $a$ are as in Definition \ref{def:tsep-alg-closed-trivially} then the difference field $E(a^\Ns)$ is transformally separably algebraic over $E$, so models of $\SCFE$ are precisely the models of $\FF$ which are existentially closed for transformally separably algebraic extensions.
\end{rem}

\begin{prop}\label{P:char of SCFE}
Let $E$ be a model of $\FF$.  Then 
\begin{enumerate}
    \item The difference field $E$ is a model of $\SCFE$ if and only if the inversive hull of $E$ is a model of $\ACFA$. Thus if $\nicefrac{L}{E}$ is purely transformally inseparably algebraic, then $L$ is a model of $\SCFE$ if and only if $E$ is so.
    \item If $E$ is relatively transformally separably algebraically closed in a model of of $\SCFE$ then it is a model of $\SCFE$.
\end{enumerate}
\end{prop}
\begin{proof}
(1) The forward direction is \cite[Proposition 5.2]{SCFEe}; we leave the other direction to the reader.

(2) We may replace $E$ and $L$ by their inversive hulls, using (1); so $E$ is transformally algebraically closed in a model of $\ACFA$ hence itself a model of $\ACFA$.
\end{proof}

\begin{prop}
\label{P:crieria-for-tsep-and-almost-tsep}
Let $\nicefrac{K}{F}$ be an extension of models of $\FF$.
\begin{enumerate}
    \item Let us assume that $K$ is transformally separable over $F$ and
fix an element $x\in K$ not in $F\otimes_{F^{p}}K^{p}$; then $x$
is transformally transcendental over $F$ and $K$ is transformally
separable over $F\left(x^{\mathbf{N}\left[\frac{\sigma}{p^{\infty}}\right]}\right)$.
\item Let us assume that $K$ is almost transformally separable over
$F$ and fix an element $x\in K$ algebraically transcendental over
$F\otimes_{F^{\sigma\cdot p^{-\infty}}}K^{\sigma\cdot p^{-\infty}}$;
then $x$ is transformally transcendental over $F$ and $K$ is almost
transformally separable over $F\left(x^{\mathbf{N}\left[\frac{\sigma}{p^{\infty}}\right]}\right)$.
\end{enumerate}
\end{prop}
\begin{proof}

Let us set $E = F\left(x^{\mathbf{N}\left[\frac{\sigma}{p^{\infty}}\right]}\right)$.

(1) In fact it will be sufficient to assume that $K$ is merely closed under twists. The requirement that $x$ is not in $F \otimes_{F^p} K^p$ means that the sequence $1, x, 
\ldots, x^{p-1}$ is linearly independent over $F \otimes_{F^p} K^p$. This condition is invariant under base extension; using descent for the various properties considered here, we may thus assume that $F$ is inversive.

We want to show that $x$ is transformally transcendental over $F$. The element $x$ is not a simple root of any polynomial with coefficients in $K^p$. Since $F$ is inversive, the element $x$ is not a simple root of any twisted difference polynomial over $F$; in other words the element $x$ is transformally transcendental over $F$.

Now we want to show that $K$ is transformally separable over $E$. Since $E$ is a model of $\FF$ it is enough to show that $K$ is algebraically separable and almost transformally separable over $E$ (Proposition \ref{criterion-algsep-almosttsep-for-fe}). 

The fact that $K$ is algebraically separable over $E$ follows immediately from our choice of $x$; the extension $E$ of $E^p$ is normal of degree $p$, and since $x$ is not in $K^p$ the minimal polynomial of $x$ remains irreducible over $K^p$, thus $E$ and $K^p$ are linearly disjoint over $E^p$. Furthermore, since $x$ is transformally transcendental over $F$, it is algebraically transcendental over $K^{\sigma}$; so $E$ and $K^{\sigma}$ are algebraically disjoint over $E^{\sigma}$. It follows that $K$ is almost transformally separable over $E$, and we are done.

(2) We are free to replace $F$ and $K$ by their perfect hulls; now $K$ is transformally separable over $F$ and the proof is similar to (1).
\end{proof}

\begin{lem} \label{lem:existence-of-generics}
Let $E$ be a sufficiently saturated model of
$\FF$ which fails to be inversive. Let $F\subseteq E$ be a small
model of $\FF$ contained in $E$ and with $E$ almost transformally
separable over $F$. Then we can find an element $a\in E$ transcendental
over $F\tensor_{F^{\sigma\cdot p^{-\infty}}}E^{\sigma\cdot p^{-\infty}}$.
Moreover, assume $E$ is of infinite degree of imperfection and that
$F$ is transformally separable over $E$ outright; then we can find
$a\in E$ not in $F\otimes_{F^{p}}E^{p}$.
\end{lem}

\begin{proof}
If $E$ is of infinite degree of imperfection then by saturation $E$
is of large dimension as a vector space over $E^{p}$; as $F$ is
small, the inclusion $F\otimes_{F^{p}}E^{p}\subseteq E$ must be strict.
Now we claim that in any case we can find an element $a\in E$ not
algebraic over $F\tensor_{F^{\sigma\cdot p^{-\infty}}}E^{\sigma\cdot p^{-\infty}}$.
Since $E$ fails to be inversive, there is an element $a\in E$ not
in $E^{\sigma}$. If $E$ is perfect, then as it is a model of $\FF$,
the field $E^{\sigma}$ is algebraically closed in $E$, so $a$ is
transcendental over $E$. In this situation for $m$ fixed the sequence
$a,a^{q},a^{q^{2}},\ldots,a^{q^{m}}$ approximates the type of a sequence
of $m$ elements of $E$ algebraically independent over $E^{\sigma}$
and we conclude by compactness and saturation. Similarly if $E$ is
imperfect by taking $a$ to be an element of $E^{\sigma\cdot p^{-\infty}}$
of large degree over $E^{\sigma}$.
\end{proof}

\section{Transformal Valued Fields}

\begin{defn}
Let $\wOGE$ be the the theory of ordered $\Zsphul$-modules; that is ordered abelian groups $\Gamma$ equipped with an order preserving endomorphism $\sigma \colon \Gamma \to \Gamma$ of abelian groups such that $n\alpha < \sigma\alpha$ for all $0 < \alpha \in \Gamma$ and all $n \in \mathbf{N}$ with the further property that $\sigma\alpha$ has all $p$-th power roots in $\Gamma$.
\end{defn}

\begin{fact}\cite[Section 2]{dor-hrushovskiVFA}
The theory $\wOGE$ admits a model companion $\TwOGA$ axiomatized by the requirement that $\Gamma$ is nonzero and divisible as a module over $\Zsphul$ (in case we say that $\Gamma$ is \textit{transformally divisible}). The theory $\TwOGA$ is complete and $o$-minimal. If $\Gamma$ is nonzero then there is a model of $\TwOGA$ over $\Gamma$ which embeds uniquely over $\Gamma$ in any other, namely the transformal divisible hull $\Gamma \otimes \Qs$.
\end{fact}

\begin{defn}
Let $\Gamma$ be a model of $\wOGE$. We say that $\Gamma$ is \emph{tamely transformally divisible} if we have $\nu \Gamma = \Gamma$ for all $\nu \in \Zsphul$ with a nonzero constant term. Equivalently, the module $\Gamma$ is tamely transformally divisible if and only if its inversive hull is transformally divisible.
\end{defn}
\begin{defn}\label{def:transformal-valued-fields}
By a \emph{transformal valued field} we mean a valued field $F$
equipped with an endomorphism $\sigma\colon F\to F$ of fields with
the property that $\sigma^{-1}\left(\mathcal{O}\right)=\mathcal{O}$
and $\sigma^{-1}\left(\mathcal{M}\right)=\mathcal{M}$. By an \emph{embedding}
of transformal valued fields we mean an embedding of fields which
is simultaneously an embedding of the underlying difference fields
and of the underlying valued fields. The transformal valued field
$F$ is said to be $\omega$-\emph{increasing} if the induced action
of $\sigma$ on the value group is $\omega$-increasing; to wit,
for all $0<\alpha\in\Gamma$ and $n\in\mathbf{N}$ we have the inequality
$n\alpha<\sigma\alpha$. We say that $K$ is a \emph{Frobenius transformal valued field} if it is of positive characteristic $p > 0$ and $\sigma = p^n$ is a power of the Frobenius endomorphism of $K$.
\end{defn}

The class of $\omega$-increasing transformal valued fields which
are perfect and inversive was studied in \cite{dor-hrushovskiVFA}. Here we seek to omit these assumptions.

\begin{rem}
In Definition \ref{def:transformal-valued-fields}, assume that $F$ is inversive; then the compatibility condition on $\mathcal{M}$ is automatic. However outside the inversive realm the condition cannot be dispensed with.

More precisely, let $\mathcal{O}$ be a valuation ring and let $\sigma$ be an \emph{automorphism} of $\mathcal{O}$ in the category of rings. Then $\sigma$ is automatically a local homomorphism, i.e, it carries non-units to non-units, or equivalently we have $\sigma^{-1}\left(\mathcal{M}\right) = \mathcal{M}$.

However, the condition that $\sigma$ is an automorphism is necessary for this to hold, even when $\sigma$ is injective.

For example, let $\mathcal{O}$ be a valuation ring which is abstractly isomorphic to its own localization $\widetilde{\mathcal{O}}$ at a prime ideal $\mathfrak{p}$. Thus after relabeling $\widetilde{\mathcal{O}}$ as $\mathcal{O}$ we obtain an endomorphism $\sigma$ of $\mathcal{O}$ which carries elements of positive but "small" valuation (i.e lying outside $\mathfrak{p}$) to elements of valuation zero; the field of fractions $K$ when equipped with this endomorphism is \emph{not} a transformal valued field according to Definition \ref{def:transformal-valued-fields}. For instance, the residue field does not inherit the action of $\sigma$.
\end{rem}

\subsection{The Theory $\wVFE$}\label{ss:wvfe}

We will be interested in the following class of $\omega$-increasing
transformal valued fields:
\begin{defn}
\begin{enumerate}
    \item Let $\wVFE$ be the first order theory of $\omega$-increasing
transformal valued fields $F$ with underlying difference field a
model of $\FF$. The language is an expansion of the language of the language of valued fields by a unary function symbol for the action of $\sigma$, so an embedding is an embedding of transformal valued fields as in \ref{def:transformal-valued-fields}, i.e, an embedding need not be transformally separable.
\item The definition of $\wVFE_{\lambda}$ is as in \ref{D:FE-lambda}; in this language an embedding is an embedding of transformal valued fields which is required to be transformally separable. 
\item For $e \in \mathbf{N}_{\infty}$,  we let $\VFE_{\leq e}$ be the theory of models of $\wVFE$ of degree of imperfection at most $e$
\item Finally, we write $\wVFA$ for the theory models of $\wVFE$ which are perfect and inversive.
\end{enumerate}
\end{defn}


In \cite{dor-hrushovskiVFA} the theory $\TVFA$ was studied in detail; we summarize below the main results that we will need.

\begin{fact}\label{fact:wvfa}
\begin{enumerate}
    \item The class of models of $\wVFA$ admits a model companion $\TVFA$.
    \item In $\TVFA$ the residue field and the value group are stably embedded and fully orthogonal; the induced structure is that of a model of $\ACFA$ and a pure model of $\TwOGA$, respectively.
    \item Let $K$ be a model of $\wVFA$ which is algebraically Henselian and assume that $K$ has no nontrivial finite $\sigma$-invariant Galois extensions, that is, no difference field extensions which are finite and Galois as abstract field extensions. Then the theory of models of $\TVFA$ over $K$ is complete and in particular $K$ admits a unique algebraic closure in the category of models of $\wVFA$, up to a (generally non unique) isomorphism. The converse holds if the various powers $\sigma^n$ of $\sigma$ for $0 < n \in \mathbf{N}$ are considered (that is the lack of such extensions is equivalent to the uniqueness of the lifts of the maps $\sigma^n$ for all $n$).
    \item The theory $\TVFA$ is precisely the asymptotic theory of algebraically closed and nontrivially valued Frobenius transformal valued fields, that is, the set of sentences true over all algebraically closed and nontrivially valued Frobenius transformal valued fields outside a finite set of exceptional prime powers.
    \item Let $K$ be a model of $\wVFA$ which is algebraically closed and transformally Henselian (see Definition \ref{transformally-henselian-def}). Then the category of models of $\wVFA$ admits linearly disjoint amalgamation over $K$, that is if $L$ and $M$ are models of $\wVFA$ over $K$ then there is a model $N$ of $\wVFA$ over $K$ in which $L$ and $M$ are jointly embedded over $K$ and linearly disjoint over $K$. Moreover, the residue fields in the amalgamation can be taken linearly disjoint, and the transformal divisible hulls of the value groups rationally independent.
    \item Let $K$ be an algebraically closed model of $\wVFA$ and let $L=K\left(x^\Ns\right)$. Assume the residue class of $x$ is transformally transcendental over $k$ or else the valuation of $x$ is not in $\Gamma \otimes \Qs$. Then the inversive, algebraically Henselian hull of $L$ has no nontrivial finite $\sigma$-invariant Galois extension. Moreover the quantifier free type of $x$ is determined in the first case by the fact that the residue class is transformally transcendental, and in the second case by the cut of $vx$ over $\Gamma \otimes \Qs$.
    \item Let $K$ be a model of $\TVFA$. Then $K$ is transformally algebraically maximal, that is, if $L$ is a model of $\wVFA$ transformally algebraic over $K$ with the same value group and residue field, then $K = L$. Furthermore, there is an elementary immediate extension of $K$ which is spherically complete.
    \item The theory $\TVFA$ admits the amalgamation property over an algebraically closed base. Thus if $K$ is a model of $\VFA$ then completions of the theory $\TVFA_{K}$ of models of $\TVFA$ over $K$ are determined by an isomorphism type of an extension of the transformal valued field structure to an algebraic closure of $K$ and $\TVFA$ eliminates quantifiers in the language of transformal valued fields provided that formulas ranging over the field theoretic algebraic closure are permitted.
\end{enumerate}

\begin{proof}

All references are from \cite{dor-hrushovskiVFA}.
    (1) See Theorem 9.20,
    (2) See Theorem 9.14,
(3) See Proposition 4.29,
(4) See Theorem 9.10,
(5) See Theorem 9.1,
(6) See Proposition 6.16,
(7) See Proposition 9.9.
\end{proof}

\end{fact}

\begin{rem}
\label{rem-henselian-hull-not-wVFE}
\begin{enumerate}
    \item Let $K$ be a model of $\wVFE$. Then the residue field $k$ of $K$ need not be a model of $\FF$. It follows from Lemma \ref{lem:hens-residue-is-FE} below that $\wVFE$ is not in general closed under Henselian hulls.
    Indeed, Example 9.1 of \cite{haskell2008stable} shows that if $F$ is a countable field, then every separably algebraic extension $\widetilde{F}$ of $F$ occurs as the residue field of $F\left(x,y\right)$ where $x,y$ are elements of $\widetilde{F} \left(\left(\mathbf{Z}\right)\right)$ algebraically independent over $F$. One can easily modify this example to show that if $F$ is a countable (inversive) difference field of characteristic zero and $\widetilde{F}$ is a countable transformally algebraic extension of $F$, then $\widetilde{F}$ is the residue field of a difference subfield $L$ of $\widetilde{F}\left(\left(\mathbf{Z}\left[\sigma\right]\right)\right)$, where $L$ is generated by a pair of transformally independent elements over $F$. By \ref{lem:transformally-transcendental-FE}, the field $L$ is a model of $\VFE$; since $\widetilde{F}$ is essentially arbitrary it need not be a model of $\FF$.
    \item Let $K$ be a non-principal ultraproduct of Frobenius transformal valued fields. Then $K$ is a model of $\wVFE$. For $K$ of this form the pathology of the previous item cannot occur since $K$ is stationary over $K^{\sigma}$ in the sense of $\ACVF$, that is, for all separably algebraic valued field extensions $L$ of $K^{\sigma}$, the valuation lifts uniquely to $K \otimes_{K^\sigma} L$.
    \end{enumerate}
\end{rem}

We do have however:
\begin{lem}
\label{lem:hens-residue-is-FE}Let $K$ be a model of $\VFE$
which is algebraically Henselian, then the residue field $k$ of $K$
is a model of $\FF$.
\end{lem}

\begin{proof}
It is clear that $k$ is closed under twists; for this we do not need
$K$ to be algebraically Henselian. Furthermore, the embedding $K^{\sigma}\hookrightarrow K$
is a primary extension of Henselian valued fields; the induced
extension $k^{\sigma}\hookrightarrow k$ of residue fields must therefore
also be primary, by Hensel lifting.
\end{proof}

\begin{defn}\label{defn:strict-amalgamation-basis}
Let $F$ be a model of $\VFE$. We say that $F$ is a $\textit{strict amalgamation basis}$ if the perfect, inversive, algebraically Henselian hull of $F$ admits no nontrivial finite $\sigma$-invariant Galois extensions.
\end{defn}

\begin{rem}\label{rem:strict-amalgamation-vs-unique-extension}
    Let $F$ be a model of $\VFE$. Let us assume that $F$ is a strict amalgamation basis. Then a separable algebraic closure of $F$ in the category of models of $\VFE$ is uniquely determined up to isomorphism. Indeed, by \cite[Proposition 4.29]{dor-hrushovskiVFA},  an inversive, algebraically closed hull $E$ of $F$ is uniquely determined up to isomorphism; the relative separable algebraic closure of $F$ in $E$ is then a separably algebraically closed hull of $F$, and is likewise uniquely determined up to isomorphism. Conversely, assume that for all $0 < n \in \mathbf{N}$, a separable algebraic closure of $\left(F, \sigma^n\right)$ is uniquely determined up to isomorphism; then $F$ is a strict amalgamation basis. This is again a consequence of \cite[Proposition 4.29]{dor-hrushovskiVFA}. 
\end{rem}

 \begin{rem}
\label{rem:strict-amalgamation-base-extension} Let $F$ be a model
of $\wVFE$. Let $E$ be a model of $\wVFE$ which is primary over $F$. Fix separable algebraic closures $F^{\sep}$ and $E^{\sep}$ of $F$ and $E$, respectively, in the category of valued fields. Let us assume:

\begin{enumerate}
    \item The field $E$ is a strict amalgamation basis; i.e, for each $n$, the group $\Aut\left(\nicefrac{E^{\sep}}{E}\right)$ acts transitively on the set of lifts of $\sigma^n$ from $E$ to $E^{\sep}$.
    \item The valued field structure on the amalgamation $F^{\sep} \otimes_{F} E$ is uniquely determined.
\end{enumerate}

Then $F$ must likewise be a strict amalgamation basis. Indeed, under the assumptions, the restriction map $\Aut\left(\nicefrac{E^{\sep}}{E}\right) \to \Aut\left(\nicefrac{F^{\sep}}{F}\right)$ is well defined and surjective.

Condition (2) need not hold in general assuming only that $E$ is primary over $F$ (for example, the residue field of $E$ might be a proper separably algebraic extension of the residue field of $F$, while $E$ is still primary over $F$). However, it holds whenever $\qftp_{\widetilde \VFA}(\nicefrac{E}{F})$ admits a global quantifier free extension which does not split over $F$.
%
\end{rem}

\begin{rem}\label{rem:strict-amalgamation-basis-vs-complete-vfa-type}
    Let $F$ be a model of $\VFE$ which is separably algebraically closed and let $a$ be a tuple in an ambient extension of $F$. Let $E$ be the smallest model of $\VFE$ over $F$ containing $a$. Then the following conditions are equivalent:
    \begin{enumerate}
        \item The field $E$ is a strict amalgamation basis in the sense of Definition \ref{defn:strict-amalgamation-basis}
        \item Fix $0< n \in \mathbf{N}$. Then the type $\qftp_{\TwVFA_n}\left(\nicefrac{a, \ldots, a^{{\sigma}^{n-1}}}{F}\right)$ is complete.
    \end{enumerate}
    Here the notation $\TwVFA_{n}$ refers to the theory of models of $\TwVFA$ over $\left(F, \sigma^n\right)$. This follows from Remark \ref{rem:strict-amalgamation-vs-unique-extension} and Fact \ref{fact:wvfa}(3).
\end{rem}

\subsection{The Transformal Henselization}

The following Definition will play an important role in this work.
\begin{defn}
\label{transformally-henselian-def}Let $K$ be a model of $\VFE$.
We say that $K$ is \emph{transformally Henselian} if it obeys the
following transformal analogue of Hensel's lemma. Let $fx\in\mathcal{O}\left[x^{\mathbf{N}\left[\frac{\sigma}{p^{\infty}}\right]}\right]$
be a twisted difference polynomial and let $a\in\mathcal{O}$ be such
that $vfa>0$ and $vf'a=0$; then an element $b\in\mathcal{O}$ is
found with $fb=0$ and whose residue class coincides with that of
$a$.
\end{defn}

In the situation of Definition \ref{transformally-henselian-def},
we say that $b$ is the \emph{canonical Hensel lift} of its residue
class. Hensel lifts are unique if they exist.

\begin{rem}
Let $K$ be a model of $\VFE$. Then $K$ is transformally Henselian if and only if the Frobenius twists of $K$ all obey the transformal analogue of Hensel's lemma for ordinary difference polynomials (not necessarily twisted). In particular for models of $\VFA$, Definition \ref{transformally-henselian-def} agrees with the Definition given in \cite{dor-hrushovskiVFA}.
\end{rem}
\begin{prop}
\label{prop:rel-tsep-thens}Let $\nicefrac{L}{K}$ be an extension
of models of $\VFE$. Let us assume that $L$ is transformally
Henselian and that $K$ is transformally separably algebraically closed
in $L$; then $K$ is transformally Henselian. Moreover, the difference
field $k$ is transformally separably algebraically closed in $l$.
\end{prop}

\begin{rem}
Implicit in the statement of Proposition \ref{prop:rel-tsep-thens}
is the assumption that the residue field of a transformally Henselian
model of $\VFE$ is a model of $\FF$. Since a transformally
Henselian model is algebraically Henselian, Lemma \ref{lem:hens-residue-is-FE}
applies.
\end{rem}

\begin{proof}
Let $fx\in\mathcal{O}_{K}\left[x^{\mathbf{N}\left[\frac{\sigma}{p^{\infty}}\right]}\right]$
be a twisted difference polynomial and let $a\in\mathcal{O}$ be such
that $vfa>0$ and $vf'a=0$; we must find $b\in\mathcal{O}_{K}$ with
$fb=0$ and whose residue class coincides with that of $a$. Since
$L$ is transformally Henselian, such $b$ can be found in $\mathcal{O}_{L}$.
By assumption however the field $K$ is transformally separably algebraically
closed in $L$; since $b$ is a simple root of a twisted difference
polynomial with coefficients in $K$, the element $b$ lies in $K$ by Theorem \ref{thm:main-thm-on-tsep}(2).
It follows that $K$ is transformally Henselian.

We have seen that $K$ is transformally Henselian; we now show that
$k$ is transformally separably algebraically closed in $l$. Fix
an element $\alpha\in l$ which is transformally separably algebraic
over $k$; then $g\alpha=0$ and $g'\alpha\neq0$ for some twisted
difference polynomial $gx\in k\left[x^{\mathbf{N}\left[\frac{\sigma}{p^{\infty}}\right]}\right]$.
Let $fx\in\mathcal{O}_{K}\left[x^{\mathbf{N}\left[\frac{\sigma}{p^{\infty}}\right]}\right]$
be a twisted difference polynomial lifting $gx$. Since $L$ is transformally
Henselian there is an element $a\in\mathcal{O}_{L}$ with residue
class equal to $\alpha$ such that $fa=0$ and $f'a\neq0$. Since
$K$ is transformally separably algebraically closed in $L$, the
element $a$ lies in $K$, hence its residue class lies in $k$, as
wanted.
\end{proof}

\begin{lem}
\label{lem:thens-insep-invariant}Let $\nicefrac{L}{K}$ be a purely
transformally inseparably algebraic extension of models of $\VFE$.
Then $K$ and $L$ are simultaneously transformally Henselian.
\end{lem}
\begin{proof}
It is enough to prove this when $L$ is the inversive hull of $K$.
First assume that $K$ is transformally Henselian; then $L$ is the
directed union of copies of $K$ and hence transformally Henselian
as the directed union of transformally Henselian models of $\wVFE$.
The converse follows from Proposition \ref{prop:rel-tsep-thens}.
\end{proof}

\begin{defn}
Let $\nicefrac{E}{F}$ be an extension of models of $\wVFE$.
We say that $E$ is a \emph{separably algebraically closed, transformally
Henselian hull} of $F$ if the following conditions are satisfied:
\begin{enumerate}
    \item The field $E$ is separably algebraically closed and transformally
Henselian.
\item Let $F\subseteq E_{0}\subseteq E$ be a model of $\wVFE$
which is separably algebraically closed and transformally Henselian;
then $E_{0}=E$.
\end{enumerate}
\end{defn}

\begin{lem}
\label{lem:amalgamation-over-strictly}Let $F$ be a model of $\wVFE$
which is a strict amalgamation basis. Then there is up to isomorphism a
unique separably algebraically closed and transformally Henselian
hull $E$ of $F$. Moreover, the field $E$ is transformally separable
over $F$.
\end{lem}
\begin{proof}
Let us assume first that $F$ is inversive. In this situation, an algebraically closed, transformally Henselian hull of $F$ is unique up to isomorphism, using \cite[Corollary 5.4]{dor-hrushovskiVFA}. Thus $F$ admits an inversive, algebraically closed, transformally Henselian hull $E$ in the category of models of $\wVFE$ which is unique up to isomorphism. Using Lemma \ref{lem:thens-insep-invariant}, the relative transformal separable algebraic closure of $F$ in $E$ is a separably algebraically closed, transformally Henselian hull of $F$.
\end{proof}

\begin{prop}
\label{prop:lin-disjoint-amalgamation}(Amalgamation for Models of
$\VFE$) 
Let $F$ be a model of $\VFE_{\leq e}$ which is separably
algebraically closed and transformally Henselian. Let $E,K$ be models
of $\VFE_{\leq e}$ transformally separable over $F$. Then there
is a model $L$ of $\VFE_{\leq e}$ transformally separable over
$F$ in which $K$ and $E$ jointly embed over $F$. The fields $E$
and $K$ can be taken linearly disjoint over $F$ in $L$ and $L$
can be taken transformally separable over both $E$ and $K$.
\end{prop}

\begin{proof}
By Proposition \ref{prop:amalgamation-for-FE} and Remark \ref{rem:purely-insep},
the only problem is to equip the difference field $E\tensor_{F}K$
with the structure of a model of $\VFE$ subject to the evident compatibility constraints. If $F$ is inversive,
then it is algebraically closed and transformally Henselian; the result
then follows from Fact \ref{fact:wvfa}. In general, by Lemma \ref{lem:thens-insep-invariant},
the inversive hull of $F$ is algebraically closed and transformally
Henselian; since $E$ and $K$ are transformally separable over $F$,
the truth of the statement with all fields replaced by their inversive
hulls descends to the original settings, whence the proposition.
\end{proof}

\begin{prop}
\label{prop:th-rel}
Let $K$ be a model of $\wVFE$ whose residue field $k$ is a model of $\FF$. Let us assume that we are given a model $\widetilde{k}$ of $\FF$ transformally separably algebraic over $k$. Then there is a transformally Henselian model $\widetilde{K}$ of $\wVFE$ transformally separably algebraic over $K$, reproducing the embedding $\nicefrac{{\widetilde{k}}}{k}$ at the level of residue fields, and enjoying the following universal mapping property. Let $L$ be a transformally Henselian model of $\wVFE$ over $K$; then every embedding of $\widetilde{k}$ in $l$ over $k$ lifts to an embedding of $\widetilde{K}$ in $L$ over $K$, and the lifting is unique. This extension does not enlarge the value group. Furthermore, if $K$ is a strict amalgamation basis and $\widetilde{k}$ is separably algebraically closed, then $\widetilde{K}$ is a strict amalgamation basis.
\end{prop}

\begin{proof}
When $K$ is inversive, this follows from \cite[Theorem 5.2]{dor-hrushovskiVFA}. In general, the usual descent argument shows that the relative transformal separable algebraic closure of $K$ inside the promised extension of the inversive hull will do.
\end{proof}

\begin{defn}\label{def:sth}
Let $K$ be a model of $\wVFE$. 
\begin{enumerate}
\item We say that $K$ is \emph{strictly transformally Henselian} if it is transformally Henselian, and the residue field $k$ of $K$ is a model of $\SCFE$. 
\item In the situation of Proposition \ref{prop:th-rel}, assume that $\widetilde{k}$ is a model of $\SCFE$; then $\widetilde{K}$ is said to be a \emph{strict transformal Henselization} of $K$ (with respect to the embedding $\nicefrac{\widetilde{k}}{k}$).
\end{enumerate}
\end{defn}

\begin{rem}
\label{rem:strictly-hens-purely-insep}
Let $\nicefrac{L}{K}$ be a purely transformally inseparably algebraic extensions of models of $\wVFE$. Then $L$ is strictly transformally Henselian if and only if $K$ is so; this follows immediately from Proposition \ref{P:char of SCFE} and Proposition \ref{prop:rel-tsep-thens}.
\end{rem}

\subsection{The Theory $\widetilde{\VFE}$}\label{ss:Twvfe}
\begin{defn}
\label{def:tsep-alg-closed}
Let $K$ be a model of $\VFE$. We say that $K$ is a model of $\TwVFE$ if the following requirements are satisfied:
\begin{enumerate}
    \item The field $K$ is strictly transformally Henselian in the sense of Definition \ref{def:sth}.
    \item The valuation group $\Gamma$ of $K$ is nonzero and \emph{tamely transformally divisible}, i.e, for all $\nu\in\mathbf{Z}\left[\frac{\sigma}{p^{\infty}}\right]$
with nonzero constant term we have $\nu\Gamma=\Gamma$.
\item Let $\tau x\in K\left[x^{\mathbf{N}\left[\frac{\sigma}{p^{\infty}}\right]}\right]$
be a twisted additive difference operator with $\tau'\neq 0$; then
$\tau$ is onto on $K$-points.
\end{enumerate}

 For $e\in\mathbf{N}_{\infty}$ we let $\widetilde{\VFE_{e}}$
be the theory of models of $\widetilde{\VFE}$ which fail to
be inversive, of imperfection degree equal to $e$.
\end{defn}

\begin{rem}
\label{rem-on-val-gp-and-res-of-twvfe}
It will follow from Lemma \ref{lem:deep-ramification} that if $K$ is a model of $\TwVFE$ then it lies dense in its inversive hull; so the residue field is in fact a model of $\ACFA$, and the value group is a model of $\TwOGA$, using Remark \ref{rem:strictly-hens-purely-insep}.
\end{rem}

\begin{prop}\label{P:pure insep inside twvfe}
\begin{enumerate}
\item Let $K$ be a model of $\wVFE$. Then $K$ is a model of $\TwVFE$ if and only if the inversive hull of $K$ is a model of $\TVFA$. Thus if $L$ is a model of $\wVFE$ purely transformally inseparably algebraic over $K$, then $L$ is a model of $\TwVFE$ if and only if $K$ is so.
\item Let us assume that $K$ is nontrivially valued and relatively transformally separably algebraically closed in a model of $\TwVFE$; then $K$ is itself a model of $\TwVFE$.
\end{enumerate}
\end{prop}
\begin{proof}
Clear using Remark \ref{rem:strictly-hens-purely-insep} and Proposition \ref{prop-on-simple-roots}.
\end{proof}

\begin{defn}
Let $F$ be a model of $\VFE$. We say that $F$ is \emph{deeply
transformally ramified} if is nontrivially valued and lies dense in
its inversive hull.
\end{defn}

\begin{lem}
\label{lem:deep-ramification}Let $F$ be a model of $\VFE$ which is nontrivially valued.
Let us assume that for all nonzero $t\in\mathcal{M}$, the operator
$x^{\sigma}-tx$ is surjective on $F$-points. Then $F$ is deeply
transformally ramified. In particular, if $F$ is a model of $\widetilde{\VFE}$,
then $F$ is deeply transformally ramified. 
\end{lem}

\begin{proof}
By induction it is enough to prove that $F$ lies dense in $F^{\sigma^{-1}}$.
Fix an element $a\in F$; we need to show that $a^{\sigma^{-1}}$
can be approximated to arbitrary precision by elements of $F$. Let
$fx=x^{\sigma}-a$ be the displayed difference polynomial. Then $fx$
can be regarded as the limit of the difference polynomials $f_{t}\left(x\right)=x^{\sigma}-tx-a$
for $t\in\mathcal{M}$ nonzero in the sense of the valuation topology.
Namely, by assumption, we can find $b_{t}\in F$ with $f_{t}b_t=0$; since $f_t \to f$ coefficient-wise and nearby polynomials have nearby roots, as $vt\to\infty$ we will have $vfb_{t}\to\infty$, whence
the statement.
\end{proof}

\begin{prop}
\label{prop:ec-tsep-alg}Let $K$ be a saturated model of $\widetilde{\VFE}$.
Then $K$ is existentially closed for transformally separably algebraic
extensions in the following sense. Let $F$ be a strictly amalgamative
model of $\VFE$ contained in $K$ and such that $K$ is transformally
separable over $F$. Let $E$ be a small model of $\VFE$ transformally
separably algebraic over $F$. Then there is an embedding of $E$
in $K$ over $F$. Moreover, such an embedding automatically renders
$K$ transformally separable over the image.
\end{prop}

\begin{proof}
If an embedding exists, then it automatically renders $K$ transformally
separable over the image, using Theorem \ref{thm:main-thm-on-tsep}(3).
It is therefore enough to find an embedding. By compactness, this
means that if $a\in E$ is a finite tuple and $\varphi$ is a quantifier
free formula in the language of transformal valued fields with parameters
in $F$ such that $\models\varphi\left(a\right)$, then $\varphi$
has a solution in $K$.

For simplicity of notation we assume that $\varphi$ has a single
free variable. Replacing $K$ by the relative transformal separable
algebraic closure of $F$ inside, we may assume that $K$ is transformally
separably algebraic over $F$. So let $E$ be a model of $\VFE$
transformally separably algebraic over $F$ and let $a\in E$ be an
element such that $\models\varphi\left(a\right)$. Then $a$ is a
simple root of a twisted difference polynomial $fx\in F\left[x^{\mathbf{N}\left[\frac{\sigma}{p^{\infty}}\right]}\right]$;
strengthening $\varphi$ we may assume that all solutions of $\varphi$
are simple roots of $fx$. 

Since $F$ is strictly amalgamative, the
theory of models of $\widetilde{\VFA}$ over $F$ is complete (Fact \ref{fact:wvfa}); the inversive hull of $K$ is a model of $\TVFA$, 
so by model completeness, the formula $\varphi$ has a solution
over the inversive hull of $K$. But the simple roots of $f$ over $K$
and over its inversive hull coincide, using Theorem \ref{thm:main-thm-on-tsep}, thus there is a solution in $K$.
\end{proof}

\section{Genericity in a Ball}

In this section we fix a large enough saturated model $\mathbf{K}$ of $\widetilde{\VFA}$; all transformal valued fields in this section live inside $\mathbf{K}$.

\subsubsection{The Model Theoretic Settings}

We work in the language of transformal valued fields, that is, the
expansion of the language $\mathcal{L}_{\VF}$ by a unary function
symbol for the action of $\sigma$. Let $\VF$ denote the valued field
sort. By a \emph{definable set} we mean definable in the language
of transformal valued fields.

\subsubsection{Split and Definable Balls}
Let $B\subseteq\VF=\VF^{1}$ be a definable set, in one variable.
We say that $B$ is a \emph{closed ball} if $B=a+\gamma\mathcal{O}$
for some $a\in\VF$ and $\gamma\in\Gamma$. We say that $B$ is an
\emph{open ball} if $B=\VF$ or $B=a+\gamma\mathcal{M}$ for some
$\gamma\in\Gamma$ and $a\in\VF$. By a \emph{possibly degenerate
ball} we mean either a ball or a singleton. Finally, by an $\infty$-definable
ball we mean the intersection of a (small) family of balls linearly
ordered under reverse inclusion, regarded as a partial quantifier
free type.

Let $F\subseteq\mathbf{K}$ be a small model of $\VFE$ contained
in $\mathbf{K}$. Let $B\subseteq\VF$ be a ball. We say that $B$
is \emph{definable} over $F$ if it is definable with parameters
in $F$ and in the language of transformal valued fields or equivalently
if it is setwise fixed by automorphisms of $\mathbf{K}$ over $F$.
If $B$ is an $\infty$-definable ball then we say that it is $\infty$-definable
over $F$ if it is the intersection of balls definable over $F$.
\begin{rem}
Let $F$ be a model of $\VFE$ which is separably algebraically
closed and transformally Henselian. The theory of models of $\widetilde{\VFA}$
over $F$ is complete using Fact \ref{fact:wvfa}, so the notion of a ball definable over $F$
makes sense intrinsically, independently of the choice of $\mathbf{K}$. 
In order to avoid ambiguity we will only discuss definable balls over
a model of $\VFE$ which is separably algebraically closed and
transformally Henselian.
\end{rem}

Let $B\subseteq\VF$ be a ball. We say that $B$ is \emph{split}
over $F$ if it takes the form $a+\gamma\mathcal{O}$ for some $a\in F$
and $\gamma\in\Gamma_{F}$. If $B$ is split over $F$ then it is
evidently definable over $F$ but the converse is of course false.
Nevertheless we have the following:
\begin{prop}
\label{prop:over-tilde-wvfe-balls-are-triv}Let $F$ be a model of
$\widetilde{\VFE}$. Then all $F$-definable balls are split
over $F$.
\end{prop}

\begin{proof}
Let us assume first that $F$ is inversive. Then $F$ is a model of
$\widetilde{\VFA}$; by model completeness every $F$-definable
ball has a point over $F$, and in particular, all $F$-definable balls
are split over $F$.  For the general case, let $E$ be the inversive
hull of $F$. Then every element of $E$ is definable over $F$, so
there is no distinction between $F$-definable balls and $E$-definable
balls. Moreover, by Lemma \ref{lem:deep-ramification}, the field
$F$ lies dense in $E$; so a ball is split over $F$ if and only
if it is split over $E$; this reduces us to the inversive case already
dealt with, and the proof is finished.
\end{proof}

\subsubsection{Genericity in a Ball}
\begin{lem}
\label{lem:closed-balls-are-compact}Let $F$ be a model of $\VFE$ contained in $\mathbf{K}$. Let $B\subseteq\VF$
be a closed ball definable over $F$. Then $B$ enjoys the following
compactness property. Let $fx\in F\left[x^{\mathbf{N}\left[\sigma\right]}\right]$
be a difference polynomial; then the function $x\mapsto vfx$ admits
a minimum on $B$.
\end{lem}

\begin{proof}
If $F$ is increased then the conclusion descends; so assume that
$F$ is a model of $\widetilde{\VFA}$. Then the ball $B$ is
split over $F$, hence there is an affine map $\varphi$ with
coefficients in $F$ which carries $B$ to the unit ball $\mathcal{O}$.
Replacing $f$ by $f\circ\varphi$ we may therefore assume that $B=\mathcal{O}$.
If $f=0$ then the claim is trivial so assume that $f\neq0$. Replacing
$f$ by a scalar multiple shifts the function $x\mapsto vfx$ by an
additive constant, so we may assume that $f$ has coefficients in
$\mathcal{O}_{F}$ and at least one of the coefficients is of valuation
zero. Then $vfa\geq0$ for all $a\in\mathcal{O}$, with equality in
the case where the residue class of $a$ is transformally transcendental
over $k_{F}$.
\end{proof}

\begin{defn}
Let $F$ be a model of $\VFE$ which is separably algebraically
closed and transformally Henselian. Let $B$ be an $F$-definable
ball or a properly infinite intersection. Let $a$ be an element of
$B$.
\begin{enumerate}
    \item Let us assume that $B$ is closed. Then $a$ is \emph{generic}
in $B$ over $F$ if for all difference polynomials $fx\in F\left[x^{\mathbf{N}\left[\sigma\right]}\right]$,
the quantity $vfa$ is the generic value of $vfx$ on the ball $B$ (i.e, the minimum, which is attained as explained in in Lemma \ref{lem:closed-balls-are-compact}).

\item Let us assume that $B$ is open or a properly infinite intersection;
then $a$ is \emph{generic} in $B$ over $F$ if it avoids $U$,
whenever $U\subseteq B$ is a (possibly degenerate) closed subball
definable over $F$.
\end{enumerate}
\end{defn}

\begin{prop}
\label{prop:generic-type-of-a-ball-is-complete-and-amalgamative}Let
$F$ be a model of $\VFE$ which is separably algebraically
closed and transformally Henselian. Let $B$ be an $F$-definable
ball or a properly infinite intersection. Then genericity in $B$
over $F$ determines a complete, consistent quantifier free type over
$F$ in the language of transformal valued fields, a realization of
which is transformally transcendental over $F$. Furthermore let $a$
be generic in $B$ over $F$ and set $E=F\left(a^{\mathbf{N}\left[\frac{\sigma}{p^{\infty}}\right]}\right)$;
then $E$ is a model of $\VFE$ which is a strict amalgamation basis.
\end{prop}

\begin{proof}
By Remark \ref{lem:transformally-transcendental-FE}, the field $E$ is a model of $\VFE$. Now \cite[Corollary 6.19]{dor-hrushovskiVFA} essentially gives the result. However, it is not directly quotable in the present setting for the following reasons:
\begin{enumerate}
    \item The paper \cite{dor-hrushovskiVFA} deals exclusively with the case of models of $\VFA$ which are perfect and inversive, whereas here we deal with models of $\VFE$. Thus it does not apply to $E$ or $F$ directly, but rather to their perfect, inversive hulls.
    \item The definition of a "generic type" in \cite[Definition 6.13]{dor-hrushovskiVFA} is slightly more restrictive compared to the definition of generic considered here. More precisely, a generic type in the sense of \cite{dor-hrushovskiVFA} corresponds to the generic type of a \emph{split} ball, while the statement here refers to definable (or $\infty$-definable) balls.
\end{enumerate}

Now (1) is completely formal: the field $E$ is a strict amalgamation basis if and only if the inversive, perfect hull of $E$ is a strict amalgamation basis. For (2), using Remark \ref{rem:strict-amalgamation-basis-vs-complete-vfa-type} (and the notation there), the statement here refers to the completeness of a type in the theory $\TwVFA$ (and the various reducts $\TwVFA_n$ for $0 < n \in \mathbf{N}$ varying). The completeness of a type can be checked after base extension; replacing $F$ by a sufficiently saturated model of $\TwVFA$ we can assume the ball $B$ splits over $F$. In this situation,  \cite[Corollary 6.19]{dor-hrushovskiVFA} applies.\qedhere

\end{proof}

\begin{prop}
\label{prop:over-tilde-wvfe-transcendental-is-generic-in-ball}Let
$F$ be a model of $\widetilde{\VFE}$. Let $E$ be a model
of $\VFE$ transformally separable over $F$. Let us assume
that $k_{F}$ is transformally separably algebraically closed in $k_{E}$;
then $F$ is transformally separably algebraically closed in $E$.
In fact every element $a\in E$ not in $F$ is generic over $F$ in
an $F$-definable ball or a properly infinite intersection.
\end{prop}

\begin{rem}
\label{rem:triv-on-generic}In the situation of Proposition \ref{prop:over-tilde-wvfe-transcendental-is-generic-in-ball}. Recall that we work in an ambient $\TVFA$; so "definable" means in the sense of this ambient model.
Fix an element $a\in E$ not in $F$. Note that $a$ is not definable
over $F$, i.e, it is not in any degenerate $F$-definable subball:
the only elements definable over $F$ and not in $F$ are the elements
of the inversive hull (as the latter is a model of $\TVFA$), whereas by assumption the field $E$ is transformally
separable over $F$. Also, since $F$ lies dense in its inversive
hull, the difference field $k_{F}$ is inversive, so the assumption
implies that no element of $k_{E}$ is transformally algebraic over
$k_{F}$, unless it already lies in $k_{F}$.
\end{rem}

\begin{proof}
Let $\mathcal{B}$ be the family of all $F$-definable balls containing
$a$, and $B$ the intersection of the members of $\mathcal{B}$.
If $\mathcal{B}$ has a maximal element under reverse inclusion, which
is open, or else no minimal element under reverse inclusion, then
$a$ is generic over $F$ in $B$, using Remark \ref{rem:triv-on-generic}.
We may therefore assume that $B$ is closed. Since $F\models\widetilde{\VFE}$,
the ball $B$ is split over $F$, so without loss we have $B=\mathcal{O}$.
By Remark \ref{rem:triv-on-generic} the residue class of $a$ must
then be transformally transcendental over $k_{F}$, which precisely
means that $a$ is generic over $F$ in $\mathcal{O}$.
\end{proof}

\section{The Model Companion}\label{S:model companion}

The purpose of this section is to prove that the theory $\TwVFEe$ is the model companion of $\wVFEe$ in an expansion of the language of transformal valued fields by the transformal $\lambda$-functions.

\subsection{Density of Generic Types and Deep Ramification}

In this subsection we prove a key statement towards the proof of model
completeness of $\widetilde{\VFE_{e}}$. Before we turn to the
statement and its proof let us explain the motivation. 

Fix a saturated model $K$ of $\widetilde{\VFE_{e}}$. Let us
consider definable subsets of $K$ in one variable. There are at least
two notions of largeness for such definable sets. The first notion
of largeness comes from the valuation topology: if $B\subseteq K$
is an open ball, then a definable subset $X$ of $B$ is large if
it is not contained in any proper definable subball. The second notion
of largeness comes from the failure of $\sigma$ to be surjective:
the field $K$ is an infinite dimensional vector space over $K^{\sigma}$,
so a definable subset of $K$ is large if it is not contained in any
finite dimensional $K^{\sigma}$-subspace of $K$. 

The main observation is that these two notions of genericity are not
in conflict with each other, in the sense made precise below:
\begin{prop}
\label{prop:density-of-generic-types}
Let $K$ be a saturated model
of $\widetilde{\VFE_{e}}$. Let $F\subseteq K$ be a small model
of $\VFE$ contained in $K$, separably algebraically closed
and transformally Henselian. Let us assume that $K$ is almost transformally
separable over $F$. Let $B$ be an $F$-definable ball or a properly
infinite intersection of $F$-definable balls. Let $E=F\left(a^{\mathbf{N}\left[\frac{\sigma}{p^{\infty}}\right]}\right)$
where $a$ is generic over $F$ in $B$.
\begin{enumerate}
    \item Then there is an embedding of $E$ in $K$ over $F$ so as to
render $K$ almost transformally separable over the image.
\item Let us assume that $K$ is transformally separable over $F$ outright
and that the perfect hull of $K$ does not split over $F$; then there
is an embedding of $E$ in $K$ over $F$ so as to render $K$ transformally
separable over the image.
\end{enumerate}
\end{prop}

\begin{proof}
We prove the first point; the second point is similar, using the fact
that $K^{p}$ lies dense in $K$. By compactness, we may assume that
$B$ is an honest ball, rather than an $\infty$-definable ball. Moreover,
we deal with the case of a closed ball; the open case is similar. 

Using Proposition \ref{P:crieria-for-tsep-and-almost-tsep} our task
is to find an element $x\in K$ simultaneously obeying the following
two properties. First, we want $x$ to be algebraically transcendental
over $F\tensor_{F^{\sigma\cdot p^{-\infty}}}K^{\sigma\cdot p^{-\infty}}$.
Second, we want $x$ to be generic over $F$ in $B$. 

Increasing $F$ only makes these tasks more difficult, so assume $F$
is an elementary substructure of $K$. Then $F$ is a model of $\widetilde{\VFE}$,
and $K$ is transformally separable over $F$. By Proposition \ref{prop:over-tilde-wvfe-balls-are-triv},
the ball $B$ splits over $F$. We may therefore assume without loss
of generality that $B=\mathcal{O}$.

For the sake of the argument, say that $y\in K$ is \emph{topologically
generic} over $F$ if $vy=0$ and the residue class is transformally
transcendental over $k_{F}$, i.e, it is generic over $F$ in $\mathcal{O}$.
Moreover, say that $z\in K$ is \emph{$\lambda$-generic} over $F$
if it is algebraically transcendental over $F\tensor_{F^\sphul}K^{\sphul}$.
Finally, say that $x$ is \emph{generic} over $K$ if it is simultaneously
$\lambda$-generic and topologically generic over $F$. Our task,
then, is to find an element of $K$ which is generic over $F$.

\textbf{Claim 1.} The set $X_{v}$ of elements of $K$ which are topologically
generic over $F$ is nonempty. Indeed, it contains an open ball.

\textbf{Proof.} The fact that $X_{v}$ is nonempty follows from the
fact that $k$ is a saturated model of $\ACFA$. Given that it is
nonempty, it is by definition the union of residue classes, and hence open.

\textbf{Claim 2.} The set $X_{\lambda}$ of elements of $K$ which
are $\lambda$-generic over $F$ is nonempty. Indeed, it lies dense
in $K$ for the valuation topology.

\textbf{Proof.} Since $K$ fails to be inversive, it follows from
saturation that $X_{\lambda}$ is nonempty; see Lemma \ref{lem:existence-of-generics}.
It is clear that $X_{\lambda}$ is stable under affine translation
with coefficients in $K^{\sigma}$. By Lemma \ref{lem:deep-ramification},
the field $K^{\sigma}$ lies dense in $K$ for the valuation topology;
the same must then be true of $X_{\lambda}$.

Using the notation of the above claims, the set $X_{v}$ contains
an open ball, and the set $X_{\lambda}$ is dense for the valuation
topology; the intersection $X_{v}\cap X_{\lambda}$ is therefore nonempty.
By definition, every element of the intersection is generic over $F$,
whence the proposition.
\end{proof}

\subsection{The Model Companion} \label{ss: final}
\begin{prop}
\label{prop:every-wvfe-embed-in-tilde}Let $F$ be a model of $\VFE_{\leq e}$.
Then there is a model of $\widetilde{\VFE_{e}}$ transformally
separable over $F$.
\end{prop}

\begin{proof}
We work inside an ambient model $\mathbf{K}$ of $\TVFA$
over $F$. We gradually replace $F$ by a model $E$ of $\VFE_{\leq e}$
inside $\mathbf{K}$, transformally separable over $F$, so as to
make sure that all the axioms are satisfied. Basic closure properties
of the category of models of $\FF$ and elementary properties of transformally
separable extensions will be used without mention.

\textbf{Step 1.} We may assume that $F$ fails to be inversive. Indeed, in
this case the field $F$ must also be perfect. Let $x\in\mathbf{K}$
be transformally transcendental over $F$ and let $E=F\left(x^{\mathbf{N}\left[\sigma\right]}\right)^{\alg}$
be the displayed transformal valued field. Then $E$ is a perfect
model of $\VFE$ which by construction fails to be inversive.

\textbf{Step 2.} We may assume that $F$ has the correct Ershov invariant,
i.e that $\left[F\colon F^{p}\right]=p^{e}$. Indeed let $x=\left(x_{1},\ldots,x_{d}\right)\in\mathbf{K}$
be a tuple transformally transcendental over over $F$ and set $E=F\left(x^{\mathbf{N}\left[\frac{\sigma}{p^{\infty}}\right]}\right)$.
Then $x$ is a relative $p$-basis of $E$ over $F$ hence $\left[E\colon E^{p}\right]=\left[F\colon F^{p}\right]\cdot p^{d}$.
In this way we can increase the Ershov invariant which settles the
case $e<\infty$; in the case $e=\infty$ take a directed union.

\textbf{Step 3.} We may assume that $F$ is nontrivially valued. Indeed
let $x\in\mathbf{K}$ be transformally transcendental over $F$ with
$vx>0$. Then $E=F\left(x^{\mathbf{N}\left[\sigma^{\pm1},p^{\pm1}\right]}\right)$
is a model of $\VFE$ transformally separable over $F$; by
construction it is relatively perfect over $F$ hence a model of $\VFE_{\leq e}$
as well.
\\

Now let $E$ be the relative transformal separable algebraic closure
of $F$ in $\mathbf{K}$. Then $E$ is nontrivially valued and transformally
separably algebraically closed in $\mathbf{K}$; by Proposition \ref{P:pure insep inside twvfe},
the field $E$ is a model of $\widetilde{\VFE}$. By Proposition
\ref{tsep-alg-of-fe-preserves-ershov}, the perfect hull of $E$ splits
over $F$, so the imperfection degrees are the same; and of course
$E$ fails to be inversive if $F$ does. Thus $E$ is a model of $\widetilde{\VFE_{e}}$,
and we are done.
\end{proof}

\begin{thm}
\label{thm:wVFE-is-EC}Let $K$ be a saturated model of $\widetilde{\VFE_{e}}$.
Then $K$ is existentially closed for transformally separable extensions
of the same Ershov invariant in the following sense. Let $F\subseteq K$
be a small strictly amalgamative model of $\VFE_{\leq e}$ with $K$
transformally separable over $F$. Let us assume that we are given
a small model $E$ of $\VFE_{\leq e}$ which is transformally separable
over $F$. Then there is an embedding of $E$ in $K$ over $F$ so
as to render $K$ transformally separable over the image.
\end{thm}

Before we turn to the proof of Theorem \ref{thm:wVFE-is-EC}, we wish
to make several elementary remarks.
\begin{rem}
\label{rem:perfect-hull-splits}In the situation of Theorem \ref{thm:wVFE-is-EC},
assume that the perfect hull of $E$ does not split over $F$. Then
the perfect hull of $K$ does not descend to $F$, either. Indeed
a transformally separable extension of models of $\FF$ is also algebraically
separable. In the case $e=\infty$, the vector space $K$ is of large
dimension over $K^{p}$; the inclusion $F\otimes_{F^{p}}K^{p}\subseteq K$
must therefore be proper as $F$ is small. If $e<\infty$ then the
inequality $\dim_{E^{p}}E\leq\dim_{K^{p}}K$ holds by assumption,
and again we are in the clear.
\end{rem}

\begin{rem}
\label{rem:increase-harmless}Let us assume that $L$ is a model of
$\VFE_{\leq e}$ transformally separable over $E$. Suppose we are given
given an embedding of $L$ in $K$ over $F$ so as to render $K$
transformally separable over the image. By abuse of notation, let
us regard $L$ - and hence $E$ - as a transformal valued subfield
of $K$ over $F$. Then as $\nicefrac{K}{L}$ and $\nicefrac{L}{E}$
are transformally separable, the field extension $\nicefrac{K}{E}$
is transformally separable as the composition of such. It is therefore
harmless to increase $E$ by a transformally separable extension;
using Proposition \ref{prop:every-wvfe-embed-in-tilde}, we may therefore
assume that $E$ is a model of $\widetilde{\VFE_{e}}$, and
in particular that it is separably algebraically closed and transformally
Henselian.
\end{rem}

\begin{rem}
\label{rem:find-proper-inductively}We may assume, of course, that
$E$ properly extends $F$, for otherwise the claim is trivial. Now
assume further $F\subseteq E_{0}\subseteq E$ is a model of $\VFE$
which is strictly amalgamative and that $E$ is transformally separable
over $E_{0}$. If $E_{0}$ properly extends $F$, then the truth of
the statement for $\nicefrac{E}{E_{0}}$ and $\nicefrac{E_{0}}{F}$
separately implies it for $\nicefrac{E}{F}$. By transfinite induction
and compactness, we can therefore assume that no such $E_{0}$ can
be found, unless it coincides with $E$.
\end{rem}

\textbf{Proof of Theorem \ref{thm:wVFE-is-EC}.} Let $F,K$ and $E$
be as in the statement. By Remark \ref{rem:increase-harmless}, it
is harmless to assume that $E$ is a model of $\widetilde{\VFE_{e}}$.

We first reduce ourselves to the case where $F$ is transformally
separably algebraically closed in $E$. Indeed, let $F\subseteq E_{0}\subseteq E$
be the relative transformal separable algebraic closure of $F$ in
$E$. The field $E_{0}$ is separably algebraically closed, as it
is relatively separably algebraically closed inside the separably
algebraically closed field $E$; in particular the field $E_{0}$
is a strict amalgamation basis. By Proposition \ref{prop:ec-tsep-alg},
the models of $\widetilde{\VFE_{e}}$ are existentially closed
for transformally separably algebraic extensions, so there is an embedding
of $E_{0}$ in $K$ over $F$, and such an embedding automatically
renders $K$ transformally separable over the image. Thus inductively
using Remark \ref{rem:find-proper-inductively}, we may assume that
$F$ is transformally separably algebraically closed in $E$. 

By Proposition \ref{P:pure insep inside twvfe}(2), $F$ is a model of $\TVFE$ and so using Proposition \ref{prop:th-rel}, Proposition \ref{prop:over-tilde-wvfe-transcendental-is-generic-in-ball} implies
that an element $a\in E$ lying outside $F$ is generic over $F$
in an $F$-definable ball or a properly infinite intersection of $F$-definable
balls.

Next we claim that $E$ can be taken to be a relatively perfect extension
of $F$. If $E$ is not relatively perfect over $F$, fix an element
$a\in E$ not in $F\otimes_{F^{p}}E^{p}$ and set $E_{0}=F\left(a^{\mathbf{N}\left[\frac{\sigma}{p^{\infty}}\right]}\right)$.
Then $a$ is generic over $F$ in a ball, so by Proposition \ref{prop:generic-type-of-a-ball-is-complete-and-amalgamative},
the field $E_{0}$ is a strict amalgamation basis. By Proposition \ref{P:crieria-for-tsep-and-almost-tsep},
the field $E$ is transformally separable over $E_{0}$. Using Remark
\ref{rem:find-proper-inductively}, we can replace $E$ by $E_{0}$;
in this situation however Proposition \ref{prop:density-of-generic-types}
on the density of the generic types applies.

Now $E$ is relatively perfect over $F$. In this situation, an embedding
of $E$ in $K$ over $F$ will \emph{automatically} render $K$
algebraically separable over the image by Fact \ref{fact:rel-perf-transitive-in-towers};
using the criterion of Proposition \ref{criterion-algsep-almosttsep-for-fe}
it is therefore enough to exhibit an embedding of $E$ in $K$ over
$F$ with $K$ almost transformally separable over the image. We now
turn to this task. In order to achieve it, we reinstall our settings
and weaken our assumption to allow ourselves to argue inductively.
Thus we forsake the assumption that $K$ is transformally separable
over $F$, and merely assume that it is almost transformally separable
over $F$. Moreover, we do not make any assumptions on the degree
of imperfection of $E$, and allow $E$ to be an arbitrary model of
$\VFE$ transformally separable over $F$.

By compactness and finite character of almost transformally separable
extensions, we may assume that $E$ is transformally separably generated
over $F$, as Theorem \ref{thm:main-thm-on-tsep}. Fix an element
$a\in E$ forming a part of a separating transformal transcendence
basis of $E$ over $F$; then $E$ is transformally separable over
$E_{0}$, where $E_{0}=F\left(a^{\mathbf{N}\left[\frac{\sigma}{p^{\infty}}\right]}\right)$.
The field $E_{0}$ is strictly amalgamative; arguing as in Remark
\ref{rem:find-proper-inductively} and using the fact that almost
transformally separable extensions are closed under composition, we
may assume that $E=E_{0}$. But Proposition
\ref{prop:density-of-generic-types}(2) applies, and the proof is finished.\qed

\begin{cor}
\label{cor:qe}Let $F$ be a model of $\VFE_{\leq e}$ which is a strict amalgamation basis. Then there is up to elementary equivalence a unique
model of $\widetilde{\VFE_{e}}$ which is transformally separable
over $F$.
\end{cor}

\begin{proof}
This follows immediately from Theorem \ref{thm:wVFE-is-EC} and Proposition
\ref{prop:every-wvfe-embed-in-tilde}.
\end{proof}

Let $\mathcal{L}_{\lambda}$ be the language of models of $\VFE$ with the transformal $\lambda$-functions. Thus for  $K \models \TwVFE_{\lambda}$, an $\mathcal{L}_{\lambda}$-substructure of $K$ is a model $F$ of $\VFE$ contained in $K$ such that $K$ is transformally separable over $F$. 

\begin{cor}\label{cor:qe-explained}
    In the theory $\TwVFE_{\lambda}$ every formula is equivalent to a boolean combination of formulas of the form $\exists y \varphi\left(x,y\right)$ where $\varphi \in \mathcal{L}_{\lambda}$ is quantifier free and each $y_i$ obeys a monic (algebraic) polynomial with coefficients in $\mathbf{Z}\left[x\right]$.
\end{cor}
\begin{proof}
Let $K$ be a saturated model of $\TwVFE_{\lambda}$. By definition, an $\mathcal{L}_{\lambda}$-substructure of $K$ is a model $F$ of $\VFE$ contained in $K$ such that $K$ is transformally separable over $F$.
Let $\mathcal{S}$ be the set of formulas in the statement. We are given an isomorphism $f \colon  F \cong F'$ between models of $\VFE$ contained in $K$, with $K$ transformally separable over both, and that the isomorphism preserves the truth value of formulas in $\mathcal{S}$; we must show that $f$ lifts to a global automorphism.
Now using formulas in $\mathcal{S}$ we can determine the isomorphism type of the action of $\sigma$ on the separable algebraic closure of both $F$ and $F'$ inside $K$ (see, e.g., the proof of \cite[Proposition 9.17]{dor-hrushovskiVFA}). Thus we may extend $f$ to the relative separable algebraic closures inside $K$. By elementary properties of transformally separable extensions, this does not disturb the property that $K$ is transformally over $F$, and likewise $F'$; and now Corollary \ref{cor:qe} applies.
\end{proof}

\begin{thm}\label{thm:model-companion}
    Let $T = \VFE_{\leq e}$ and $\widetilde{T} = \TwVFE_{e}$. Then $\widetilde{T}_{\lambda}$ is the model companion of $T_{\lambda}$.
\end{thm}

\begin{proof}
This follows immediately from Corollary \ref{cor:qe} and Proposition \ref{prop:every-wvfe-embed-in-tilde}.
\end{proof}

\begin{thm}\label{T:model theoretic acl}
Let $K$ be a model of $\widetilde{\VFE_{e}}$ and let $F\subseteq K$
be a model of $\VFE$ which is contained in $K$. Then $F$
is model theoretically algebraically closed in $K$ if and only if
it is separably algebraically closed, transformally Henselian, and
closed under the transformal $\lambda$-functions of $K$.
\end{thm}

\begin{proof}
We may assume that $K$ is sufficiently saturated and that $F$ is
small.

Recall that $K$ is transformally separable over $F$ if and only
if $F$ is closed under the transformal $\lambda$-functions of $K$,
and that the transformal $\lambda$-functions of $K$ are definable
in the language of difference fields. With this description it is
clear that the conditions are necessary. For the converse, assume
the conditions of the theorem, and fix an element $a\in K$ not in
$F$; we will show that it admits infinitely many conjugates over
$F$.

We may assume that $a$ enumerates a small separably algebraically closed and
transformally Henselian model $E$ of $\VFE$ over $F$, closed
under the transformal $\lambda$-functions of $K$; so $K$ is transformally
separable over $E$. By model completeness, Proposition \ref{prop:lin-disjoint-amalgamation}
and saturation there is a sequence $\left(E_{n}\right)_{n=0}^{\infty}$
of linearly disjoint copies of $E=E_{0}$ over $F$ inside $K$ all
closed under the transformal $\lambda$-functions of $K$. By linear
disjointness we have $E_{n}\cap E_{m}=F$ for $n\neq m$. By Corollary
\ref{cor:qe} the isomorphism $E_{n}\cong E_{m}$ for $n\neq m$ lifts
to a global automorphism over $F$. This concludes the proof.
\end{proof}

\begin{lem}
    Let $K$ be a model of $\TwVFEe$ which is $\omega$-saturated and let $F = K^{\sigma^{\infty}}$ be the maximal inversive difference subfield of $K$. Then $F$ is a model of $\TVFA$ which lies dense in $K$.
\end{lem}

\begin{proof}
    Since $K$ lies dense in its inversive hull it follows by compactness and saturation that $F$ lies dense in $K$, in particular it is nontrivially valued.   The transformal algebraic closure of $F$ in $K$ is inversive by Proposition \ref{tsep-alg-of-fe-preserves-ershov} so by maximality it is relatively transformally algebraically closed in $K$. Thus by Proposition \ref{P:pure insep inside twvfe} the field $F$ is a model of $\TVFA$.
\end{proof}

\begin{rem}
    The field $F$ is not stably embedded in $K$. Indeed:

    \textbf{Claim 1.} Every element of $K$ is the unique realization of its type over $F$ in $K$; thus in the language of \cite{hrushovski2006stable}, every imaginary of $K$ is $F$-comprehended.
    
    \textbf{Proof.} The field $F$ lies dense in $K$ for the valuation topology, so if $a \in K$ lies outside $F$, it is the unique element of $K$ inside the intersection of all $F$-definable balls containing $a$.
    
    \textbf{Claim 2.} No element of $K$ is definable over $F$ unless it already lies in $F$.
    
    \textbf{Proof.} This follows from the characterization of the model theoretic algebraic closure in $\TwVFE$ above (Theorem \ref{T:model theoretic acl}).
    By \cite{hrushovski2006stable}, a sort $D$ in a theory $T$ is stably embedded if and only if every imaginary which is $D$-comprehended is definable over $D$; the embedding of $F$ in $K$ provides an extreme example of the opposite scenario. Thus $F$ is not stably embedded in $K$.
\end{rem}

In the following we wish to understand the induced structure on the residue field and on the value group. For the latter, recall the theory $\widetilde{\omega\mathrm{OGA}}$ of $\omega$-increasing transformally divisible ordered abelian groups from \cite[Section 2]{dor-hrushovskiVFA} It is complete and admits quantifier elimination.

\begin{thm}
\label{thm:stably-embedded}In the theory $\widetilde{\VFE_{e}}$,
the residue field and the value group are stably embedded and
fully orthogonal. The induced structure is the pure difference field
structure and the pure ordered transformal module structure, respectively.
\end{thm}
\begin{proof}
Let $K$ be a sufficiently saturated model of $\TwVFE_e$. We first verify that the residue field and value group are stably embedded and fully orthogonal as a difference field structure and an ordered transformal module structure, respectively, but possibly with some added constants. After doing that we show that indeed the $\emptyset$-induced structure is that of pure structures.

It will be sufficient to prove the following: given a strict amalgamation basis $F \models \VFE_{\leq e}$, with $\nicefrac{K}{F}$ transformally separable, a tuple of residue elements $\alpha$ and a tuple of value group elements $\gamma$ we have 
\[\tp_{\mathrm{ACFA}}(\alpha/k_F)\cup \tp_{\widetilde{\omega \mathrm{OGA}}}(\gamma/\Gamma_F)\vdash \tp(\alpha,\gamma/F).\] 

We show this by proving that any pair of automorphisms of the residue field and value group (as a difference field and an ordered transformal module) lifts to an automorphism of  $K$ by the aid of Theorem \ref{thm:wVFE-is-EC}. To this end it will be sufficient to show that these automorphisms arise as an automorphism (over $F$) of some model $E$ of $\wVFE_e$ extending $F$, which is transformally separable over $F$.

We build this extension $E$ step by step by repeatedly taking care of the following kind of elements:  (1) elements constituting $\gamma$,  (2) elements of $\alpha$ which are transformally transcendental over $k_F$ and (3) elements of $\alpha$ which are transformally algebraic over $k_F$. Starting with $E=F$, at each step we will extend $E$ and the automorphism on it to induce the desired automorphism on the residue field and value group; moreover, when we take care of residue elements we will not extend the value group and vice-versa. 

(1) Let $\gamma_1,\gamma_2 \in \Gamma_K$ (singletons) be value group elements with $\tp_{\widetilde{\omega \mathrm{OGA}}}(\gamma_1/\Gamma_F)=\tp_{\widetilde{\omega \mathrm{OGA}}}(\gamma_2/\Gamma_F)$ and $\gamma_i$ not in the transformal divisible hull of $\Gamma_F$. For $i=1,2$, let $p_{\gamma_i}$ be the quantifier-free type of a generic in the closed ball of radius $\gamma_i$ around $0$, as supplied by Proposition \ref{prop:generic-type-of-a-ball-is-complete-and-amalgamative}. By density of $K^\sigma$ in $K$, Lemma \ref{lem:deep-ramification}, the type $p_{\gamma_i}$ is consistent with the maximal inversive subfield $K^{\sigma^{\infty}}$ of $K$. Let $a_i\in K^{\sigma^\infty}$ be a realization of $p_{\gamma_i}$ inside the maximal inversive subfield of $K$. It follows from the proof of Proposition \ref{prop:generic-type-of-a-ball-is-complete-and-amalgamative}, that the field $E_i=F\left(a_i^{\mathbf{N}\left[\sigma^{\pm1},p^{\pm1}\right]}\right)$ is strictly amalgamative. It is obviously a model of $\wVFE_e$ and as $(E_i)^\inv=F^\inv\left(a_i^{\mathbf{N}\left[\sigma^{\pm1},p^{\pm1}\right]}\right)$, $K$ is transformally separable over $E_i$. Finally, since  $\tp_{\widetilde{\omega \mathrm{OGA}}}(\gamma_1/\Gamma_F)=\tp_{\widetilde{\omega \mathrm{OGA}}}(\gamma_2/\Gamma_F)$, the fields $E_1$ and $E_2$ are isomorphic as difference valued fields over $F$. Note that by construction, $\Gamma_{E_i}$ lies in the transformal divisible hull of $\Gamma_F\cup \{\gamma_i\}$ and that $k_{E_i}=k_F$.

If the $\gamma_i$ are in the transformal divisible hull of $\Gamma_F$, then they are in its definable closure, so are automatically preserved.

(2) Let $\alpha_1,\alpha_2 \in k_K$ (singletons) with $\tp_{\mathrm{ACFA}}(\alpha_1/k_F)=\tp_{\mathrm{ACFA}}(\alpha_2/k_F)$ and $\alpha_i$ transformally transcendental over $k_F$. We now proceed similarly as in (1), realizing the generic of the closed ball of radius $0$ around $0$. Note that, in the notation of (1), we get $k_{E_i}=k_F(\alpha_i^{\mathbf{N}\left[\sigma^{\pm1},p^{\pm1}\right]})$ and $\Gamma_{E_i}=\Gamma_F$.

(3) Let $\alpha_1,\alpha _2\in k_K$ be singletons with $\tp_{\mathrm{ACFA}}(\alpha_1/k_F)=\tp_{\mathrm{ACFA}}(\alpha_2/k_F)$ and $\alpha_i$ are in the relative transformal algebraic closure of $k_F$ in $k_K$. Let $\widetilde{k_F}$ be the relative transformal separable algebraic closure of $k_F$ in $k_K$.  By \cite[Lemma 3.6]{dor-hrushovskiVFA}, the $\alpha_i$ are in the inversive hull of $\widetilde{k_F}$. Hence, there exists a twisted difference polynomial over $k_F$ and some natural number $n$ such that $\alpha_1^{\sigma^n}$ and $\alpha_2^{\sigma^n}$ are both simple roots of this twisted difference polynomial. Since $\alpha_i^{\sigma^n}$  and $\alpha_i$ are interdefinable, there is not harm in the replacing the latter with the former. So assume that the $\alpha_i$ are simple roots of twisted difference polynomials over $k_F$.

By \cite[Proposition 9.13]{dor-hrushovskiVFA}, the difference field $K^\inv$ is saturated in $\kappa^+$ if $K$ is. Indeed, their residue fields and value groups are equal, respectively, and a nested collection of balls of cofinality $\kappa$ in $K^\inv$ have centers in $K$ by Lemma \ref{lem:deep-ramification} so since they have non empty intersection in $K$ it follows for $K^\inv$ as well. As a result, $K^\inv$ is saturated as well.  By stable embeddedness of $k$ in $\widetilde{\VFA}$ \cite[Theorem 9.14]{dor-hrushovskiVFA}, there exists an automorphism $\rho$ of $K^\inv$ over $F$ which maps $\alpha_1$ to $\alpha_2$. By Theorem \ref{thm:main-thm-on-tsep}, $\rho$ restricts to $\widetilde F$, the relative transformal separable algebraic closure of $F$ inside $K^\inv$. It is a model of $\TwVFE$ by Proposition \ref{P:pure insep inside twvfe}. In particular, it is strictly amalgamative and $K$ is transformally separable over $F$. Finally, by model completeness, $\widetilde F \prec K$, and thus since $\widetilde F$ is relatively transformally separably
algebraically closed in $K$, the same is true for $k_{\widetilde F}$ in $k$; so $\alpha_i\in k_{\widetilde F}$.
%
%

We now verify the induced structure is precisely the pure difference field and ordered module structure; that is, we prove that $\emptyset$-definable subsets of $k^d$ definable in the language of $\VFE$ are already $\varnothing$-definable in the language of difference fields and likewise for $\Gamma$. For this, we show that given small models $k_0\prec k_K$ and $\Gamma_0\prec \Gamma_K$ any pair of automorphism of $k_0$ and $\Gamma_0$ lift to an automorphisms of $K$.

Given such models $k_0$ and $\Gamma_0$, the Hahn series $K_0=k_0((t^{\Gamma_0}))$ is a model of $\TVFA$ (see for example the proof of Proposition 9.14 in \cite{dor-hrushovskiVFA}). We claim that $K_0$ embeds in $K$.

Let us assume for a moment that this was shown and finish the proof. The map $\Aut\left(K_0\right) \to \Aut(k_0) \times \Aut(\Gamma_0)$ is evidently surjective (indeed, it even admits a section). The field $K_0$ is algebraically closed and inversive, hence a strict amalgamation basis; using Corollary \ref{cor:qe}, every automorphism of $K_0$ lifts to an automorphism of $K$.

It remains to be shown that $K_0$ can be embedded in $K$ (over the prime field, i.e, over $\varnothing$). 

Let $\mathbf{F}$ be the algebraic closure of the prime field of $k_0$. By Hensel's lemma it identifies with the algebraic closure of the prime field in $K_0$. Since $\mathbf{F}$ is a strict amalgamation basis, by Theorem \ref{thm:wVFE-is-EC} we may  embed $K_0$ in $K$ over $\mathbf{F}$ and we are done.
\end{proof}

\begin{cor}
    Let $K$ be a model of $\TwVFEe$, and let $\tau = \sigma^n p^m$ with $n,m \in \mathbf{N}$, with $n > 0$ and $n,m$ relatively prime. Let $F$ be the fixed field of $\tau$. Then every definable subset of $F^d$ which is definable with parameters in $K$ is already definable with parameters in $F$ and in the language of fields.
\end{cor}

\begin{proof}
    This follows from the corresponding statement in $\ACFA$ (see \cite{chatzidakis2002model}), stable embeddedness of the residue field, and transformal Henselianity.
\end{proof}

\begin{cor}
\label{cor:tilde-w-vfe-is-complete-rel-resd} The theory $\TwVFEe$
is complete relative to the theory of the residue field. Equivalently,
the completions of $\widetilde{\VFE_{e}}$ are given by specifying
the isomorphism type of the action of $\sigma$ on the algebraic closure
of the prime field.
\end{cor}

\begin{proof}
By Hensel's lemma the algebraic closure of the prime field identifies
via the residue map with the algebraic closure of the prime field
of the residue field. So this follows immediately from Corollary \ref{cor:qe}
by taking $F$ to be the relative algebraic closure of the prime field
and the corresponding statement for $\ACFA$.
\end{proof}

\begin{thm}
    \label{thm:wvfe-decidable} The theory $\TwVFE$ is decidable.
\end{thm}

\begin{proof}
 Let $p$ be a fixed characteristic exponent and let $e \in \mathbf{N}_\infty$ be the degree of imperfection, so $e = 0$ when $p = 1$. We will prove that the theory $T_{p,e}$ of models of $\TwVFE$ of characteristic exponent $p$ and degree of imperfection $e$ is decidable; the general case follows from compactness.
 By Corollary \ref{cor:tilde-w-vfe-is-complete-rel-resd} the theory $T_{p,e}$ of models of $\TwVFE$ is complete relative to the residue field. By Theorem \ref{thm:stably-embedded} the induced structure on the residue field is that of a pure model of $\ACFA$; since $\ACFA$ is decidable (see \cite{ChHr-Dif}), we find that $T_{p,e}$ is also decidable.
 
 Now fix a prime $p$. Let $T$ be the theory of models of $\TwVFE$ of characteristic $p$ and $T_e$ the theory of models of $\TwVFE$ of characteristic $p$ and Ershov invariant $e$; we will prove that $T$ is decidable. Given a sentence $\theta$ in the language of transformal valued fields we must determine whether or not $\theta$ holds in every model of $T$. Using the result of the previous paragraph we can effectively determine whether or not $\theta$ holds in every model of $T_{\infty}$. If $\theta$ fails in some model of $T_\infty$ then we are done; so assume that $T_\infty \vdash \theta$. By compactness there is an effective $N = N\left(\theta\right)$ such that $T_e \vdash \theta$ for all $e > N$; this reduces our problem to verifying whether or not $T_e \vdash \theta$ for finitely many choices of $e$.
 
 This shows that the theory of models of $\TwVFE$ of any fixed characteristic is decidable. Using the reasoning of the previous paragraph we conclude that $\TwVFE$ is decidable in general.
\end{proof}

\textbf{$\TwVFE$ and the asymptotic theory of the Frobenius action.} By a \textit{Frobenius transformal valued field} we mean a transformal valued field of positive characteristic $p$ whose distinguished endomorphism coincides with a power of the Frobenius endomorphism $x \mapsto x^p$. We regard Frobenius transformal valued fields as structures for the language of transformal valued fields; this expands the language of valued fields by a unary function symbol for the action of $\sigma = p^n$.

Let $\mathcal{C}$ be the class of separably algebraically closed Frobenius transformal valued fields. By the \textit{asymptotic theory} of $\mathcal{C}$ we mean the set of sentences true in all members of $\mathcal{C}$, outside a finite set of exceptional prime powers.

\begin{thm}\label{T:asymp theory}
Let $T$ be the asymptotic theory of separably algebraically closed, nontrivially valued, Frobenius transformal valued fields. Then $T$ is logically equivalent to $\TwVFE$, together with a scheme of axioms asserting that if the characteristic is $p$, then $\sigma$ and the Frobenius are simultaneously onto.
\end{thm}
\begin{proof}

Let $K$ be a model of $\TwVFE$ obeying the scheme of axioms of the theorem; we must prove that $K$ is elementarily equivalent to an ultraproduct of Frobenius transformal valued fields. If $K$ is inversive then the result follows from Fact \ref{fact:wvfa} so assume that $K$ fails to be inversive. Then either $K$ is of positive characteristic $p > 0$ and Ershov invariant $e > 0$ or else it is of characteristic zero. We deal with the positive characteristic case; the case $p = 0$ is similar. By \cite{hrushovski2004elementary} the residue field $k$ of $K$ is elementarily equivalent to an ultraproduct of algebraically closed Frobenius difference fields $\left(k_{q_i}, x \mapsto x^{q_i}\right)$ for powers $q_i$ of $p$ with $q_i \to \infty$. For each $i$, let $K_{q_i}$ be a separably algebraically closed and nontrivially valued field of degree of imperfection $e$ and with residue field $k_{q_i}$. It follows from Corollary \ref{cor:tilde-w-vfe-is-complete-rel-resd} that the ultraproduct of the $\left(K_{q_i}, x \mapsto x^{q_i}\right)$ is elementarily equivalent to $K$, and we are done.
\end{proof}

\section{Stability Theoretic Phenomena}

In the final section we study stability theoretic phenomena in the theory $\TwVFE_{
\mathbf{Q}}$ (we restrict our selves to the characteristic zero case  for simplicity). The characterization of forking independence is the same in finite characteristic but the classification of the stationary types is not. Before we turn to the details it will be useful to give an overview of stability theoretic phenomena in the various reducts of $\TwVFE$.

\textbf{Stability Theory in $\ACVF$}. The theory $\ACVF$ of algebraically closed and nontrivially valued fields is unstable but enjoys $\NIP$. Thus forking independence in $\ACVF$ is of fundamentally geometric character, namely a global type over an algebraically closed set does not fork precisely in the event that it is invariant.

In \cite{HHMStable} a fundamental role is played by stably dominated types; for such types the entire forking calculus is available, namely we have uniqueness and symmetry of forking. In $\ACVF$, they can be equivalently characterized as the generically stable types (equivalently the stationary types, since NIP), and the types orthogonal to the value group $\Gamma$. Thus in some loose sense all instability in $\ACVF$ can be traced back to the value group.

The orthogonality to the value group however is only directly apparent over spherically complete models (which are also of fundamental importance in \cite{HHMStable}). Over a spherically complete model, all types are dominated by the value group and the residue field. We will see below that the principle of spherically complete domination by the value group and the residue field remains valid in $\TVFA_{\mathbf{Q}}$ as well.

\textbf{Stability Theory in $\ACFA_{\mathbf{Q}}$.} The theory $\ACFA$ is supersimple but unstable. Forking independence is governed by linear disjointness, that is by forking in the reduct to the field language. For an extension $\nicefrac{B}{A}$ of algebraically closed inversive fields to be stationary in $\ACFA$ it is sufficient that $B \otimes_{A} C$ has no nontrivial finite $\sigma$-invariant Galois extensions, for all algebraically closed inversive extensions $C$ of $A$; types of this form are called \emph{superficially stable} in \cite{ChHr-Dif}. In characteristic zero it turns out that superficially stable types are precisely those orthogonal to the fixed field (see Fact \ref{fact:acfa-char-0-stationary} below). In finite characteristic there exist non-superficially stable types which are modular; they are related to additive equations involving $\sigma$ and the Frobenius, see \cite[Example 6.5]{ChHr-Dif}.

\begin{rem}\footnote{We are grateful to Zo\'e Chatzidakis and Ehud Hrushovski for useful discussions around this issue.}\label{rem:stability-in-acfa}
    We work in $\ACFA$. Let $\nicefrac{B}{A}$ be an extension of inversive algebraically closed difference fields. Say that $\nicefrac{B}{A}$ is \emph{generically stable} if for all algebraically closed and inversive $C$ linearly disjoint from $B$ over $A$, the difference field $B \otimes_{A} C$ admits a unique lift of $\sigma$ to the algebraic closure, up to isomorphism. If $\nicefrac{B}{A}$ is superficially stable, then it is generically stable; conversely in $\ACFA_{
\mathbf{Q}}$, it can be shown that there is no distinction between superficial stability and generic stability. It is not known if this holds for modular types in finite characteristic; the question is related to the the classification of unstable types in $\ACFA_{\mathbf{F}_p}$ orthogonal to all fixed fields.
\end{rem}

\textbf{Situation.} In what follows we will study the following situation. Let $A$ be a model of $\TVFA_{\mathbf{Q}}$ which is spherically complete. Let $B$ and $C$ be models of $\wVFA$ over $A$, jointly embedded over $A$ in an ambient extension. Let us assume that the intersection $\Gamma_B \cap \Gamma_C = \Gamma_A$ and that the residue fields $k_B$ and $k_C$ are linearly disjoint over $k_A$. Let $b$ be a tuple enumerating the elements of $B$. Moreover let $\alpha$ and $\gamma$ be tuples enumerating the residue field elements and the value group elements of $B$, respectively.

\begin{lem}\label{lem:spherically-completeq-qf}
Notation as above. Then $\qftp_{\TVFA}\left(\nicefrac{b}{A}\right) \cup \qftp_{\TwOGA}\left(\nicefrac{\gamma}{\Gamma_C}\right) \cup \qftp_{\ACFA} \left(\nicefrac{\alpha}{k_C}\right) \vdash \qftp_{\TVFA} \left(\nicefrac{b}{C}\right)$. Moreover:
    \begin{enumerate}
        \item The fields $B$ and $C$ are linearly disjoint over $A$.
        \item Let $N = B \otimes_{A} C$; then $N$ has residue field $k_B \otimes_{k_A} k_C$ and its value group is the cofibered sum $\Gamma_B \oplus_{\Gamma_A} \Gamma_C$.
        \end{enumerate}  
\end{lem}

\begin{proof}
    This follows formally from spherically complete domination in $\ACVF$, see \cite[Proposition 12.11]{HHMStable}. Namely the quantifier free type in the language of abstract valued fields is determined, hence by functoriality the lift of $\sigma$ is unique also (or just argue directly using the linear disjointness of $B$ and $C$ over $A$). The ``moreover'' clauses are part of the proof of \cite[Proposition 12.11]{HHMStable}.
\end{proof}

\begin{lem}\label{lem:twvfa-0-complete-type}
    Let $K$ be a model of $\VFA_{\mathbf{Q}}$. Assume that the value group of $K$ is algebraically divisible. Then completions of the theory $\TVFA_{K}$ of models of $\TVFA$ over $K$ correspond bijectively to completions of $\ACFA_{k}$, the theory of models of $\ACFA$ over $k$. In particular if $k$ has no nontrivial finite $\sigma$-invariant Galois extensions then the theory of models of $\TVFA$ over $K$ is complete.
\end{lem}
\begin{proof}
Using \cite[Proposition 4.7]{dor-hrushovskiVFA} we may assume that $K$ is Henselian; this does not change the value group or the residue field. By Henselianity, the valuation lifts uniquely to finite extensions of $K$. Using divisibility of $\Gamma$ and elementary ramification theory, the category of algebraic extensions of $K$ is in canonical bijection with the category of algebraic extensions of $k$ via the residue map; by functoriality, the category of transformal valued fields algebraic over $K$ is equivalent to the category of difference fields algebraic over $k$. So we finish using Fact \ref{fact:wvfa}(8).
\end{proof}

\begin{lem}
    Notation as in Lemma \ref{lem:spherically-completeq-qf} (and recall that the characteristic is zero). Then we have \[\tp_{\TVFA} \left(\nicefrac{b}{A}\right) \cup \tp_{\TwOGA} \left(\nicefrac{\gamma}{\Gamma_C}\right) \cup \tp_{\ACFA} \left(\nicefrac{\alpha}{k_C}\right) \vdash \tp_{\TVFA} \left(\nicefrac{b}{C}\right).\]
\end{lem}

\begin{proof}
Increasing $B$ and $C$ we can assume both are algebraically closed and the conclusion descends (the linear disjointness assumption on the residue fields is not violated by Fact \ref{fact:alg-disjointness-equiv-lin-disjointness-for-reg}, for example). Let $N = B \otimes_{A} C$ be the displayed model of $\VFA$. By Lemma \ref{lem:spherically-completeq-qf} the transformal valued field structure on $N$ is determined, hence using Fact \ref{fact:wvfa} (8), the type $\tp_{\TVFA}\left(\nicefrac{b}{C}\right)$ is determined by the additional data of the choice of an isomorphism type of an extension of the transformal valued field structure of $N$ to an algebraic closure. So we finish using Lemma \ref{lem:twvfa-0-complete-type}, together with Fact \ref{fact:wvfa}(8).
\end{proof}

For the sake of the next corollary we keep the notation $A, B, C, \alpha, \gamma$ and $b$.  Elementary properties of forking independence, and the characterization of forking in $\ACFA$ and $\TwOGA$ will be used freely.

\begin{cor}\label{C: forking in wVFA}
    Let $A$ be a spherically complete model of $\TVFA_{
\mathbf{Q}}$ and $B,C\supseteq A$ two models of $\VFA$ over $A$.  Then  the type $\tp_{\TVFA}\left(\nicefrac{b}{C}\right)$ does not fork over $A$ if and only if the following conditions are satisfied:
\begin{enumerate}
    \item The type $\tp_{\TwOGA}\left(\nicefrac{\gamma}{\Gamma_C}\right)$ does not fork over $\Gamma_A$.
    \item The type $\tp_{\ACFA}\left(\nicefrac{\alpha}{k_C}\right)$ does not fork over $k_A$. 
\end{enumerate}
    
    Moreover, in these two (equivalent) situations we have
    \begin{center}
$\tp_{\TVFA} \left(\nicefrac{b}{A}\right) \cup \tp_{\TwOGA} \left(\nicefrac{\gamma}{\Gamma_C}\right) \cup \tp_{\ACFA} \left(\nicefrac{\alpha}{k_C}\right) \vdash \tp_{\TVFA} \left(\nicefrac{b}{C}\right)$.
\end{center}
\end{cor}
\begin{proof}
    Assume first that $\tp_{\TVFA} \left(\nicefrac{b}{C}\right)$ does not fork over $A$.  Then $\tp_{\ACFA}\left(\nicefrac{\alpha}{k_C}\right)$ does not fork over $k_A$ by elementary properties of forking independence, and likewise for $\Gamma$. This shows the easy direction.

    Conversely, assume that (1) and (2) hold.  We show that $\tp\left(\nicefrac{b}{C}\right)$ does not fork over $A$. Towards this end there is no harm in increasing $C$. We may therefore assume that  $C$ is sufficiently saturated. In this situation, forking equals dividing for types over $C$, so it will be enough to show that this type does not divide over $A$.

    Using the characterization of forking in $\ACFA$ and full embeddedness of the residue field in $\TVFA$, we find that $k_B$ and $k_C$ are linearly disjoint over $k_A$. Similarly, using full embeddedness of $\Gamma$, we find that $\Gamma_B \cap \Gamma_C = \Gamma_A$. Using Lemma \ref{lem:spherically-completeq-qf} we have that \[\tp_{\TVFA} \left(\nicefrac{b}{A}\right) \cup \tp_{\TwOGA} \left(\nicefrac{\gamma}{\Gamma_C}\right) \cup \tp_{\ACFA} \left(\nicefrac{\alpha}{k_C}\right) \vdash \tp_{\TVFA} \left(\nicefrac{b}{C}\right).\]

    The goal is to prove that the right hand side does not divide over $A$. Since the left hand side implies the right, it is sufficient to show that any conjunction of formulas from the left hand side does not divide over $A$. 
    
    By stable embeddedness and orthogonality of the residue field and value group in $\TVFA$, $\tp_{\widetilde \VFA}(\nicefrac{\gamma,\alpha}{C})$ is equivalent $\tp_{\TwOGA} \left(\nicefrac{\gamma}{\Gamma_C}\right) \cup \tp_{\ACFA} \left(\nicefrac{\alpha}{k_C}\right)$ and so does not divide over $A$.

    Consider a conjunction $\phi(x)\wedge \psi(y,z,c)$, where $\phi(x)$ is an $A$-formula from $\tp_{\TVFA} \left(\nicefrac{b}{A}\right)$ (so $x$ is a tuple of field sort variables) and $\psi(y,z,c)$ is a formula from $\tp_{\widetilde \VFA}(\nicefrac{\gamma,\alpha}{C})$. The tuples $x,y$ and $z$ correspond to finite subtuples $b_0, \gamma_0$ and $\alpha_0$, respectively. Since $\alpha$ is an enumeration of the residue elements of $b$ and $\gamma$ is an enumeration of the value group of $b$, increasing the tuples $b_0,\gamma_0$ and $\alpha_0$ we may assume that \[b_0\models \psi'(x,y,z,c):=\phi(x)\wedge v(x)=y\wedge \mathrm{res}(x)=z\wedge \psi(y,z,c).\] 
    
    The formula $\exists x \psi'(x,y,z,c)$ lies in $\tp_{\widetilde \VFA}(\nicefrac{\gamma,\alpha}{C})$ so does not divide over $A$. Since $\phi$ is over $A$, the definition of dividing implies that  $\phi(x)\wedge \psi(y,z,c)$ does not divide over $A$ as well. 
%
 %
%
\end{proof}

The following was obtained by Chernikov and Hils in \cite{chernikov2014valued} in a relative setting.

\begin{cor}\label{C:NIP-dividing}
We work in the theory $\TVFA_{\mathbf{Q}}$. Let $A$ be a small model and $p$ a global type which forks over $A$; then there is an NIP formula $\varphi$ dividing over $A$ such that $p \vdash \varphi$.
\end{cor}
\begin{proof}
Fix a small but sufficiently saturated model $C$ over $A$, so a type over $C$ forks over $A$ if and only if it divides over $A$. Let us assume that $p = \tp(b/C)$ divides over $A$; we will prove that $p \vdash \varphi$ where $\varphi$ is an NIP formula which divides over $A$. Without loss of generality, replacing $b$ by $\mathrm{dcl}\left(Ab\right)$, we can take $b$ to enumerate a small model $B$ of $\VFA$ over $A$.

First assume that $A$ is spherically complete. In this situation, by Corollary \ref{C: forking in wVFA}, either $\tp_{\TwOGA}\left(\nicefrac{\gamma}{\Gamma_C}\right)$ divides over  $\Gamma_{A}$ or $\tp_{\ACFA}\left(\nicefrac{\alpha}{k_C}\right)$  divides over $k_{A}$. In any of these situations, the formula witnessing this is an NIP formula inside $\tp(b/C)$ dividing over $A$.

The reduction to the spherically complete case is formal. Let $A \subseteq A_0 \subseteq C$ be a spherically complete model of $\TVFA$. Replacing $A_0$ by an isomorphic copy over $A$, we can assume that $p$ continues to divide over $A_0$. By the spherically complete case, there is an NIP formula $p \vdash \varphi$ dividing over $A_0$; this formula divides over $A$ as well, and we are done.
\end{proof}

\begin{cor}
    The theory $\TVFA_{\mathbf{Q}}$ has NTP$_2$. 
\end{cor}
\begin{proof}
    By Corollary \ref{C:NIP-dividing}, $\TVFA_{\mathbf{Q}}$ has NTP$_2$-forking in the sense of  Proposition \ref{prop:equivalence of NTP2 forking}. Thus by Theorem \ref{thm:NTP2 forking iff NTP2}, $\TVFA_{\mathbf{Q}}$ has NTP$_2$.
\end{proof}
Using the characterization of forking we can pin down the stationary types in $\TVFA_{\mathbf{Q}}$.

\begin{defn}\label{def:stationary}
    Let $T$ be any theory and let $p$ be a global type. We say that $p$ is \emph{stationary} if there is a small set $C_0$ with the following property: for every small set $C$ containing $C_0$, the type $p$ is the unique global nonforking extension of $p|_{C}$.
\end{defn}

\begin{rem}
    Let $p$ be a global type which is stationary and let $C_0$ be a small set witnessing this. Then $p$ does not fork over $C_0$. In fact, since the forking ideal is invariant under automorphisms of the universal domain, the type $p$ does not split over $C_0$, i.e, it is invariant under global automorphisms fixing $C_0$ pointwise.
\end{rem}

The following is a Corollary of \cite{ChHr-Dif} and \cite{SCFEe} worth making explicit. It also follows from \cite[Lemma 3.6]{medina2019definablegroupsdcfa}, we sketch here another proof.

\begin{fact}
    \label{fact:acfa-char-0-stationary} In the theory $\ACFA_{\mathbf{Q}}$, a global invariant type is stationary if and only if it is orthogonal to the fixed field.
\end{fact}
\begin{proof}
   There are no non-algebraic stationary types in the theory of pseudofinite fields, so the condition is necessary. For the converse, assume $C$ is an algebraically closed extension of an algebraically closed field $A$ orthogonal to the fixed field; we will show that $C$ is superficially stable over $A$ in the sense of \cite{ChHr-Dif}; such types are stationary, using the characterization of forking in $\ACFA$.
   First assume that $C$ is purely transformally transcendental over $A$. In this situation, it is stationary over $A$ using \cite[Theorem 5.3]{SCFEe}.
   Next assume that $C$ is transformally algebraic over $A$. In this situation, it is superficially stable over $A$ using \cite[Theorem 4.11, Theorem 4.2.5]{ChHr-Dif}.
   We have shown the claim in the case $C$ transformally algebraic over $A$ and purely transformally transcendental over $A$. But if $A \subseteq B \subseteq C$ is a tower with $\nicefrac{B}{A}$ and $\nicefrac{C}{B}$ stationary then it follows formally that $\nicefrac{C}{A}$ is likewise stationary. Applying this with $B$ the relative transformal algebraic closure of $A$ in $C$ we obtain the general case.
\end{proof}

\begin{cor}\label{C:stationary type}
    In the theory $\TVFA_{\mathbf{Q}}$, a global invariant type is stationary if and only if it is orthogonal to the fixed field and the value group.
\end{cor}

\begin{proof}
The conditions are necessary: the value group is linearly ordered, so the only generically stable types are algebraic. Orthogonality to the fixed field is likewise necessary using (the easy direction of) Fact \ref{fact:acfa-char-0-stationary}.

The converse follows immediately from Fact \ref{fact:acfa-char-0-stationary} and Corollary \ref{C: forking in wVFA}.
\end{proof}

\begin{rem}
    In finite characteristic a classification of this form cannot hold; see  \cite[Example 6.5]{ChHr-Dif}.
\end{rem}

Let $K$ be a valued field. We say that $K$ is \emph{separably maximal} if the perfect hull of $K$ is spherically complete; this is equivalent to the requirement that $K$ admits no immediate separable extensions. In \cite{hils2018imaginaries} valued fields of this form were used to establish domination over the residue field and the value group in the theory $\SCVF_{e}$ of separably algebraically closed valued fields of \emph{finite} degree of imperfection. It is not difficult to see that a domination result of this form cannot hold in infinite Ershov invariant. Similarly, in $\TwVFE$ we cannot hope for a domination theorem relative to the value group and the residue field. In order to transfer the domination statement we will make use of the following definition:

\begin{defn}
    Let $K$ be a model of $\VFE$.
    \begin{enumerate}
        \item We say that $K$ is \emph{transformally separably algebraically maximal} if whenever $\widetilde{K}$ is a model of $\VFE$ transformally separably algebraic over $K$ with $k = \widetilde{k}$ and $\widetilde{\Gamma} = \Gamma$, then $\widetilde{K} = K$; in other words the field $K$ has no nontrivial immediate transformally separably algebraic extensions.
        \item We say that $K$ is \emph{transformally separably maximal} if whenever $\widetilde{K}$ is a model of $\VFE$ transformally separable over $K$ with $k = \widetilde{k}$ and $\widetilde{\Gamma} = \Gamma$, then $\widetilde{K} = K$; in other words, the field $K$ has no nontrivial immediate transformally separable extensions.    \end{enumerate}
\end{defn}

\begin{rem}
    Let $K$ be a model of $\VFE$. It is clear that if $K$ is transformally separably maximal, then $K$ is transformally separably algebraically maximal.
\end{rem}

\begin{rem}
    Let $k$ be an inversive algebraically closed difference field which is trivially valued and $x$ transformally transcendental over $k$ with $vx > 0$; by \cite[Proposition 6.23]{hrushovski2004elementary}  the completed algebraic closure of $k\left(x^{\Ns}\right)$ is transformally algebraically maximal.
\end{rem}

\begin{prop}\label{prop:twvfe-tsep-alg-max}
    Let $K$ be a model of $\TwVFE$. Then $K$ is transformally separably algebraically maximal.
\end{prop}

\begin{proof}
    Let $L$ be a model of $\VFE$ transformally separably algebraic and immediate over $K$. Let $\widetilde{K}$ and $\widetilde{L}$ denote the inversive hulls of $K$ and $L$ respectively. Since the value group and the residue field of $K$ are inversive, the same is true of $L$; hence $\widetilde{L}$ is an immediate extension of $\widetilde{K}$. But $\widetilde{K}$ is a model of $\TVFA$, hence transformally algebraically maximal by Fact \ref{fact:wvfa}. It follows that $\widetilde{K} = 
    \widetilde{L}$, that is, the inversive hulls of $K$ and $L$ coincide. So $L$ is simultaneously purely transformally inseparably algebraic and transformally separably algebraic over $K$; it follows that $K = L$, which is what we wanted.
\end{proof}

\begin{rem}
    Let $K$ be a model of $\VFA$ which is algebraically closed; then $K$ is maximal in the category of models of $\VFA$ if and only if it is maximal in the category of valued fields; see \cite[Proposition 4.18] {dor-hrushovskiVFA}.
\end{rem}

\begin{lem}\label{lem:tsep-max-iff-inversive-hull-max}
    Let $K$ be a model of $\TwVFE$. Then $K$ is transformally separably maximal if and only if the inversive hull of $K$ is spherically complete.
\end{lem}
\begin{proof}
    If the inversive hull of $K$ is spherically complete, then the proof of Proposition \ref{prop:twvfe-tsep-alg-max} shows that $K$ is transformally separably maximal. Conversely, if the inversive hull $M$ of $K$ is not spherically complete, let $L$ be an extension witnessing this. Since $M$ is transformally algebraically maximal we can find some $x$ in $L$ transformally transcendental over $M$; then $K\left(x^{\Nsphul}\right)$ is an immediate proper transformally separable extension of $K$, whence $K$ is not transformally separably maximal.
\end{proof}

\begin{prop}
    Let $K$ be a model of $\TwVFE_e$. Then there is a model $M$ of $\TwVFE_e$, immediate and transformally separable over $K$, which is transformally separably maximal.
\end{prop}

\begin{proof}
    Let $E$ be a spherically complete immediate extension of $K$ (Fact \ref{fact:wvfa}(7)); then $E$ is a model of $\TVFA$ and in particular inversive.  We will find the desired $M$ inside $E$; in fact the field $E$ will be the inversive hull of $M$.
    
    Let $\left(x_i\right)$ be a transformal transcendence basis of $E$ over $K$. Let $M$ be the relative transformal separable algebraic closure of $K\left(x_i^{\Nsp}\right)$ inside $E$. We claim that $M$ is as desired. For this, we must check that $M$ is transformally separably maximal, that $M$ is transformally separable over $K$, and that $M$ is a model of $\TwVFEe$.
    
    By construction the perfect hull of $M$ splits over $K$, so it is a model of $\wVFEe$ immediate over $K$; it remains to be shown that $M$ is transformally separably maximal, and that it is a model of $\TwVFEe$.
    We claim that $E$ is the inversive hull of $M$. Indeed, since the $x_i$ form a transformal transcendence basis of $E$ over $K$, the field $E$ is transformally algebraic over $M$. Furthermore, by our choice of $M$, the field $M$ is transformally separably algebraically closed in $E$. So $E$ is inversive and purely transformally inseparably algebraic over $M$; hence it is the inversive hull of $M$.
    So the inversive hull of $M$ is a spherically complete model of $\TVFA$; it follows from Lemma  \ref{lem:tsep-max-iff-inversive-hull-max} and Proposition \ref{P:pure insep inside twvfe} that $M$ is as desired.
\end{proof}

Let $\ACFtwo = {\ACFtwo}_{\mathbf{Q}}$ be the theory of pairs $\left(B, A\right)$ of algebraically closed fields with $B$ a proper extension of $A$. This theory is model complete and stable if we expand the language by relation symbols expressing linear independence of tuples of elements of $B$ viewed as a vector space over $A$ \cite{delon}. The theory $\ACFtwo$ is a reduct of $\TwVFE_{\mathbf{Q}}$ by setting $B = K$ and $A = K^{\sigma}$. The transformal $\lambda$-functions are then definable in the reduct to $\ACFtwo$. 

\begin{lem}\label{lem:scfe-stable-forking}
    We work in $\SCFE_{\mathbf{Q}}$. Let $A$ be an algebraically closed difference field and $B, C$ be algebraically closed extensions of $A$ inside a saturated model $\mathbf{K}$. Let us assume that $\mathbf{K}$ is transformally separable over $B$ and $C$. Then $B$ and $C$ are independent over $A$ (in the sense of simplicity theory) if and only if they are linearly disjoint over $A$ and $\mathbf{K}$ is transformally separable over $B \otimes_{A} C$. This independence relation is moreover given by stable formulas, namely by independence in the sense of $\ACFtwo$.
\end{lem}

\begin{proof}
    See \cite[Definition 4.9 and Proposition 4.10]{SCFEe}; the fields $B$ and $C$ are independent over $A$ in the sense of $\SCFE$ if and only if they are linearly disjoint over $A$ and a transcendence basis of $B$ over $A \otimes_{A^{\sigma}} B^{\sigma}$ lifts to a transcendence basis of $B \otimes_{A} C$ over $B^{\sigma} \otimes_{A^{\sigma}} A \otimes_{A^{\sigma}} C^{\sigma}$.
\end{proof}

\begin{lem}\label{lem:qe-in-twvfe-wvfa-plus-scfe}
    Let $A$ be a model of $\TwVFE_{\mathbf{Q}}$ and $\mathbf{K}$ a saturated model of $\TwVFE_{\mathbf{Q}}$ transformally separable over $A$. Let $B,C$ be models of $\VFE$ over $A$ inside $\mathbf{K}$, closed under the transformal $\lambda$-functions of $\mathbf{K}$, and assume that $B$ and $C$ are independent over $A$ in the sense of $\SCFE = \SCFE_\mathbf{Q}$. Let $b$ be a tuple enumerating the elements of $B$. Then $\tp_{\TwVFE}\left(\nicefrac{b}{A}\right) \cup \tp_{\ACFtwo}\left(\nicefrac{b}{C}\right) \cup \tp_{\TVFA}\left(\nicefrac{b}{C}\right) \vdash \tp_{\TwVFE}\left(\nicefrac{b}{C}\right)$.
\end{lem}

\begin{proof}
This is a formal consequence of Lemma \ref{lem:scfe-stable-forking} and quantifier elimination in $\TwVFE$ (Corollary \ref{cor:qe}). Namely by the independence assumption (encoded in the type $\tp_{\ACFtwo}\left(\nicefrac{b}{C}\right)$), the fields $B$ and $C$ are linearly disjoint over $A$, and the field $N = B \otimes_{A} C$ is closed under the transformal $\lambda$-functions of $\mathbf{K}$; given this, the type $\tp_{\TwVFE}\left(\nicefrac{b}{C}\right)$ is determined by its quantifier free part and the choice of an extension of the transformal valued field structure to the algebraic closure. This data is encoded in $\tp_{\TVFA}\left(\nicefrac{b}{C}\right)$.
\end{proof}

Let $A$ be a model of $\TwVFE_{\mathbf{Q}}$ which is transformally separably maximal. Let $B,C$ be models of $\VFE$ over $A$ inside an ambient saturated transformally separable extension $\mathbf{K} \models \TwVFE_{\mathbf{Q}}$ of $A$, and assume $B$ and $C$ are closed under the transformal $\lambda$-functions of $\mathbf{K}$. Let $b, \alpha, \gamma$ be  tuples enumerating the elements of $B$, the residue field elements of $B$, and the value group elements of $B$, respectively.

\begin{prop}\label{prop:sph-dom-vfe}
    Notation as above (and recall the characteristic is zero). Then the type $\tp_{\TwVFE}\left(\nicefrac{b}{C}\right)$ does not fork over $A$ if and only if the following conditions are satisfied: 
    \begin{enumerate}
        \item The type $\tp_{\ACFtwo}\left(\nicefrac{b}{C}\right)$ does not fork over $A$, that is $B$ and $C$ are independent over $A$ in the sense of $\SCFE$.
        \item The type $\tp_{\ACFA}\left(\nicefrac{\alpha}{k_C}\right)$ does not fork over $k_A$, that is the fields $k_B$ and $k_C$ are linearly disjoint over $k_A$.
        \item The type $\tp_{\TwOGA}\left(\nicefrac{\gamma}{\Gamma_C}\right)$ does not fork over $\Gamma_A$, that it is it admits a global extension which does not split over $\Gamma_A$.
    \end{enumerate}
    Moreover in these two (equivalent) cases we have:
    \begin{center}
    $\tp_{\TwVFE}\left(\nicefrac{b}{A}\right) \cup \tp_{\ACFtwo}\left(\nicefrac{b}{C}\right) \cup \tp_{\TwOGA}\left(\nicefrac{\gamma}{\Gamma_C}\right) \cup \tp_{\ACFA}\left(\nicefrac{\alpha}{k_C}\right) \vdash \tp_{\TwVFE}\left(\nicefrac{b}{C}\right)$.
    \end{center}
\end{prop}

\begin{proof}
The three conditions are necessary in order to guarantee nonforking. Namely (1) is expressed using stable formulas, and independence is the generic condition. Moreover (2) and (3) follow from stable embeddedness of the value group and the residue field. To show the conditions are sufficient it will be enough to establish the "moreover" part (as in the proof of Corollary \ref{C: forking in wVFA}).
Using Lemma \ref{lem:tsep-max-iff-inversive-hull-max}, the inversive hull of $A$ is a spherically complete model of $\TVFA$. So working over the inversive hull of $A$ and using Corollary \ref{C: forking in wVFA} the left hand side axiomatizes $\tp_{\TVFA}\left(\nicefrac{b}{C}\right)$; using Lemma \ref{lem:qe-in-twvfe-wvfa-plus-scfe}, this data is sufficient to axiomatize the right hand side.
\end{proof}

\begin{cor}\label{cor:twvfe-ntp2}
In the theory $\TwVFE_{\mathbf{Q}}$, forking over a model is always witnessed by an $\NIP$ formula; hence $\TwVFE_{\mathbf{Q}}$ is $\NTPtwo$.
\end{cor}

\begin{proof}
The proof for the case of $\TVFA_{\mathbf{Q}}$ goes through, using Proposition \ref{prop:sph-dom-vfe} to replace maximal models with transformally separably maximal models and noting that the theory $\ACFtwo$ is stable (so forking is always witnessed by a stable formula).
\end{proof}

\begin{rem}
    Simone Ramello proved a vast generalization of Corollary \ref{cor:twvfe-ntp2} using the proof strategy from \cite{chernikov2014valued}, see \cite[Section 5.5]{simone}.
\end{rem}

\begin{cor}
    In the theory $\TwVFE_{\mathbf{Q}}$, a global invariant type is stationary if and only if it is orthogonal to the fixed field and the value group.
\end{cor}

\begin{proof}
    This follows immediately from Proposition \ref{prop:sph-dom-vfe}, Fact \ref{fact:acfa-char-0-stationary}, and the uniqueness of nonforking extensions in $\ACFtwo$.
\end{proof}

\appendix

\section{NTP$_2$-Forking, by Itay Kaplan}

Let $T$ be a first order theory and let $\mathcal{U}\models T$ be a large saturated model. Chernikov proved in \cite[Proposition 4.14(1)]{chernikov2014theories}, that assuming that whenever $p(x) \in S(N)$ divides over $M \prec N$ then this is witnessed by an NIP (dependent) formula\footnote{In fact the proof only uses that the witness is NTP$_2$.}, then $T$ is NTP$_2$. We generalize this slightly by relaxing the assumption on $T$ from all types over models to global types (or types over saturated enough models, see Proposition \ref{prop:equivalence of NTP2 forking} below). 

\begin{defn}
    A formula $\phi(x,y)$ (perhaps with parameters) is \emph{TP$_2$} if there is an array $\sequence{a_{i,j}}{i,j<\omega}$ and $k<\omega$ such that:
    \begin{itemize}
        \item every row is $k$-inconsistent, i.e., for every $i<\omega$, $\set{\phi(x,a_{i,j})}{j<\omega}$ is $k$-inconsistent and
        \item every vertical path is consistent, i.e., for every $\eta:\omega\to \omega$, $\set{\phi(x,a_{i,\eta(i)})}{i<\omega}$ is consistent.
    \end{itemize}
    If $\phi$ does not have TP$_2$, we say that $\phi$ is \emph{NTP$_2$}.
\end{defn}

\begin{defn}\label{def:NTP2 forking}
    Say that a (complete first-order) theory $T$ has \emph{NTP$_2$-forking over a (small) model $M$} if whenever $\phi(x,a)$ forks over $M$ then there are NTP$_2$-formulas $\phi_i(x,z_i)$ over $M$ and $b_i$ for $i<n$ such that $\phi(x,a) \vdash \bigvee_{i<n} \phi_i(x,b_i)$ and $\phi_i(x,b_i)$ divides over $M$. In other words, forking over $M$ is witnessed by dividing NTP$_2$-formulas. 

    Say that $T$ has \emph{NTP$_2$-forking} if it has NTP$_2$ forking over all models $M$. 
\end{defn}

\begin{prop}\label{prop:equivalence of NTP2 forking}
    The following are equivalent for a theory $T$ and $M\vDash T$.
    \begin{enumerate}
        \item $T$ has NTP$_2$-forking over $M$.
        \item If $p(x)$ is a global type which divides/forks over $M$ then there is some NTP$_2$-formula $\phi(x,y)$ over $M$ and $b$ such that $\phi(x,b) \in p$ and $\phi(x,b)$ divides over $M$. 
        \item If $N \succ M$ is $|M|^+$-saturated and $p(x) \in S(N)$ divides/forks over $M$ then there is some NTP$_2$-formula  $\phi(x,y)$ over $M$ and $b$ such that $\phi(x,b) \in p$ and $\phi(x,b)$ divides over $M$. 
    \end{enumerate}
\end{prop}
\begin{proof}
    (1) implies (3) follows from saturation, (3) implies (2) is clear.  
    In order to prove that (2) $\implies$ (1), we assume that (1) fails and show that (2) must fail also. So suppose that $\phi(x,a)$ forks over $M$ but does not imply a disjunction of dividing NTP$_2$-formulas over $M$ as in the definition; then $\set{\neg \psi(x,b)}{\psi(x,y)\in L(M) \text{ is NTP}_2, b\in \mathcal{U} \text{ and } \psi(x,b) \text{ divides over }M }$ is consistent with $\phi(x,a)$, thus any completion gives a contradiction to (2). 
\end{proof}

We need the notions of strict invariance and strict Morley sequences.
\begin{defn}
    We say that a global type $q(x)\in S(\mathcal{U})$ is \emph{strictly invariant over a model $M$}, if $q$ is $M$-invariant and if $c \vDash q|_C$ where $C \supseteq M$, then $C \ind_M c$, i.e., $\tp(C/Mc)$ does not fork over $M$. 
    
    We write $a \ind^{ist}_M C$ for: there exists a global strictly invariant over $M$ extension of $\tp(a/MC)$. 
\end{defn}

\begin{thm} \label{thm:forking = dividing over M}
    Suppose that $T$ has NTP$_2$-forking over $M$. Then forking equals dividing over $M$: if $\phi(x,b)$ forks over $M$, then $\phi(x,b)$ divides over $M$. 

    It follows that if $T$ has NTP$_2$-forking then forking equals dividing over models. 
\end{thm}

\begin{proof}
    The proof is essentially the same as the proof of \cite[Theorem 1.1]{CheKa}. See also \cite[Theorem 5.49]{simon-NIP} and \cite{AdlerHandWritten}. 
    The proof follows the following stages:
    \begin{enumerate}
        \item\label{itm:dividing is witnessed by a MS} (Generic witness) Suppose that $\phi(x,y) \in L(M)$ is NTP$_2$ and that $\phi(x,a)$ divides over a model $M$. Then there is a global coheir $q(y)$ extending $\tp(a/M)$ such that some/any Morley sequence in $q$ over $M$ witnesses the dividing of $\phi(x,a)$. In other words, if $\sequence{a_i}{i<\omega} \vDash q^{(\omega)}|_M$, then $\set{\phi(x,a_i)}{i<\omega}$ is inconsistent. This follows from the same proof as in \cite[Lemma 3.12]{CheKa} (see also \cite[Lemma 5.35]{simon-NIP}): note that although there NTP$_2$ is assumed globally, it is only used for the specific formula in question. 
        \item \label{itm:forking implies quasi-dividing} (Forking implies quasi-dividing) If $\phi(x,a)$ forks over $M$, then there are finitely many conjugates of $a$ over $M$, $a=a_0,\ldots,a_{m-1}$ such that $\set{\phi(x,a_i)}{i<n}$ is inconsistent. Indeed, by assumption there are NTP$_2$-formulas $\phi_i(x,z_i) \in L(M)$ and $b_i\in \mathcal{U}$ for $i<n$ such that $\phi_i(x,b_i)$ divides over $M$ and $\phi(x,a)\vdash \bigvee_{i<n} \phi_i(x,b_i)$. By (\ref{itm:dividing is witnessed by a MS}), there are global coheirs $q_i(z_i)\in S(\mathcal{U})$ such that any Morley sequence in $q_i$ over $M$ witnesses that $\phi(x,b_i)$ divides. This allows us to apply the ``broom lemma'', \cite[Lemma 3.1]{CheKa} (see also \cite[Lemma 5.36]{simon-NIP}) which in \cite{CheKa} does not require NTP$_2$ and allows us to conclude exactly as in the proof of \cite[Corollary 3.13]{CheKa}. 
        \item \label{itm:existence of strictly inv} (Existence of strictly invariant extensions) if $p(x)\in S(M)$ then there is a strictly invariant over $M$ type (in fact even a coheir, but we will not use this) $q(x)\in S(\mathcal{U})$ extending $p$. This follows from \cite[Proposition 3.7(1)]{CheKa} which doesn't require NTP$_2$ but does assume (\ref{itm:forking implies quasi-dividing}).
        \item \label{itm:Kim's lemma} (Kim's lemma for NTP$_2$-formulas) If $\phi(x,y) \in L$ is NTP$_2$ and $\phi(x,a)$ divides over a model $M$, and $q(y) \supseteq \tp(a/M)$ is a global strictly invariant type over $M$, then any Morley sequence in $q$ over $M$ witnesses dividing. In other words, if $\sequence{a_i}{i<\omega}$ is a Morley sequence of $q$ over $M$ then $\set{\phi(x,a_i)}{i<\omega}$ is inconsistent. This follows from (the proof of) \cite[Lemma 3.14]{CheKa}. 
        \item (Conclusion) If $\phi(x,a)$ forks over $M$, then there are NTP$_2$-formulas $\phi_i(x,z_i) \in L(M)$ and $b_i\in \mathcal{U}$ for $i<n$ such that $\phi_i(x,b_i)$ divides over $M$ and $\phi(x,a)\vdash \bigvee_{i<n} \phi_i(x,b_i)$. By concatenating and adding dummy variables (these operations clearly preserves being NTP$_2$) we may assume that $b_i = a$ and $z_i = z$ for all $i<n$. By (\ref{itm:existence of strictly inv}) there exists a global $M$-strictly invariant type $q(z) \supseteq \tp(a/M)$. Let $\sequence{a_j}{j<\omega}$ be a Morley sequence in $q$. By (\ref{itm:Kim's lemma}), $\set{\phi_i(x,a_j)}{j<\omega}$ is inconsistent for all $i<n$. It follows by pigeonhole that $\sequence{a_j}{j<\omega}$ witnesses that $\phi(x,a)$ divides over $M$ (see also the proof of \cite[Corollary 3.16]{CheKa}). 
        \qedhere
    \end{enumerate}
\end{proof}

Now we can conclude:

\begin{thm}\label{thm:NTP2 forking iff NTP2}
    Let $T$ be any theory. Then $T$ has NTP$_2$-forking if and only if $T$ is NTP$_2$.
\end{thm}

\begin{proof}
    Right to left is obvious so we prove left to right.

    Suppose not: $T$ has NTP$_2$ forking and TP$_2$. By \cite[Theorem 4.9]{chernikov2014theories}, there is a model $M$ and sequence $\sequence{b_i}{i<\kappa}$ for $\kappa = ({|M|+|T|})^+$, such that $b_i \ind^{ist}_M b_{<i}$ and some $a$ such that $\tp(a/Mb_i)$ divides for all $i<\kappa$. For each $i<\kappa$, let $\phi_i(x,b_i) \in \tp(a/Mb_i)$ divide over $M$. By assumption there are NTP$_2$-formulas $\sequence{\psi_{i,j}(x,y_{i,j})}{j<n_i}$ over $M$ and $\sequence{c_{i,j}}{j<n_i}$ such that $\phi_i(x,b_i)\vdash \bigvee_{j<n_i} \psi_{i,j}(x,c_j)$. Let $y_i$ be the concatenation of all the variables $y_{i,j}$, and let $c_i$ be the appropriate concatenation of the $c_{i,j}$'s (we cannot incorporate $b_i$ as well since this may cause the loss of the strict independence over $M$). Note that $\psi_i(x,y_i):=\bigvee_{j<n_i}\psi_{i,j}(x,y_{i,j})$ is NTP$_2$ by \cite[Lemma 7.1]{chernikov2014theories}. 
    
    By Theorem \ref{thm:forking = dividing over M}, $\psi_i(x,c_i)$ divides over $M$ for all $i<\kappa$. Thus, for each $i<\kappa$, there is an $M$-indiscernible sequence $\sequence{d_{i,j}}{j<\omega}$ starting with $d_{i,0} = c_i$ and some number $k_i <\omega$ such that $\set{\psi_i(x,d_{i,j})}{j<\omega}$ is $k_i$-inconsistent. Trimming, we may assume that $k_i =: k$, $\phi_i =: \phi$ and $\psi_i(x,y_i) =: \psi(x,y)$ for all $i<\omega$. For each $d_{i,j}$ find $b_{i,j}$ such that $b_{i,j}d_{i,j}\equiv_M b_i c_i$ and $b_{i,0} = b_i$. In particular $\phi_i(x,b_{i,j}) \vdash \psi_i(x,d_{i,j})$. We may assume, by Ramsey and applying an automorphism, that $\sequence{b_{i,j}}{j<\omega}$ is $M$-indiscernible. By \cite[Lemma 4.4]{chernikov2014theories}, there are $M$-mutually indiscernible sequences $\sequence{\sequence{e_{i,j}}{j<\omega}}{i<\omega}$ such that $e_{i,0} = b_{i,0} = b_i$ and $\sequence{e_{i,j}}{j<\omega} \equiv_M \sequence{b_{i,j}}{j<\omega}$ for all $i<\omega$. 
    
    For each $i<\omega$, find $\sequence{f_{i,j}}{j<\omega}$ such that $\sequence{e_{i,j}f_{i,j}}{j<\omega} \equiv_M \sequence{b_{i,j} d_{i,j}}{j<\omega}$. We claim that the array $\sequence{f_{i,j}}{i,j<\omega}$ and $k$ witnesses that $\psi(x,y)$ has TP$_2$. 
    \begin{enumerate}
        \item Every row is $k$-inconsistent, i.e., $\set{\psi(x,f_{i,j})}{j<\omega}$ is $k$-inconsistent, because for each $i<\omega$, $\sequence{f_{i,j}}{j<\omega}\equiv_M \sequence{d_{i,j}}{j<\omega}$ and choice of the latter sequence. 
        \item Every vertical path is consistent: as $e_{i,0} = b_i$ and $\phi(a,b_i)$ holds for all $i<\omega$, it follows from mutual indiscernibility of the array $\sequence{e_{i,j}}{i,j<\omega}$ that for every $\eta:\omega\to\omega$, $\set{\phi(x,e_{i,\eta(i)})}{i<\omega}$ is consistent, and let $a_\eta$ be a realization. As  $\phi(x,b_{i,j}) \vdash \psi(x,d_{i,j})$, $\phi(x,e_{i,j}) \vdash \psi(x,f_{i,j})$ and hence $\psi(a_\eta,f_{i,\eta(i)})$ holds and in particular the set $\set{\psi(x,f_{i,\eta(i)})}{i<\omega}$ is consistent as required. 
    \end{enumerate}
    Together, we get a contradiction.
\end{proof}

\begin{rem}
    \begin{enumerate}
        \item From the proof of (\ref{itm:forking implies quasi-dividing}) in the proof of Theorem \ref{thm:forking = dividing over M}, one can deduce the following statement, which may be interesting on its own. Let $T$ be any theory, and let $\phi(x,y)$ be an NIP formula. Suppose that $M$ is a model and that $\pi(x)$ is a partial type over $\mathcal{U}$ which is $M$-invariant. Then there is some $\pi^* \supseteq \pi$ which is $M$-invariant and such that $\pi^*\restriction \phi$ is a complete global $\phi$-type. Indeed, it is enough to show that $\pi(x) \cup \set{\phi(x,c)\leftrightarrow \phi(x,b)}{b\equiv_M c}$ is consistent. If not then $\pi \vDash \bigvee_{i<n} \neg(\phi(x,c_i)\leftrightarrow \phi(x,b_i))$. For each $i<n$, let $q_i$ be a global coheir extending $\tp(b_i/M) = \tp(c_i/M)$, and let $a_i \vDash q_i|_{Mb_ic_i}$. Then we get that $\pi \vDash \bigvee_{i<n} \neg(\phi(x,a_i)\leftrightarrow \phi(x,b_i)) \vee \neg(\phi(x,a_i)\leftrightarrow \phi(x,c_i))$. Note that as $\phi$ is NIP, $\neg(\phi(x,a_i)\leftrightarrow \phi(x,b_i))$ divides over $M$ as witnessed by any Morley sequence in $q_i^{(2)}$ over $M$ (and the same is true replacing $b_i$ by $c_i$), and that $\psi(x,yz):=\neg(\phi(x,y)\leftrightarrow \phi(x,z))$ is NIP thus NTP$_2$. So we have that $\pi(x)$ implies a disjunction of dividing over $M$ NTP$_2$ formulas. By the ``broom lemma'', \cite[Lemma 3.1]{CheKa}, we get that $\pi$ is inconsistent, contradiction. 
        \item Call a formula $\phi(x,y)$ \emph{simple} if it does not have the tree property (see e.g., \cite[Definition 2.3.5]{Kim-simple}). Say that $T$ has \emph{simple forking} if whenever a formula forks over a model $M$, this is witnessed by simple formulas over $M$, as in Definition \ref{def:NTP2 forking}. Every simple formula is NTP$_2$ so if $T$ has simple forking, it has NTP$_2$ forking and thus forking equals dividing over models. A disjunction of simple formulas is still simple (see \cite[Lemmas 2.3.2 and 2.3.4(3)]{Kim-simple}), so similar ideas as in the proof of Theorem \ref{thm:forking = dividing over M} imply that in this case $T$ is simple (if not then by lack of local character there is some $a$ and an increasing continuous sequence $\sequence{M_i}{i<|T|^+}$ of models of size $\leq |T|$ such that $\tp(a/M_{i+1})$ forks over $M_i$, as witnessed by $\psi_i(x,d_i)$ where $\psi$ is over $M_i$. By the above there is a simple formula $\phi_i(x,y_i)$ over $M_i$ and some $c_i$ such that $\psi_i(x,d_i) \vdash \phi_i(x,c_i)$ and $\phi_i(x,c_i)$ divides over $M_i$. By Fodor and pigeonhole we may assume that for all $i<\omega$, $\phi_i = \phi$ and $\psi_i = \psi$ which are over $M_0$ and that $\phi(x,c_i)$ $k$-divides over $M_i$ for some constant $k<\omega$.  Now, exactly as in the proof of Theorem \ref{thm:NTP2 forking iff NTP2}, we conclude by constructing a tree of instances of $\phi$, contradiction). This is analogous to Chernikov's \cite[Proposition 4.14(2)]{chernikov2014theories} in the same way that Theorem \ref{thm:NTP2 forking iff NTP2} is analogous to \cite[Proposition 4.14(1)]{chernikov2014theories}.
    \end{enumerate}
    \end{rem}

\bibliographystyle{plain}
\bibliography{wvfe}

\end{document}